\documentclass[10pt]{IEEEtran}
\usepackage{graphicx,psfrag, xcolor}
\usepackage{amsmath}
\usepackage{amsthm}
\usepackage{cite}
\usepackage{latexsym}
\usepackage{amssymb}
\usepackage{graphics}
\usepackage{dsfont}
\usepackage{mathrsfs}
\usepackage{subfigure}
\usepackage{bbm}
\usepackage{latexsym}

\usepackage{caption}

\usepackage{graphics} 
\usepackage{epsfig} 
\usepackage{amssymb}
\usepackage{upgreek}
\usepackage{forest} 
\usepackage{comment}
\usepackage{cancel}
\usepackage{url}
\usepackage{soul}
\usepackage{graphicx,psfrag, color}
\usepackage{latexsym}
\usepackage{subfigure}
\usepackage{bbm}
\usepackage{latexsym}
\usepackage{lettrine}
\usepackage{dsfont}
\usepackage{array}
\usepackage{amsthm}
\usepackage{multicol}

\usepackage[english]{babel}
\usepackage{graphicx}
\usepackage{pgf, tikz, pgflibraryshapes}
\usepackage{amssymb,bm}
\usepackage{colortbl}
\usepackage{multicol}
\usepackage{stackrel}
\usepackage{pgfplots}
\usepackage{ulem}

\newcommand{\argmin}[1]{\underset{#1}{\mathrm{argmin\,}}}

\newcommand{\ba}{\begin{array}}
	\newcommand{\ea}{\end{array}}
\usepackage{color}
\usepackage[mathscr]{euscript}

\usepackage{amsfonts}

\newcommand{\diff}{{\rm\,d}}
\newcommand{\equaref}[1]{(\ref{eq:#1})}

\newtheorem{corollary}{Corollary}

\newtheorem{theorem}{Theorem}

\newtheorem{proposition}{Proposition}

\newtheorem{remark}{Remark}

\DeclareMathOperator{\EX}{\mathbb{E}}

\theoremstyle{definition}

\newtheorem*{vpolicy1}{MVF Policy}
\newtheorem*{vpolicy2}{MSF Policy}

\newtheorem*{cpolicy1}{Rate Control}
\newtheorem*{cpolicy2}{HT Control}
\newtheorem{ass}{Assumption}

\renewcommand{\eqref}[1]{Eq.~(\ref{#1})}

\definecolor{darkcyan}{rgb}{0.0, 0.55, 0.55}
\definecolor{internationalorange}{rgb}{1.0, 0.31, 0.0}

\usepackage{color}
\usepackage{ulem}

\begin{document}
	\author{Franco Galante,  Chiara Ravazzi,  Michele Garetto, Emilio Leonardi\vspace{-0.5cm}\thanks{F. Galante is with the Politecnico di Torino, Department of Electronics and Telecommunication (DET), Corso Duca Degli Abruzzi, 10129, Italy. 
			E-mail: franco.galante@polito.it
			
			C. Ravazzi is with the Institute of Electronics, Computer and Telecommunication Engineering, National Research Council of Italy (CNR-IEIIT), c/o Politecnico di Torino, Corso Duca Degli Abruzzi, 10129, Italy. E-mail: chiara.ravazzi@ieiit.cnr.it
			
			M. Garetto is with the University of Turin,  Computer Science Department, Corso Svizzera 185, 10149 Torino, Italy. E-mail: michele.garetto@unito.it
			
			E. Leonardi is with the Politecnico di Torino, Department of Electronics and Telecommunication (DET) and research associate at National Research Council of Italy (CNR-IEIIT).
			E-mail: emilio.leonardi@polito.it.
	}}
	\title{Planning interventions in a controlled pandemic:\\ the COVID-19 case}
	\maketitle

\begin{abstract}
	Restrictions on social and economic activities, as well as vaccinations, have been a key intervention in containing the COVID-19 epidemic. 
	Our work focuses on better understanding the options available to
	policymakers under the conditions and uncertainties created by the onset of a new pandemic. 
	More precisely, we focus on two control strategies. The first aims to control the rate of new infections to prevent congestion of the health care system. The latter directly controls hospitalizations and intensive care units (ICUs) occupation. By a first-order analysis, we show that, on the one hand, due to the difficulty in contact tracing and the lack of accurate information, controlling the transmission rate may be difficult, leading to instability. On the other hand, although hospitalizations and ICUs are easily accessible and less noisy than the rate of new infections, a delay is introduced in the control loop, which may endanger system stability. Our framework allows assessing the impact on economic and social costs of the above strategies in a scenario enriched by: i) population heterogeneity in terms of mortality rate and risk exposure, ii) closed-loop control of the epidemiological curve, and iii) progressive vaccination of individuals.
\end{abstract}

\section{Introduction}
	{\color{black}{Throughout the recent COVID-19 pandemic, governments worldwide faced the challenge of developing effective strategies to contain the virus while minimizing its economic and societal impact. As new waves of infection emerge, swift implementation of regulations becomes imperative to curtail human mobility and activities that facilitate virus transmission. From the pandemic onset, it was evident \cite{Anderson2020} that governments would face a delicate balance between reducing COVID-19 fatalities and mitigating the economic fallout caused by the spread of the virus. While individuals prioritize preserving life, governments must take action to manage the inevitable economic downturn.
			
	Once the epidemiological curve starts to decline, policymakers in democratic nations often face pressure from various stakeholders to ease restrictions and resume suspended activities. Striking an equilibrium becomes desirable to ensure both economic and social well-being, where the spread of the virus is effectively controlled with minimal limitations, as emphasized in \cite{Kantner20}.
			
	About a year into the COVID-19 pandemic, several vaccines against the virus started to become available. However, the availability of these vaccines was initially limited, and their effectiveness in preventing infection was not yet fully understood. In light of these uncertainties, developing strategies to prioritize vaccine distribution became challenging. The interested reader can refer to \cite{8a74160db44a45239bdb2a4b6e2e711d} for insights and analysis on the complexities of developing effective vaccine prioritization strategies in such uncertain times.
			
	}}

	{\color{black}{In this study, our objective is to enhance our comprehension of the choices that policymakers have at their disposal during conditions of a pandemic emergency. We aim to explore different strategies in-depth, measures and approaches that policymakers can consider and implement to effectively respond to the challenges posed by a pandemic aiming at safeguarding public health, minimizing the impact on society, and ensuring the well-being of individuals.

\subsection{Modeling epidemic spread}
				
	{\color{black}{The introduction of the SIR model in \cite{KermackMcKendrick1927} marked a significant milestone in utilizing mathematical models to understand disease dynamics and forecast the spread of epidemics. SIR-like models divide the population into compartments based on disease status, defining transitions among them. We refer to \cite{SM} for an overview.}} Many of these models assume homogeneous populations, disregarding the intrinsic heterogeneity of society. The authors of \cite{Vespignani_contacts} provide a more accurate description, where they introduce age-specific contact patterns.
				
	{\color{black}{Historically, the focus of studies in the field of epidemiology has primarily been on understanding the phenomenological aspects of epidemics. These studies aimed to assess the accuracy of different models in predicting the evolution of epidemics and identifying easily interpretable parameters that capture the qualitative behavior of infectious diseases. One such parameter of significant interest has been the \textit{basic reproduction number}, which characterizes the conditions under which an infection outbreak can occur.
	There is a growing recognition of the need to proactively design \textit{simple} and effective control measures to combat and mitigate the spread of infectious diseases in the possibility of future pandemics.

\subsection{Control via non-pharmaceutical interventions}

	In this paper, we assume a central planner perspective, where the government can impose control measures on the population for the overall benefit of public health.
	Several papers published in the 1970s, such as \cite{Taylor68,Abakuks73,dcf1cfcc-df09-3675-9da2-3693405a72f1,WICKWIRE1975325,SethiStaats1978}, focused on studying initial optimal control problems in the context of the classical SIR model. Building upon these foundations, subsequent works like \cite{hansen2011optimal} and \cite{RePEc:cwl:cwldpp:2229} expanded on the research and extended the models. These studies specifically addressed the challenge of minimizing the size of an epidemic outbreak and the cost of interventions. The control mechanisms explored in these models included regulating social distancing levels and implementing measures like isolation. These control strategies were also subject to rate constraints, ensuring the rate remained below a specified threshold.
	Papers on optimal control problems \cite{Behncke00,BOLZONI201786} consider the minimization of a composite function taking into account the epidemic cost, related to the size of the outbreak or number of deaths, and the economic cost. The control perspective has been widely embraced in recent literature regarding the COVID-19 pandemic. Consequently, there has been considerable discussion about the effects of lockdown measures on healthcare, society, and the economy.
	The problem of minimizing the cost of a lockdown under the only constraint of maintaining the infection below a certain threshold to cope with ICU congestion problems is also considered in \cite{doi:10.1287/mnsc.2022.4318}. 
						
\subsection{Control via vaccination}

	The paper by \cite{britton2019optimal} focuses on the optimal control of vaccination dynamics during an influenza epidemic. It provides insights into the design of vaccination strategies to effectively control the spread of the disease, considering factors such as limited vaccine supply and variations in transmission and severity across different groups.
	In \cite{feng2013optimal}, optimal vaccination and treatment strategies are studied in a multi-group epidemic model. The analysis explores the trade-offs between vaccination coverage and treatment allocation to maximize overall disease control, considering the interactions between different population groups.
	Finally, \cite{zhang2014optimal} explores the optimal timing and allocation of vaccinations based on age groups to maximize the effectiveness of the vaccination campaign and minimize the spread of infectious diseases.
	The research conducted in \cite{altarelli} and \cite{galeotti2013strategic} contributes to the field by addressing the challenge of resource allocation for vaccination efforts in the context of epidemic control. By utilizing optimization techniques, these studies provide insights into the most effective strategies for targeting specific nodes in a contact network or groups within a population, considering both the budgetary constraints and the dynamics of disease transmission.

\subsection{Main contribution}
	This paper investigates a novel class of compartmental models that draw inspiration from the features of the COVID-19 epidemic.  We enhance  previous models by incorporating: i) variability in mortality rate and risk exposure among different population segments, ii) closed-loop control mechanisms to regulate the epidemiological curve, and iii)  progressive vaccination campaigns.
}}
	More precisely, differently from previous works on the subject (see \cite{matrajt21}, \cite{SaadRoy2021}, \cite{Goldsteine21}, \cite{galvani09}, \cite{Sandmann2021}, \cite{monod21}, \cite{prio21}), our modeling framework explicitly represents the heterogeneity of risk exposure across population segments. 
				
	Our study takes a different perspective than the optimal control approach. We deliberately examine \textit{simple} control strategies to provide practical insights and guidelines for decision-makers who may not have access to sophisticated optimization techniques or detailed knowledge of the underlying epidemic mathematical laws  (e.g., parameters). Our focus is on identifying practical control strategies that can be implemented in real-world scenarios without the need for complex computations or extensive data.
				
	Indeed, optimal control approaches in epidemiology have certain drawbacks worth mentioning. These approaches often lack closed-form analytical solutions, requiring numerical methods for their implementation \cite{2156-8472_2022_1_201, NBERw27794}. While numerical solutions can provide valuable insights, they can be challenging to interpret and translate into practical control measures. This lack of interpretability hampers the ability to univocally understand the implications and consequences of the obtained optimal control strategies.
	Another drawback is that optimal control solutions may strongly depend on parameters, leading to abrupt changes in the optimal interventions. These \lq \lq phase transitions" can make implementing and managing the control measures in practice difficult. Minor changes in the parameters or system conditions may result in significant shifts in the optimal strategies.
	Furthermore, optimal control approaches often rely on detailed and precise knowledge of the system dynamics, including accurate parameter values and functional forms of the underlying equations. In real-world scenarios, such detailed information may be unavailable or subject to significant uncertainties.
	Considering these drawbacks, alternative approaches focusing on more straightforward and tangible strategies become appealing. Strategies based on explicit rules or thresholds offer the advantage of being easily understandable, implementable, and interpretable. They allow policymakers and stakeholders to make informed decisions based on practical considerations, even with uncertainties or limited data availability.
	In the same spirit as \cite{9682977}, we will mainly focus on two control feedback strategies based on controlling the rate of new infections or maintaining the occupation of healthcare facilities below a given level, and we evaluate the economic cost of non-pharmaceutical interventions and the social cost in terms of mortality.
	For the sake of simplicity and analytical tractability, we consider an \lq\lq ideal"  scenario in which the system operates near the equilibrium point, where the effective reproduction number equals one (we provide local stability results).
	This regimen appears to be a desirable condition whereby the number of infected individuals, and thus those requiring intensive treatment, is maintained at a sustainable level, even over long periods, while applying minimal durable mobility restrictions.
				
	It is worth noting that while it can be expected that some individuals will naturally reduce their interactions out of fear of illness, our model does not explicitly incorporate this behavior as explored in \cite{carmona16} and \cite{Aurell22}.
	Additionally, we do not consider the concept of \lq\lq cost of anarchy" as explored in~\cite{Romuald20}.

	Finally, with the same spirit, we compare simple vaccination policies based on different assigned priorities and intervals between the first and second dose administration. 

	It appears that in the literature, no works explicitly address the combined aspects of rate control/ICU occupation in feedback and vaccination prioritization, exploring them separately. 
	The absence of research specifically addressing the integration of non-pharmaceutical control in feedback and vaccination prioritization highlights a significant gap in the literature, particularly noteworthy given the current context of the COVID-19 pandemic, where vaccination campaigns have been implemented alongside other control measures. 
	Our paper aims to fill, at least in part, this gap.
				
\subsection{Outline of the paper}
	Section \ref{subsec:model} presents an in-depth examination of the proposed extended SIR model in the absence of public intervention. Section \ref{sect:epidemic} discusses two control measures to limit the spread of the epidemic: non-pharmaceutical, e.g., lockdown measures, and pharmaceutical interventions, i.e., vaccination prioritization. The impact of controlling the infection rate or exerting control on hospitalizations and ICU occupancy through mobility reduction 	measures are explored in Sections  \ref{sect:lambda} and Section \ref{sec:controlHT}, respectively. 
	In Section \ref{sec:prelimres}, we support our modeling choices by first comparing the performance of our \textit{simple} control strategies with optimal control, and then showing the impact of population heterogeneity on the dynamics.
	We perform extensive simulations and a thorough discussion for a \textit{reference} scenario inspired by the COVID-19 pandemic in Section  \ref{sec:reference}. Section \ref{sect:conc} concludes the article. We provide in the Appendix a more comprehensive exposition of the model, and we include details on the data-driven derivation of the $f_{r,p}$ distributions and the choice of parameters of the reference scenario.
}}
		
\section{Base model}\label{subsec:model}
	
	We start by describing a base version of our compartmental model to describe the spread of a disease in a non-homogeneous population of size $N$ in the absence of any intervention (either pharmaceutical or non-pharmaceutical).
		
	Socio-demographic groups are described by the joint distribution $f_{r,p}$ related to the risk exposure $r$ (also referred to as contact rate) and the death probability $p$ of the individuals.

	We consider six epidemiological states: let $S_{r,p}(t)$, $I_{r,p}(t)$, $M_{r,p}(t)$, $H_{r,p}(t)$, $T_{r,p}(t)$, and $D_{r,p}(t)$ denote the number of individuals characterized by $(r,p)$ who at time $t$ are susceptible, infected, immune, hospitalized, under intensive treatment and dead, respectively.
		
	{\color{black}{The system presented here can be derived from stochastic processes, as in \cite{britton2019mathematical}. Thus there exists an underlying individual-based model in which all the states have a probabilistic and statistical interpretation. In particular, the amount of time spent by an individual in the infected, hospitalized, intensive therapy, the immune compartment is exponentially distributed with mean values $1/\gamma$, $1/\phi$, $1/\tau$, $1/\mu$, respectively. These parameters determine the average time an individual remains in each respective state. We utilize a set of ordinary differential equations to describe the system dynamics. \mbox{Figure \ref{fig:our_model}} visually depicts the transitions between different compartments and illustrates the flow of individuals within the system. The transitions between different states in the epidemiological model can be described as follows:

	\begin{itemize}
		\item Susceptible to Infected: Susceptible individuals (S) become \textcolor{black}{I}nfected (I) when they come in contact with infected individuals. The corresponding transition rate depends on the contact rate (risk exposure $r$) and the number of infected individuals in the population.
		\item  Infected to Immune: Infected individuals (I) can recover from the disease and acquire immunity, transitioning to the i\textcolor{black}{M}mune state (M). The \textit{recovery} rate governs the transition and depends on the average infection duration before recovery. 
		\item Infected to Hospitalized: Some \textcolor{black}{I}nfected individuals (I) may develop severe symptoms and require \textcolor{black}{H}ospitalization (H).Various factors influence the corresponding transition rate, such as the healthcare capacity, the proportion of infected individuals needing hospital care, and also disease severity linked to the \textit{fragility} of individuals. 
		\item Hospitalized to Under Intensive Treatment: Hospitalized individuals (H) who require intensive care \textcolor{black}{T}reatment may be transferred to the intensive therapy state (T). The corresponding rate depends on factors such as the availability of intensive care units and the duration of hospitalization before the transfer to ICU.
		\item Under Intensive Treatment to  Deceased:  Unfortunately, some infected individuals under intensive \textcolor{black}{T}reatment (T) may succumb to the disease and move to the \textcolor{black}{D}eceased state (D).  
	\end{itemize}
				
	It is important to note that the specific transition rates between states are governed by the model's parameters, which can be estimated based on empirical data or derived from previous studies. These transition dynamics capture the progression of the disease within the population and are crucial for understanding the spread and impact of the epidemic. In the Appendix (Appendix A-C), we show how  direct transitions $I \to D$ and $H\to D$ can be added to the model.
		}}
	More precisely, the following set of ordinary differential equations describes the system dynamics:
	\begin{align}\begin{split}\label{eq:dyn_sys}
			\dot S_{r,p}(t) & = -{\sigma(t)}
			\left( \sum_{r',p'} r' I_{r',p'}(t) \right)
			\frac{r S_{r,p}(t)}{\sum_{r',p'} r' N f_{r',p'}} + \mu M(t) \\
			\dot I_{r,p}(t) & = {\sigma(t)}
			\left( \sum_{r',p'} r' I_{r',p'}(t) \right)
			\frac{r S_{r,p}(t)}{\sum_{r',p'} r' N f_{r',p'}}
			- \gamma I_{r,p}(t)  \\
			\dot H_{r,p}(t) & = \gamma p_{r,p}^{IH}  I_{r,p}(t) - \phi H_{r,p}(t)  \\
			\dot T_{r,p}(t) & = \phi p_{r,p}^{HT} H_{r,p}(t) -\tau T_{r,p}(t) \\
			\dot D_{r,p}(t)   &= \tau p_{r,p}^{TD}(t) T_{r,p}(t) \\
			\dot M_{r,p}(t)  &= 
			\gamma  (1-p_{r,p}^{IH}) I_{r,p}(t) 
			+ \phi (1-p_{r,p}^{HT}) H_{r,p}(t) \\
			& + \tau (1-p_{r,p}^{TD}(t)) T_{r,p}(t)
			-\mu M(t)
		\end{split}
		\vspace{-2mm}
	\end{align}
	where 
	$\sigma(t) \geq 0$
	represents all exogenous (uncontrolled) factors changing the infection strength (e.g., seasonal effects).
	In this paper, we will assume for simplicity that $\sigma(t) = \sigma$	is constant.       
		
	\begin{figure}[h!]
		\centering
		\includegraphics[width=0.95\columnwidth]{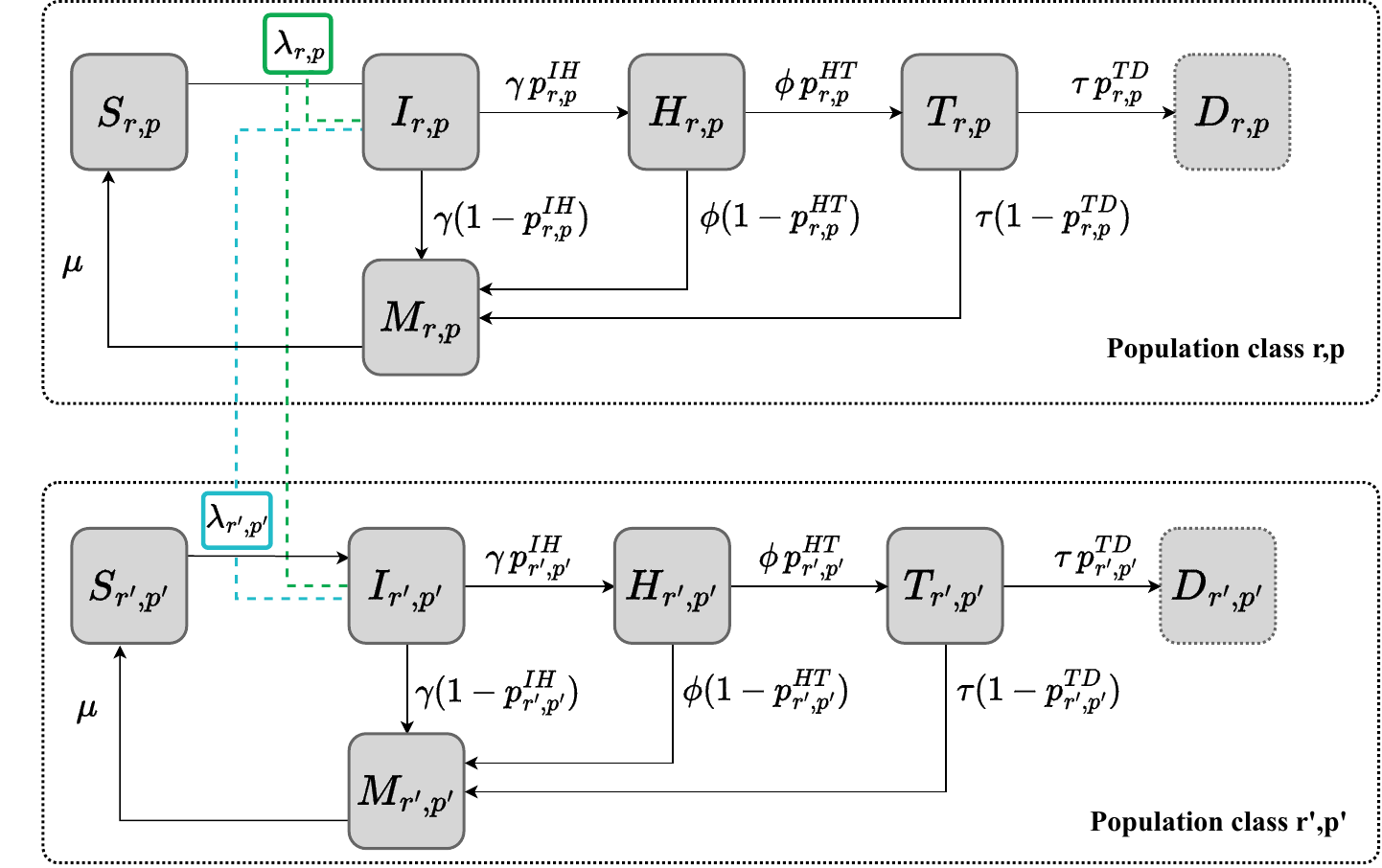}
		\caption{Schematic representation of the proposed model.}
		\label{fig:our_model}
	\end{figure}
		
	The total (uncontrolled) rate of new infections is equal to:
	$$\lambda_{\textcolor{black}{U}}(t) = \sigma(t) \left( \sum_{r,p} r I_{r,p}(t) \right)
	\frac{\sum_{r,p} r S_{r,p}(t)}{\sum_{r,p} r N f_{r,p}}, $$
	The total number of susceptible people is $S(t) = \sum_{r,p} S_{r,p}(t)$.
	Similarly, we introduce the total number of people in the other compartments: $I(t)$, $H(t)$, $T(t)$, $M(t)$, $D(t)$.
	Probabilities $p_{r,p}^{IH}$, $p_{r,p}^{ HT }$ and $p_{r,p}^{ TD }(t)$
	denote the probability that an individual of type $(r,p)$ moves between the two compartments indicated in the superscript.
	We make probability $p_{r,p}^{ TD }(t)$ depend on $T(t)$, i.e., on the instantaneous total number of people in ICUs since the death probability dramatically increases when ICUs are saturated.
	Denoted with $\widehat T$ the number of available ICUs,
	when $T(t) \le  
	\widehat{T}$, the overall death probability of an infected person
	is assumed to be equal to $p$:
	\begin{equation}\label{eq:prodp}
		p_{r,p}^{IH} \cdot p_{r,p}^{HT} \cdot \hat{p}_{r,p}^{TD} = p \qquad 
		\text{if }T(t) \le {\widehat T}, 
	\end{equation} 
	where $\hat{p}_{r,p}^{TD}$ is the probability
	to transit from state $T$ to state $D$ in \lq normal' conditions, i.e., when
	$T(t) \le 
	{\widehat{T}}$. Therefore, $p_{r,p}^{TD}(t) = \hat{p}_{r,p}^{TD}$ as long as
	$T(t) \le 
	{\widehat{T}}$.   
	
	When $T(t) > 
	{\widehat{T}}$, we assume that the death probability of people who cannot receive treatment is increased by a factor $\theta$,
	hence $p_{r,p}^{ TD }(t)$ is dynamically adjusted as follows:
	\begin{multline}
		p_{r,p}^{TD}(t) = \hat{p}_{r,p}^{TD} \frac{ {\widehat{T} }   } {T(t)} +  
		\min\{1,\theta \cdot \hat{p}_{r,p}^{TD}\} 
		\frac{T(t)-{\widehat{T}} }  {T(t)}.
			\label{eq:pTD}
	\end{multline}
			
	We consider the case in which individuals might lose immunity with rate $\mu$, thus becoming susceptible again. 
	It should be noticed that the mass preservation 
	$\dot S_{r,p}(t) +\dot I_{r,p}(t) +\dot M_{r,p}(t) + \dot H_{r,p}(t) + \dot T_{r,p}(t) + 
	\dot D_{r,p}(t)=0 $ holds for all $t\geq0$.
	
	\medskip
	{\color{black}{References as \cite{acemoglu2020optimal,bhattacharyya2021modelling,britton2019mathematical,chowdhury2020dynamic,li2020epidemiological} explore SIR-like models with various extensions, including population heterogeneity, different compartments (such as susceptible, infected, immune, hospitalized, under intensive treatment, and deceased), and considerations of specific epidemics like COVID-19. 
			They provide insights into such extended SIR models' dynamics and control measures. 
			Unlike \cite{charpentier20}, our model does not distinguish between infected individuals who remain undetected and those who are detected, nor does it consider this distinction for those who recover. 
			Nonetheless, our proposed model introduces several innovative features that can be summarized as follows.}}
	\begin{remark}[Heterogeneity of population in terms of fatality rate and
		risk exposure]
		At the country level, populations exhibit remarkable differences in their characteristic features. For instance, in terms of age distribution, overall health condition, and daily contacts among individuals, which in turn depend on the country's customs and, more broadly, on its welfare.
		Heterogeneity in the population contact patterns may play a role in disease transmission, as it may favor a faster virus outbreak. Starting from available data, we characterized the population by a joint distribution function $f_{r,a}$. We then used it to derive the distribution  $f_{r, p}$. The last step was accomplished by exploiting data that relate death probability $p$ to the age $a$ of individuals.
		Specific distributions $f_{r, p}$ for several countries were obtained from contacts patterns reported in \cite{Vespignani_contacts} and the case fatality rate from \cite{CFR2020}, specified for various age classes.
		\begin{figure}[h!]\begin{tabular}{ccc}
				\includegraphics[trim=1.9cm 0.6cm 1.9cm 0.8cm, clip, width=0.28\columnwidth]{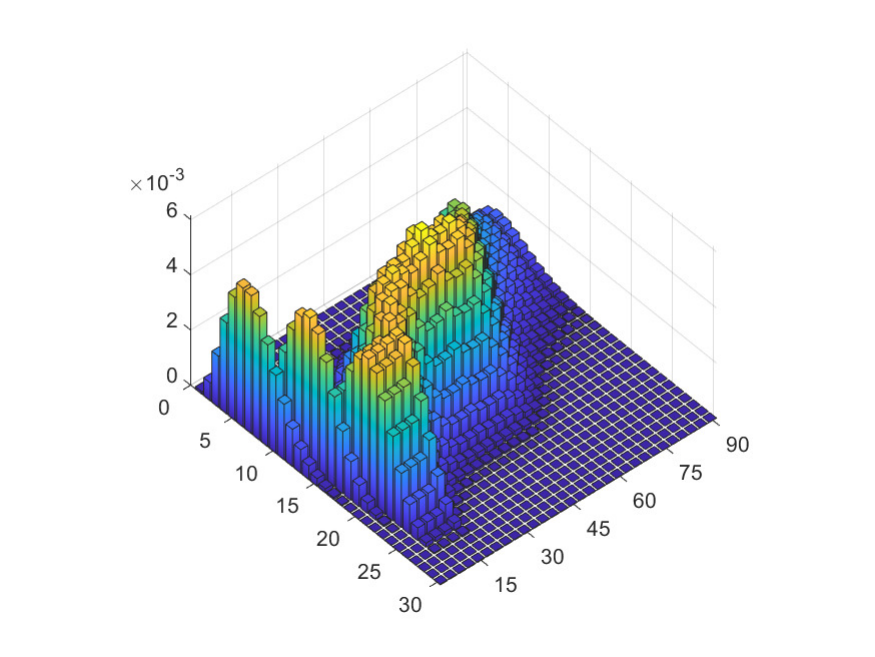}&
				\includegraphics[trim=1.9cm 0.6cm 1.9cm 0.8cm, clip, width=0.28\columnwidth]{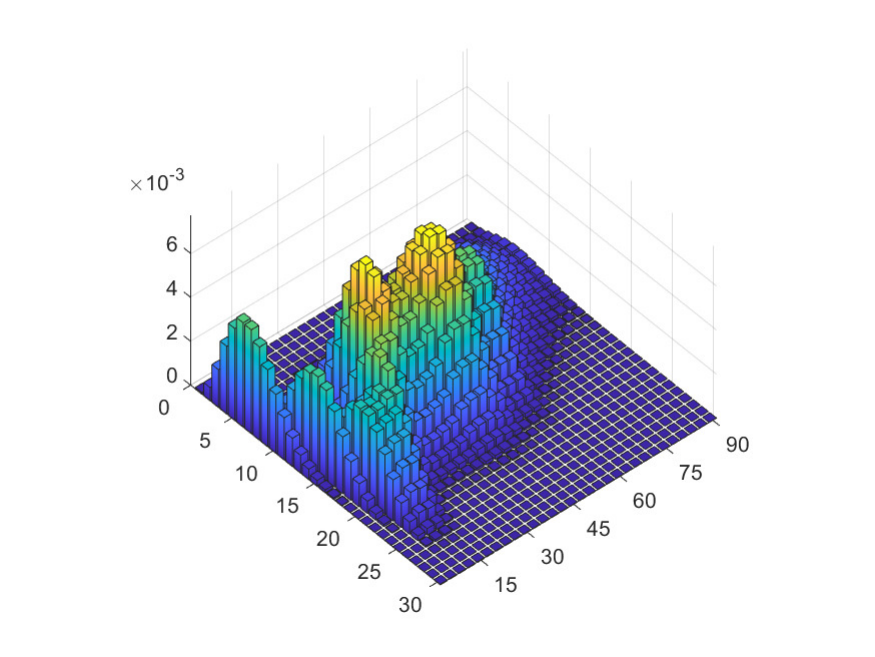}&
				\includegraphics[trim=1.9cm 0.6cm 1.9cm 0.8cm, clip, width=0.28\columnwidth]{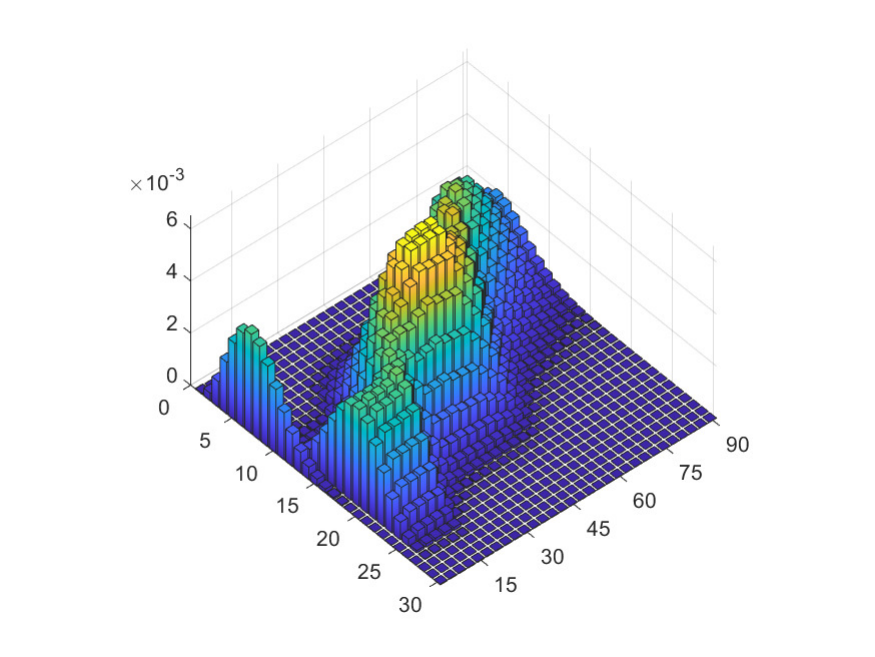}\\
				(a)&(b)&(c)
			\end{tabular}
			\caption{$f_{r,a}$ distributions for different countries: (a) Australia (b) China (c)~Italy. The age bins are $A=\{\text{0-2},\text{3-5},..,\geq 87 \}$ the daily number of contacts $r \in [0,r_{\max}=30]$.} \label{fig:synthetic_populations}
		\end{figure}

		Few examples of distributions $f_{r,a}$, capturing the negative correlation between risk exposure and age for different countries, are depicted in Fig. \ref{fig:synthetic_populations}. We refer the readers to the Appendix for more details.
		\label{remark:datadriven}
	\end{remark}
	{\color{black}{The continuous model used in this study can be interpreted as a mean field approximation of an epidemic model that operates over a dynamic network  \cite{pastor2015epidemic,newman2002spread,keeling2005networks}. According to this interpretation, the risk exposure parameter represents the average number of contacts (per time unit) an individual experiences with others over a fixed time window. Therefore it can be seen as the degree of the corresponding node within a network, in which nodes represent individuals, and the edges represent the contacts between them. Note that pairs of individuals establishing contacts are randomly selected, as for the configuration model. This approach allows us to understand the dynamics of epidemics in terms of the interactions between individuals in a network setting.
	}}
	
	\begin{remark}[Quadratic dependence on the  risk exposure $r$]
		Note that individuals with large $r$, i.e., pronounced social attitudes, represent at the same time the component of the population with the highest risk of infection and the highest chance of transmitting the disease.
		Therefore,  the \lq\lq impact" of every individual to the spread of the infection depends quadratically on $r$.
	\end{remark}
	{\color{black}\begin{remark}[Single class and reduction to the SIR model]
			Consider a scenario where the population consists of a single class with parameters $(r,p)$ and set $\mu=0$. We remark that by merging the compartments H (hospitalized), T (treated), and I (infected) and combining those representing the recovered (M) individuals and deaths (D) into a single recovered compartment, the model described by the system of equations (\ref{eq:dyn_sys}) simplifies to the classical SIR model.
	\end{remark}}
			
	\begin{remark}[Edge-perspective analysis]\label{oss:edge_perspective} 
		Defining 
		$\widetilde{I}(t)=\sum_{r,p}rI_{r,p}(t)$
		as the number of infected contacts, 
		multiplying the first and second equation in (\ref{eq:dyn_sys}) 
		by $r$ and summing over $r$ and $p$, we obtain
		$\dot{\widetilde{I}}(t) = \gamma \left(\mathcal{R}(t)-1 \right) \widetilde{I}(t)
		$
		where $\mathcal{R}(t) = \frac{\sigma \sum_{r,p} r^2 S_{r,p}(t)}{\gamma \sum_{r,p} r N f_{r,p}}$.
		
		At early stages of epidemic, we  can approximate  
		$S_{r,p}(t) \approx N f_{r,p}$, obtaining:
		$\dot{\widetilde{I}}(t) = \gamma \left(\mathcal{R}_0-1 \right) \widetilde{I}(t)
		$
		where we define the related basic reproduction number \mbox{$\mathcal{R}_0=\frac{\sigma}{\gamma}\mathbb{E}[r^2]/\mathbb{E}[r]$.}
		As it is  clear from the system of equation describing the evolution of the state variables, an edge-perspective analysis provides a fundamental tool to study the dynamics as a natural generalization of the SIR model. 
			
		\end{remark}

\section{Epidemic control}\label{sect:epidemic}
	To mitigate the epidemic, several interventions are possible:
	(a) investments in the public health system, e.g., increasing the available number of ICUs $\widehat T$ and hospitalization facilities $\widehat H$, (b) non-pharmaceutical interventions, namely, public health measures 
	preventing and/or controlling virus transmission in the community; (c) vaccination that aims to reduce both the transmission and clinical severity of the disease.

	Our analysis will focus on quantifying the cost and the impact of different control strategies that jointly exploit 
	non-pharmaceutical interventions and vaccination. 			
	\vspace{-0.2 cm}
\subsection{Control via non-pharmaceutical interventions}
		
	In our framework, we do not model the effects of social distancing and other countermeasures at a microscopic (class-specific) level. 
	Instead, we summarize their effects by a single control parameter $\rho(t)$ that scales down the overall rate of potential (uncontrolled) new infections. 
	
	Specifically, we \textcolor{black}{include the control in the model described by} \eqref{eq:dyn_sys} by posing the actual intensity of new infections ${\lambda}(t) $ equal to  
	$ \frac{{\lambda}_{U}(t)}{\rho(t)} $,   
	leading to an effective reproduction number:
	$$\mathcal{R}^{\rho}{(t)}=\frac{\sigma}{\rho(t) \gamma} \frac{\sum_{r,p}r^2S_{r,p}(t)}{\sum_{r,p}Nrf_{r,p}}.$$
	
	In this scenario, we will distinguish two main contributions to the cost: the social and the economic cost.
	{\color{black}{It is crucial to note that the distinction between social and economic costs is not always clear-cut. Lockdown measures, while aimed at minimizing the social cost of the pandemic in terms of reducing deaths, can also have economic repercussions. Similarly, the economic cost of the pandemic, such as job losses and reduced economic activity, can have social implications. 
	Moreover, for technical reasons in some cases we add a third component related to healthcare stress to the cost.  Accordingly we  define:
	}}
	\begin{enumerate}
		\item[(a)] the {\em social cost}, evaluated in terms of {\color{black}{the cumulative number of deaths as defined \cite{8085142}}};
		\item[(b)]{\color{black} the {\em  stress on the healthcare system} induced by the disease's severity; }
		\item[(c)] the {\em economic cost} $\mathfrak{C}=\mathfrak{C}(\rho)$, since widespread lockdowns cause a massive negative impact on the economy. 
	\end{enumerate}
	{\color{black}{In Figure \ref{Fig:1} we show some examples of economic costs as a function of the control parameter $\rho$. The economic costs are assumed monotone increasing with $\mathfrak{C}(1)=0.$ 
			\begin{figure}[h]\label{fig:costs}
				\begin{center}
					\begin{tikzpicture}[thick, scale=0.6]
						\begin{axis}[
							xmin = 0, xmax = 10,
							ymin = -0.5, ymax = 10,
							axis lines = left,
							axis x line = center,
							axis y line = center,
							x label style={at={(axis description cs:0.6,0)},anchor=north, below=1mm, font=\large},
							xlabel={Control $\rho$},
							y label style={at={(axis description cs:0,0.9)},rotate=90,anchor=south east, font=\large },
							ylabel={Economic cost $\mathfrak{C}(\rho)$},
							xtick = {1,10},
							ytick = {0,1,10},
							legend pos={north west}
							]
							\addplot[
							domain = 1:10,
							samples = 100,
							line width = 3pt,
							color = teal
							]
							{10*(1-1/x)^2};
							\addplot[
							domain = 1:10,
							samples = 100,
							line width = 3pt,
							color = cyan
							]
							{0.1*(x-1)^2};
							\addplot[
							domain = 1:10,
							samples = 100,
							line width = 3pt,
							color = violet
							]
							{0.9*(x-1)};
							\legend{$10(1-1/\rho)^2$,$0.9(\rho-1)$,$0.1(\rho-1)^2$}
						\end{axis}
					\end{tikzpicture}
				\end{center}
				\caption{Examples of economic costs as a function of control parameter}\label{Fig:1}
			\end{figure}
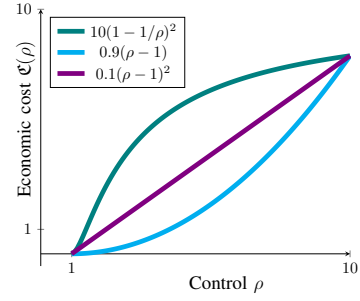
	}}
	In the optimal control formulation (see \cite{Kantner20} and reference therein) a terminal cost is generally defined by taking a linear combination of the above  costs and the policymaker aims at solving the following optimization problem:
	{\color{black}{\small{\begin{align}\vspace{-2mm}
					\rho^{\star}(t_{\max})&=\argmin{\rho:[0,t_{\max}]\rightarrow[1,\infty)}\kappa_1 \frac{D(t_{\max})}{N} + \nonumber\\
					&\quad+ \int_0^{t_{\max}}\left[\kappa_2 \left(\frac{T(t)}{N}\right)^\zeta+\kappa_3\mathfrak{C}(\rho(t))\right]\mathrm d t \nonumber \\
					& \text{s.t. dynamics in \ \eqref{eq:dyn_sys}} \label{eq:obj}
	\end{align} }}}
	\textcolor{black}{where the exponent $\zeta$ is typically assumed greater than~1, while
	$\kappa_{1},\kappa_2,\kappa_3 \ge 0$ are the parameters that weigh the social, the healthcare stress  and   the economic costs in the objective function, according to how much one values one over the others.}
	We emphasize that the selection of $t_{{\max}}$ should be carefully considered based on the specific context and dynamics of the epidemic under investigation.
	
	{\color{black}{Our goal is not to develop a mathematical theory of optimal control for epidemics but to provide a practical framework that informs public policy in controlling the spread of epidemics. We intend to offer decision-makers a means to compare and evaluate a set of feasible controls, allowing them to make informed choices based on the outcomes and trade-offs associated with different control strategies.}}
	Similarly to  \cite{9682977}, we will consider two \textit{simple} control strategies:  
	
	\begin{cpolicy1} [Control on New Infections]\label{policy-lambda}
		The rate of new infections is tightly controlled and kept at a certain desired level $\lambda_C$. The main goal is to avoid congestion in the sanitary system by controlling the circulation of the virus. 
	\end{cpolicy1}
		
	\begin{cpolicy2}  [Control on Hospitalizations and intensive Therapy occupancy]  \label{policy-therapies}
		It directly uses the current level of hospitalization/intensive-therapy occupancy as a control \mbox{signal}. 
		Such a signal is readily available and less noisy than the rate of new infections. However, it may introduce a delay in the control loop, which may endanger system stability.
	\end{cpolicy2}
			
\subsection{Control via vaccination prioritization}\label{subsec:priority}

	Vaccines are assumed to guarantee partial protection. 
	According to classification in \cite{MATRAJT201517}, we consider two efficacy descriptors: reduction in the probability of becoming infected (vaccine efficacy on susceptibility) and reduction in the pathogenicity (vaccine efficacy to prevent or diminish symptoms). 
	For simplicity, we neglect the vaccine response transient, and we consider a single type of vaccine administered in two doses separated by a fixed interval of $\Delta$ days.
	We assume that the administration rate of either dose is fixed, equal to $\xi$, so the entire population can be potentially vaccinated (with two doses) after $\mathcal{T}_v$ days. Hence we set $\xi= N/(\mathcal{T}_v- \Delta)$.
	
	Let $\mathsf{VE}^{1}, \mathsf{VE}^{2}$ be the vaccine efficacy on susceptibility after one or two doses, respectively. 
	Moreover, we assume that mortality is reduced by a factor $q_{\texttt{post}}$ after a single dose of vaccine.
	
	We assume that $N^{\textrm{novax}}$ people refuse vaccination uniformly distributed over the population. Their state evolution is still described by equations \equaref{dyn_sys}. Let $S_{r,p}^{\textrm{novax}}(t)$ be the number of no-vax people in class $(r,p)$ who are still susceptible at time~$t$.
	
	We describe the dynamics assuming individuals do not return to the susceptible state after infection or vaccination. This extension is not difficult, but we omit it for brevity.

	Vaccinations require the addition of a few more compartments:   
	Let $V_{r,p}^{\textrm{1m}}(t)$ be the number of people in class $(r,p)$ who have received just the first dose, which is already effective against the virus, i.e., they can no longer be infected.
	Let $V_{r,p}^{\textrm{1s}}(t)$ be the number of people in class $(r,p)$ still susceptible after receiving just the first dose.
	Let $V_{r,p}^{\textrm{2m}}(t)$ be the number of people in class $(r,p)$ who have received both doses and are immune. At last, let $V_{r,p}^{\textrm{2s}}(t)$ be the number of people in class $(r,p)$ who have received both doses but are still susceptible.
	Due to strict prioritization among classes, a given  class $(r,p)$ receives the first dose at full rate $\xi$ only within a specific time window:  $[\mathcal{T}^{\min}_{r,p},\mathcal{T}^{\max}_{r,p}]$ (to be specified later):
	\begin{equation*}
		\xi_{r,p}^{(1)}(t)=
		\left\{\begin{array}{ll}
			0  &   t< \mathcal{T}^{\min}_{r,p}\\
			\xi  &    \mathcal{T}^{ \min}_{r,p} \le t < \mathcal{T}^{\max}_{r,p} \\ 
			0  &  t \ge \mathcal{T}^{\max}_{r,p}\\
		\end{array} \right.
	\end{equation*}
						
	Let  $V_{r,p}^1(t) = \int_{t-\Delta}^t  \xi _{r,p}^{(1)}(t) \mathrm{d} t $
	be the number of people in class $(r,p)$ who have received just the first   dose of vaccine at time $t$.
	The second dose of vaccine is administered at a rate
	\[
	\xi_{r,p}^{(2)}(t)=\frac{V_{r,p}^{1s} (t)+V_{r,p}^{1m} (t) }{V_{r,p}^1(t)}  \xi_{r,p}^{(1)}(t-\Delta)
	\]
	only to individuals who have received the first dose and have not been 
	infected in the meanwhile.
	At last, let
	\[
	\widehat S(t)= \sum_{r,p}r(S_{r,p}(t)+
	V_{r,p}^{\textrm{1s}}(t)+
	V_{r,p}^{\textrm{2s}}(t)+
	S_{r,p}^{\textrm{novax}}(t))
	\]
	be the total number of susceptible edges at time $t$.
	Note that
	$\lambda(t) = \frac{\sigma}{\rho(t)} \left( \sum_{r,p} r I_{r,p}(t) \right)
	\frac{\widehat{S} (t) }{\mathbb{E}[r]N} $.
	Since people who receive at least one dose are less likely to die,
	we need to keep track of them, hence vaccinated people who get infected 
	traverse a separate chain of compartments  
	$I_{r,p}^v(t), H_{r,p}^v(t), T_{r,p}^v(t)$ with respect to those
	who do not receive any dose \textcolor{black}{(see Figure 11 in the Appendix)}.
	Dynamics governing the evolution of $H_{r,p}^v(t), T_{r,p}^v(t)$
	are analogous to those in \equaref{dyn_sys} with
	the only difference that $p_{r,p}^{TD}(t)$ is replaced by $p_{r,p}^{TD}(t)/q_{\texttt{post}}$.
	\textcolor{black}{The complete system of differential equations is an extension of \equaref{dyn_sys}, it is omitted here for brevity and reported in sect.~V of \cite{SM}}
	The vaccination window for each class is computed based on the class priority:
	$
	\mathcal{T}^{(1), \max}_{r,p}= \inf \{t: 	S_{r,p}(t) =0 \}
	$; 
	$
	\mathcal{T}^{(1),\min}_{r,p}= \max_{ (r',p')\in HP(r,p)}\{  \mathcal{T}^{\max}_{r',p'}\}
	$,
	where $HP(r,p)$ is the set of classes with higher priority than~$(r,p)$.

	We consider the  Most Vulnerable First (MVF) and the Most Social First (MSF) policies: 
	\begin{vpolicy1}[Most Vulnerable First]\label{policy:MVF}
		The MVF policy aims to protect the most clinically vulnerable people, with the goal of minimizing the number of deaths. It prioritizes classes with a higher value of $p$. For the same $p$, classes with higher $r$ are vaccinated first.
	\end{vpolicy1}	 
	\begin{vpolicy2}[Most Social First]\label{policy:MSF}
		The MSF policy prioritizes people with a high contact rate, aiming to minimize the force of infection. Classes with a higher value of $r$ are prioritized. For the same $r$, classes with higher $p$ are vaccinated first.
	\end{vpolicy2}
	The MSF policy is similar in spirit to the degree-based vaccination policy in contact networks  \cite{9716860}, which targets the high-degree nodes first before moving on to lower-degree nodes.
	The interval $\Delta$ is another design parameter: prolonging the interval between doses, say from 3 to 12 weeks, might be a sensible choice under limited vaccine supplies, de facto minimizing hospitalization and deaths, especially when the efficacy of the first dose is sufficiently high.
						
	{\color{black}
		\begin{remark}
			The MSF and MVF policies defined above constitute only two examples of possible vaccination policies. 
			In principle, any possible prioritization (permutation) $\mathcal{\pi}_{(r,p)}$ of classes $(r,p)$ corresponds to a different vaccination policy.
			The optimal control problem defined in  \eqref{eq:obj} can be easily extended to take into considerations vaccinations as follows:
			
			{\small{\begin{align}\vspace{-2mm}
						(\rho^{\star}(t_{\max}), \pi^*_{(r,p)}  )&=\argmin{\stackrel{\rho:[0,t_{\max}]\rightarrow[1,\infty)}{   \pi_{(r,p)}   }} \kappa_1 \frac{D(t_{\max})}{N}+\nonumber\\
						&+\int_0^{t_{\max}}\left[\kappa_2 \left(\frac{T(t)}{N}\right)^{\zeta}+\kappa_3\mathfrak{C}(\rho(t))\right]\mathrm d t  \nonumber\\
						&\text{s.t. dynamics in (9), Appendix.}
			\end{align}}}
		\end{remark}
	}
						
\section{Control on new infections}\label{sect:lambda}
						
	In this section, we show that if function $\mathfrak{C}(\cdot)$ is  convex, we can devise a simple strategy  to minimize the overall economic cost. 
	As already observed,  a key role in the epidemic dynamics is played by $\widetilde I(t)$, which, roughly speaking,  represents the number of potentially infected contacts (see Remark \ref{oss:edge_perspective}). Thus, a sensible strategy is to control such a quantity.
	\textcolor{black}{In our derivations, we assume that 
		$S_{r,p}(t) \approx N f_{r,p}$, i.e., $S_{r,p}(t)$ can be considered constant.}
	{\color{black}{\begin{remark}The assumption of a constant number of susceptible individuals is accurate when dealing with a large population. In such a case, the rate of infection spread may have a minimal impact on the overall number of susceptible individuals, making it reasonable to treat it as constant for modeling purposes. The assumption holds true for a relatively short time horizon where the dynamics of infection spread and recovery do not significantly impact the population susceptibility.\end{remark}}}
	
	\textcolor{black}{This assumption allows for simplifications in the mathematical modeling and analysis. Indeed, given the definition of $\widetilde{I}(t)$, 
		multiplying the second equation in \eqref{eq:dyn_sys} by $r$ and summing over $r$ and $p$, we get:
		\begin{center}$
			\dot{\widetilde{I}}(t)=
			\gamma \left(\frac{\sigma}{\rho(t) \gamma} \frac{\sum_{r,p}r^2S_{r,p}(t)}{\sum_{r,p}Nrf_{r,p}}-1\right) \widetilde{I}(t) .
			$\end{center}
		under the assumption $ S_{r,p}(t)\approx Nf_{r,p}$, and defining  $\mathcal{R}^\rho(t)= \frac{\mathcal{R}_0}{\rho(t) }$, we obtain the equation:}
	\begin{equation}\label{eq:edgeinf}
		\dot{\widetilde{I}}(t) = \gamma \left( \frac{\mathcal{R}_0}{\rho(t)} - 1 \right) \widetilde{I}(t) =\gamma \left( \mathcal{R}^{\rho}(t) - 1 \right) \widetilde{I}(t).
	\end{equation}
						
\subsection{Minimizing the economic cost in a fixed window}

	Fixing a target value $\widetilde{I}^{\star}$  for  $\widetilde I(t)$, to be met \textcolor{black}{within a prefixed a time horizon $t_{\max}$}, Proposition \ref{prop:minimizing_cost} establishes optimality conditions.
	
	{\color{black}{\begin{proposition}\label{prop:minimizing_cost}
		Let $\mathfrak{C}(\rho)$ be a monotone increasing and convex function in $\rho\in[1,+\infty]$ and assume $ S_{r,p}(t)\approx Nf_{r,p}$. 
		Among all trajectories, such that $\widetilde I(t_{\max})=\sum_{r,p}rI_{r,p}(t_{\max})=\widetilde{I}^{\star}$, the one that minimizes the overall economic cost in $[0,t_{\max}]$, is  the one corresponding to:  $$\mathcal{R}^{\rho}(t)=1+\frac{1}{\gamma }\log \left(\widetilde{I}^{\star}/\widetilde{I}(0)\right)\quad\forall t\in[0,t_{\max}],$$ and, 
		$$\rho(t)= \frac{\sigma}{\gamma}\frac{\mathbb{E}[r^2]
		}{\mathbb{E}[r]  }\left[1+\frac{1}{\gamma T}\log \left(\widetilde{I}^{\star}/\widetilde{I}(0)\right)\right]^{-1}
		\quad \forall t\in[0, t_{\max}].$$
	\end{proposition}}}
	\begin{proof}
		Consider Eq. (\ref{eq:edgeinf}) and note that the unique solution of the associated Cauchy problem with initial condition   $\widetilde{I}(0)$ is given by:
		$
		\widetilde{I}(t)=\widetilde{I}(0) \exp\left(\gamma \int_0^t \left(\mathcal{R}^\rho(\tau)- 1 \right) \mathrm{d}\tau\right).
		$
		Imposing the constraint $\widetilde{I}(t_{\max})=\widetilde{I}^{\star}$ leads to:
		\begin{equation}\label{eq:constraint_r_eff}\frac{1}{t_{\max}} \int_0^{t_{\max}} {\mathcal{R}}^\rho(\tau)   \mathrm{d}\tau  =1+\frac{1}{\gamma \, t_{\max}}\log \left(\widetilde{I}^{\star}/\widetilde{I}(0)\right).
		\end{equation}
		Now, focusing on a generic  trajectory satisfying \eqref{eq:constraint_r_eff},  we have:
		$\frac{1}{t_{\max}} \int_0^{t_{\max}} \mathfrak{C}( \rho(\tau)) \mathrm{d} \tau=\frac{1}{t_{\max}} \int_0^{t_{\max}} \mathfrak{\chi}( \mathcal{R}^\rho(\tau)) \mathrm{d} \tau$
		with $\chi=\mathfrak{C}\circ \rho$, and $\rho(\mathcal{R}^\rho)=\frac{\sigma  \mathbb{E}[r^2] }{\gamma \mathbb{E}[r]{R}^ {\rho}(t)}$. Since $\mathfrak{C}$ is a monotonic increasing and convex
		function in $\rho\in[1,+\infty]$
		then $\chi $ is a convex function  over its domain,
		and by Jensen inequality, we conclude
		$\frac{1}{t_{\max}} \int_0^{t_{\max}} \mathfrak{\chi}( \mathcal{R}^\rho(\tau)) \mathrm{d} \tau
		\ge 
		\chi \left(\frac{1}{t_{\max}} \int_0^{t_{\max}}  \mathcal{R}^\rho(\tau)   \mathrm{d}\tau \right)
		$
		Therefore, from \eqref{eq:constraint_r_eff} the choice given by $\rho(t)= \frac{\sigma\mathbb{E}[r^2]}{\gamma\mathbb{E}[r]}\left[1+\frac{1}{\gamma t_{\max}}\log \left(\widetilde{I}^{\star}/\widetilde{I}(0)\right)\right]^{-1}$, 	$\forall t\in[0,t_{\max}]$
		minimizes the cost.
	\end{proof}
	Observe that  the economic cost  of previously defined optimal policy
	monotonically decreases while  increasing the target~$\widetilde{I}^{\star}$.
	\begin{corollary}
		Under the assumptions that $\mathfrak{C}(\rho)$ is a monotone increasing and convex function and $ S_{r,p}(t)\approx Nf_{r,p}$, 
		among all control strategies  that maintain the number of infected less than 
		or equal the initial value $\widetilde I(0)$, the overall economic cost is minimized when 
		$\mathcal{R}^\rho(t)$ is kept equal to 1. 
	\end{corollary}
	\begin{proof}
		From Proposition \ref{prop:minimizing_cost} we have that among all strategies 
		guaranteeing $\widetilde I( t_{\max})=\widetilde I(0)$, the one  forcing  $\mathcal{R}^{\rho}(t)=1$ is cost-optimal. The proof is completed by observing that  such a strategy guarantees  $\widetilde I(t) \le \widetilde I(0)$ for every $t\in [0,t_{\max}]$.
	\end{proof}
	\begin{remark}
		$\mathcal{R}^\rho(t)=1$  can be achieved by controlling the rate 
		of  new infections and maintaining it equal to the target   $\lambda_C=\gamma \widetilde I(0) \mathbb{E}[r]/\mathbb{E}[r^2]$. 
		The resulting control function is
		$
		\rho(t)={ {\lambda}_{\textcolor{black}{U}}(t-\varepsilon)}/{\lambda_C}=
		\lambda(t-\epsilon)\rho(t-\varepsilon)/\lambda_C, 
		$
		where $\varepsilon$ is an arbitrarily small positive constant.
	\end{remark}

	In conclusion,
	given an initial condition $\widetilde I(0)$, a maximum allowable number of infected contacts $\widetilde I^{\star}$ and a time horizon $t_{\max}$, if the goal is to keep $\widetilde I(t)\le \widetilde I^{\star} \quad \forall t\in [t^{\star},t_{\max})$, with $t^{\star}$ as small as possible,
	the following strategy appears to be  
	the natural answer:  if  $\widetilde I^{\star}>\widetilde I(0)$,  set $\mathcal{R}^{\rho}(t)=1+\frac{1}{\gamma t_{\max}}\log \left(\widetilde{I}^{\star}/\widetilde{I}(0)\right),\forall t\in[0,t_{\max}].$
	This strategy, indeed, minimizes  the economic cost in $[0,t_{\max}]$, among all strategies that guarantee  
	$\widetilde I(t) \le \widetilde I^{\star},\forall t\in[0,t_{\max}]$, (i.e., $t^{\star}=0$). 
	If, instead, $\widetilde I^{\star}<\widetilde I(0)$, we can not guarantee $t^{\star}=0$, and therefore  to minimize $t^{\star}$ 
	it is necessary to minimize  $\mathcal{R}^{\rho}(t)$   in $[0, t^{\star})$ and then to set  $\mathcal{R}^{\rho}(t)=1 ,\forall t\in [t^{\star},t_{\max}]$.
	Indeed, this is the strategy that minimizes the economic cost  in $[0,t_{\max}]$, among all strategies  minimizing $t^{\star}$.
	Previous arguments can be formalized in the following proposition.
	\begin{proposition}\label{prop:2}
		Given $\widetilde I(0)$, $\widetilde I^{\star}$ and $t_{\max}$,
		whenever our goal is to keep $\widetilde I(t)\le \widetilde I^{\star} , \forall t\in [t^{\star},t_{\max})$, with $t^{\star}$ as small as possible, the  strategy described above is cost-optimal.
	\end{proposition}	
					
\subsection{Rate Control with feedback delay}

	Policymakers cannot instantaneously react to changes in the rate of new infections due to several reasons:
	i) new infections are discovered by tests performed several days after infection, and high-risk individuals are more likely to undergo testing \cite{9444780}, ii) new regulations take time to be introduced and become effective, iii) decisions are based on trends obtained by averaging epidemiological curves, iv) the actual process of new infections in unknown (think of asymptomatic but infectious people). Consequently, the measured process is a delayed, noisy subsample of the actual process.
	Therefore, we consider the case in which the actual, instantaneous effectiveness of mobility restrictions, modeled by $\rho(t)$, is  given by:
	$\rho(t) = \max\left\{1,\frac{\int \mathfrak{f}_d(\tau) {\lambda}_{\textcolor{black}{U}}(t-\tau) \diff \tau}{\lambda_C}\right\}$ where  $\mathfrak{f}_d(\cdot)$ is a feedback delay distribution.  
					
	One of our main results is that the system becomes unstable if the feedback delay is too large with respect to $1/\gamma$ (the average time in the infectious state). To simplify the analytical derivations, we start with 
	the case of deterministic feedback delay of constant 
	duration $d$ (days). Then we extend the result to a delay distribution $\mathfrak{f}_d$.
	
	\begin{theorem}[Stability analysis with constant delay]
		Assume $$\rho(t) = \max\left\{1,\frac{ {\lambda}_{{U}}(t-d)}{\lambda_C}\right\}=\max\left\{1,\frac{ {\lambda}(t-d)\rho(t-d)}{\lambda_C}\right\}$$ 
		and $S_{r,p}(t) \approx N f_{r,p}$.
		If the delay $d<\frac{\pi}{2}\gamma$ then the system is locally stable, otherwise 
		the system is unstable. 
	\end{theorem}
	{\color{black}{\begin{proof}
		Since  \textcolor{black}{under the assumption $S_{r,p}(t) \approx N f_{r,p}$,} 
		the equation governing the evolution of the number of infected edges 
		under delayed rate control becomes: 
		\begin{equation}\label{eq:fixdelay}
			\dot{\widetilde{I}}(t) = \frac{\widetilde{I}(t)}{\widetilde{I}(t-d)} 
			\lambda_C \frac{\EX[r^2]}{\EX[r]} - \gamma \widetilde{I}(t)
		\end{equation}
		
		System stability can be analyzed by considering small perturbations around 
		the equilibrium point $\widetilde{I}^* = \frac{\lambda_C}{\gamma} \frac{\EX[r^2]}{\EX[r]}$:
		$ \widetilde{I}(t) = \widetilde{I}^* + \eta(t)$, with $  \eta(t) \ll \widetilde{I}^* $.
		Exploiting the approximation $\frac{1}{1+x} \sim 1-x$, when $x\approx0$, 
		from \equaref{fixdelay} we obtain: 
		\begin{align*}
			\dot{\widetilde{I}}(t) = 
			\gamma \widetilde{I}^*  \frac{1 + \frac{\eta(t)}{\widetilde{I}^*}}
			{1 + \frac{\eta(t-d)}{\widetilde{I}^*}} - \gamma (\widetilde{I}^* + \eta(t))
			& \approx  - \gamma \eta(t-d)  
		\end{align*}
		where we have discarded the second-order term $\eta(t)\eta(t-d)$.
		We end up with the simple differential equation with delay:
		\begin{equation}\label{eq:eta}
			{\dot \eta(t)}= -\gamma\,\eta(t-d)
		\end{equation}
		Taking the Laplace transform $ \mathcal{L}\{\eta(t)\}$ we obtain 
		$\mathcal{L}\{ \eta(t) \}  = \frac{\eta(0)}{s + \gamma e^{-s d}}$.
		Equation \equaref{eta} admits solutions of the form 
		$\eta(t) = A e^{b t} \cos(\omega t + \theta)$
		under the conditions:
		\begin{eqnarray}\label{eq:bomega}
			\begin{cases}
				b = -\gamma e^{-b d} \cos(\omega d) \\
				\omega = \gamma e^{-b d} \sin(\omega d) 
			\end{cases}
		\end{eqnarray}    
		While $A$ and $\theta$ can take any value, i.e., can be used to match desired values of
		$\eta(0)$ and $\eta'(0)$, $b$ and $\omega$ are uniquely determined by the feedback delay $d$.
		Besides the trivial solution $b = \omega = 0$, there exists a stationary solution
		$b = 0$, $\omega = \gamma$ for the special case $d = \frac{\pi}{2 \gamma}$.
		If $d < \frac{\pi}{2 \gamma}$, from the first constraint we have that $b < 0$, corresponding
		to dumped oscillations.
		For $\frac{\pi}{2 \gamma} < d < \frac{3 \pi}{2 \gamma}$, we have instead amplifying oscillations
		($b > 0$). Therefore, $d = \frac{\pi}{2 \gamma}$ is the critical value for stability.
	\end{proof}
	}}
					
	The analysis can be extended to a delay distribution $\mathfrak{f}_d$.
	\begin{theorem}[Stability analysis with delay distribution]
		Assume that 
		$
		\rho(t) = \max\left\{1,{\int \mathfrak{f}_d (\tau){\lambda}_{\textcolor{black}{U}}(t-\tau) \diff \tau}/{\lambda_C}\right\}
		$
		and $S_{r,p}(t) \approx N f_{r,p}$. 
		Let $ \mathcal {Z}=\{z\in\mathbb{C}:  z + \gamma \,\mathfrak{F}_d(z)=0 \}$, where
		$\mathfrak{F}_d(z)$ is the Laplace transform of the delay distribution.
		Then, if $\mathrm{Re}(z)<0$  $\forall z\in \mathcal {Z}$,  the system is locally stable.
	\end{theorem}
	{\color{black}{\begin{proof}
				Repeating the same approximations as before for small variations around 
				the equilibrium $\widetilde{I}^*$, we
				obtain the differential equation with delay distribution:
				\begin{equation}\label{eq:etafd}
					{\dot \eta(t)} = -\gamma \int \mathfrak{f}_d(\tau) \eta(t-\tau) \diff \tau
				\end{equation}
				Taking the Laplace transform, we get
				$H(s) = {\eta(0)}/{(s + \gamma \, \mathfrak{F}_d(s))}$.
				Note that when $\mathfrak{f}_d(\tau) = \delta(\tau -d)$,
				we obtain the case with constant delay. 
				We evince that we need the set of zeros $\mathcal{Z}=\{z\in\mathbb{C}:z + \gamma \, \mathfrak{F}_d(z)=0\}$ to lie in the left half-plane to ensure stability.
			\end{proof}
	}}
	In the following corollaries, whose proof is given in \cite{SM}	we explore two interesting cases of feedback delay distributions. 
	\begin{corollary}[Exponential delay distribution]
		If $\mathfrak{f}_d(\tau) = u(\tau) \delta e^{-\delta (\tau)}$, then 
		the system is always (locally) stable. 
	\end{corollary}
					
	\begin{corollary}[Shifted exponential delay distribution]
		Let $\mathfrak{f}_d(\tau) = u(\tau-d) \delta e^{-\delta (\tau-d)}$. 
		For any given $\delta > 0$, there exists a critical delay
		$d^* = \frac{1}{\gamma} f(\delta)$, such that the system
		is (locally) stable if $d < d^*$, otherwise the system is unstable.
		As $\delta$ grows from 0 to $\infty$, $d^*$ grows
		from $1/\gamma$ to $\pi/(2 \gamma)$.
	\end{corollary}
	%
		
	The shifted exponential distribution can represent a system where: i) an exponentially weighted moving average (with parameter $\delta$) is used to estimate the current trend of the epidemiological curve, ii) some fixed delay $d$ is introduced before the control becomes effective.
	Our results suggest that system stability is crucially tied (by a factor between 1 and $\pi/2$ that  depends on $\delta$) to the mean sojourn time $1/\gamma$ in the infectious state. If $d$ is too large with respect to $1/\gamma$, the control based on the force of infection is prone to instability.
	
	\medskip
	
	In a finite population system, as time goes on, we can no longer assume that $S_{r,p}(t) \approx N f_{r,p}$, since the number of initially susceptible individuals is progressively reduced by the number of people who get infected (see \equaref{dyn_sys}). 
	Moreover, $S_{r,p}(t)$ can vary because of vaccinations and the finite duration of immunity.
	Nevertheless, we can still apply the above results by resorting to a {\it{time-scale separation approach}}, i.e., by assuming that $S_{r,p}(t)$, though not equal to $N f_{r,p}$, are almost constant at the time scale over which we analyze stability.
	
	Indeed, recall from Remark \ref{oss:edge_perspective} that the evolution of the total number of infected edges can be written as:   
	\begin{equation}\label{eq:Itilde}
		\dot{\widetilde{I}}(t) = \gamma \left(\frac{\mathcal{R}(t)}{\rho(t)} -1 \right) \widetilde{I}(t)
	\end{equation}
	where $\mathcal{R}(t) = \frac{\sigma \sum_{r,p} r^2 S_{r,p}(t)}{\gamma \sum_{r,p} r N f_{r,p}}$ \textcolor{black}{is the basic reproduction in the  general case.} 
	This equation is formally identical to (\ref{eq:edgeinf}) upon substituting $\mathcal{R}_0$ with $\mathcal{R}(t)$.
	Since our stability results do not depend on $\mathcal{R}_0$, they apply also to a system in which $\mathcal{R}(t)$ can be considered approximately constant at the time scale at which we analyze the system stability (i.e., time scale of $1/\gamma$).
	
\section{Control on hospitalizations and Intensive therapy occupancy}\label{sec:controlHT}

	Recall that, according to the HT strategy, the control variable $\rho(t)$ is directly related to the instantaneous numbers $H(t)$ and $T(t)$ of patients who are currently hospitalized or under intensive treatment, respectively.
	Many countries have widely adopted this strategy, being particularly simple to implement.

	We assume that Hospitals and ICUs have a maximum capacity {$\widehat H$ and $\widehat T$}, correspondingly.
	A maximum level of restrictions $\rho_{\max}$ is applied whenever either $H(t)$ exceeds $H_{\max}$ (with $ H_{\max} \le \widehat H $),  or $T(t)$ exceeds  $T_{\max}$   ($T_{\max}  \le \widehat T $). 
	When  $H(t)<H_{\max}$ and $T(t)<T_{\max}$,
	we assume that two control functions
	$\rho_H: \mathbb{R}^+\rightarrow[1,\infty)$ and $\rho_T: \mathbb{R}^+\rightarrow[1,\infty)$
	provide two different levels of restrictions, the larger (i.e. stricter) of which is 
	actually applied: $\rho: = \max\{\rho_H\circ H,\rho_T\circ T\}$.
	{\begin{ass}\label{ass:rho}
			Let $\rho_H \in C^1[0,H_{\max}]$,  $\rho_T\in C^1[0,T_{\max}]$  such that $\rho_H(0)=\rho_T(0)=1$,  $\rho_H(H_{\max})=\rho_T(T_{\max})=\rho_{\max}$, with $\inf_{x\in(0,H_{\max} ) } \dot\rho_H(x)>0$ and  $\inf_{x\in(0,T_{\max})} \dot\rho_T(x)>0$.
	\end{ass}}
	
	To analyze the system stability under the above type of control,
	we first  assume 
	$S_{r,p}(t) \approx N f_{r,p}$. 
	We will later extend the analysis to the general case
	through a time-scale separation approach.
	{{
	Under the assumption $S_{r,p}(t) \approx N f_{r,p}$ we have that the total number of infected \lq edges' is governed by \eqref{eq:edgeinf}.
	\begin{proposition}[Stationary solutions]\label{prop:equil}
	Under the assumption $S_{r,p}(t) \approx N f_{r,p}$ and Assumption \ref{ass:rho} the stationary solutions satisfy:
	\begin{align}\label{eq:equil}
		H^*  = \frac{\gamma}{\phi} \widetilde{I}^* \frac{\EX[r \, p_{r,p}^{IH}]}{\EX[r^2]}, \quad
		T^*  = \frac{\gamma}{\tau} \widetilde{I}^* \frac{\EX[r \, p_{r,p}^{IH} \, p_{r,p}^{HT}]}{\EX[r^2]} .
	\end{align}
	\end{proposition}
	{\color{black}{\begin{proof}
			From the definition we have 
			$I_{r,p}(t) = \widetilde{I}(t) \frac{r f_{r,p}}{\EX[r^2]},$ $ I(t) = \widetilde{I}(t) \frac{\EX[r]}{\EX[r^2]}.$
			It should be noted that at equilibrium  necessarily $\rho^*(t)={{\mathcal{R}}}_0$ for all $t$ and, by monotonicity of $\rho_H$ and $\rho_T$, we have one of the following cases:
			\begin{itemize} \item $H^*=\rho_H^{-1}({\mathcal{R}}_0), $ and $
				T^*\le \rho_T^{-1}({\mathcal{R}}_0)$;
				\item $T^*= \rho_T^{-1}({\mathcal{R}}_0),\;  H^*\le \rho_H^{-1}({\mathcal{R}}_0)$. 
			\end{itemize} 
			Hence,
			$$
			\widetilde{I}^*=\min\left( \rho_H^{-1}({\mathcal{R}}_0)  \frac{\phi}{\gamma} \frac{\EX[r^2]}{\EX[r \, p_{r,p}^{IH}]},
			\rho_T^{-1}({\mathcal{R}}_0) \frac{\tau}{\gamma} \frac{\EX[r^2]}{\EX[r \, p_{r,p}^{IH} \, p_{r,p}^{HT}]} \right). 
			$$
			Now, from (\ref{eq:dyn_sys}), we obtain detailed equilibrium points: 
			\begin{align*}
				I_{r,p}^*= \widetilde{I}^* \frac{r f_{r,p}}{\EX[r^2]},H_{r,p}^*  = \frac{\gamma}{\phi} I_{r,p}^* \,  p_{r,p}^{IH},
				T_{r,p}^*  = \frac{\phi}{\tau} H_{r,p}^* \, p_{r,p}^{HT} 
			\end{align*} 
			Therefore, summing over $(r,p)$, we get corresponding equilibria for the total 
			number of people hospitalized or under intensive therapy as given by \eqref{eq:equil}.
	\end{proof}}
	\begin{theorem}[Stability analysis]\label{thm:stabilityHT}
		Let $\rho_H$ and $\rho_T$ satisfy Assumption \ref{ass:rho} and $H^*$ and $T^*$ be stationary solutions as given in Proposition \ref{prop:equil}. If at least one of the following conditions is satisfied:
		\begin{itemize}
			\item $\rho_H(H^*) > \rho_T(T^*)$
			\item $\rho_T(T^*) \ge  \rho_H(H^*)$ and $\phi+ \tau \geq \frac{T^* \dot\rho_T(T^*) \gamma}{\mathcal{R}_0}$
		\end{itemize}
		then the system is locally stable.
	\end{theorem}
	}}
	{\color{black}{
	\begin{proof}
		Let us consider small perturbations around the equilibrium point $\widetilde{I}^*$: 
		$ \widetilde{I}(t) = \widetilde{I}^* + \widetilde{\eta}(t) $ with $  \widetilde{\eta}(t) \ll \widetilde{I}^* $.
		
		We will assume that $0 < H^* < H_{\max}$, and $0 < T^* < T_{\max}$. From Assumption \ref{ass:rho}, by denoting 
		with $\alpha_H^*=\dot{\rho}(H^*)$ and $\alpha_T^*=\dot{\rho}(T^*)$  we have the following cases.
		\begin{enumerate}
			\item{ If $\rho_H(H^*) > \rho_T(T^*)$ by continuity we get that $\rho(t) = \rho_H(H(t)) > \rho_T(T(t))$
			}
			and assuming initial conditions $H(0) = H^*$, $T(0) = T^*$,
			after some algebra we get the Laplace transform of $\eta(t)$:
			{{\begin{equation}\label{eq:iotaH}
						\mathcal{L}\{\eta(t)\} = \frac{\eta(0) (s+\phi)}{s(s+\phi) + 
							\frac{H^* \alpha_H^* \phi \gamma}{\mathcal{R}_0}} 
			\end{equation}}}
			In this case, the system is always stable for any value of parameters {$\phi$}, $\gamma$,$\mathcal{R}_0$,
			since the real part of the poles of (\ref{eq:iotaH}) is always negative.
			As we increase the amplitude of coefficient $\frac{H^* \alpha_H^* \phi \gamma}{R_0}$, the real part of the dominating pole moves from 0 to $-\phi$.  
			
			\item If  $\rho_T(T^*) > \rho_H(H^*)$ then, by continuity, we have $\rho(t) = \rho_T(T(t)) > \rho_H(H(t))$ and, by first order analysis and computing the Laplace transform, we get
			$$ \mathcal{L}\{\eta(t)\} = \frac{\eta(0) (s+\phi)(s+\tau)}{s(s+\phi)(s+\tau)+ 
				\frac{T^* \alpha_T^* \tau \phi \gamma}{\mathcal{R}_0}} $$
			The system  may be unstable since we obtain in the denominator a third-order
			equation whose complex solutions can fall in the positive half-plane. 
			In particular, the system is stable when:
			\begin{equation}\label{eq:stableregion}
				\phi + \tau \geq \frac{T^* \alpha_T^* \gamma}{\mathcal{R}_0}
			\end{equation}
			while it becomes unstable otherwise.
			Indeed, pure imaginary solutions $s = i \omega$ are roots of
			the above third order equation when $\omega = \sqrt{\tau \phi}$, while  
			relation (\ref{eq:stableregion}) is satisfied with equality
		\end{enumerate}
	\end{proof}}}
	Theorem \ref{thm:stabilityHT} provides conditions guaranteeing the local stability of the system. 
	
	In particular, it is worth remarking that once  $H_{\max}< \widehat{H}$ has been fixed, condition $\rho_H(H^*) > \rho_T(T^*)$ can always be achieved by arranging a  sufficiently large number of available intensive therapy facilities.
	Indeed, even when ${\mathcal{R}}_0$ is not perfectly known, it is sufficient to guarantee:
	$$ \rho_H^{-1}(y) \frac{\phi}{\gamma} \frac{\EX[r^2]}{\EX[r \, p_{r,p}^{IH}]}<
	\rho_T^{-1}(y) \frac{\tau}{\gamma} \frac{\EX[r^2]}{\EX[r \, p_{r,p}^{IH} \, p_{r,p}^{HT}]} $$  for
	every $ \rho_{\min}<y<\rho_{\max}$, i.e.
	$ \frac{\rho_T^{-1}(y)   }{ \rho_H^{-1}(y)}> \frac{\phi}{\tau} 
	\frac{\EX[r \, p_{r,p}^{IH} \, p_{r,p}^{HT} ]}{\EX[r \, p_{r,p}^{IH}]}.$   Observe that the above constraint can be met if 
	\begin{equation}
		T_{\max} > \frac{\phi}{\tau} 
		\frac{\EX[r \, p_{r,p}^{IH} \, p_{r,p}^{HT} ]}{\EX[r \, p_{r,p}^{IH}]}{H_{\max}} \label{condICU}
	\end{equation}  by adopting  controllers  that satisfy the relationship: $ \rho_H(xH_{\max})  \ge \rho_T(xT_{\max})\quad \forall \; 0\le x \le 1$.
	
	When the number of intensive therapies is, instead, under-dimensioned,  we have $\rho_H(H^*) > \rho_T(T^*)$, and the system stability essentially depends on the average time spent in hospitals and ICU, through the sum $\phi + \tau$ of transitions rates out of compartments $H$, $T$ (both are equally important). 

	Assuming that $S_{r,p}(t)$ are almost constant on the time scale over which stability is studied, the analysis can be extended by replacing the basic reproduction number ${\mathcal{R}}_0$ with the effective reproduction number ${\mathcal{R}}(t)$.Indeed, by doing so, the evolution of the total number of infected edges \equaref{Itilde} becomes formally identical to (\ref{eq:edgeinf}).

{\color{black}
\section{A comparative analysis with optimal control and homogeneous models}\label{sec:prelimres}
	In this section, we perform a comparative analysis of the proposed model against optimal control and homogeneous models to assess its effectiveness and advantages in addressing the research problem.
	
\subsection{Optimal control versus Rate/HT Control} 
	\begin{figure*}[t]
		\begin{center}
			\begin{tabular}{ccc}
				\includegraphics[width=0.64\columnwidth]{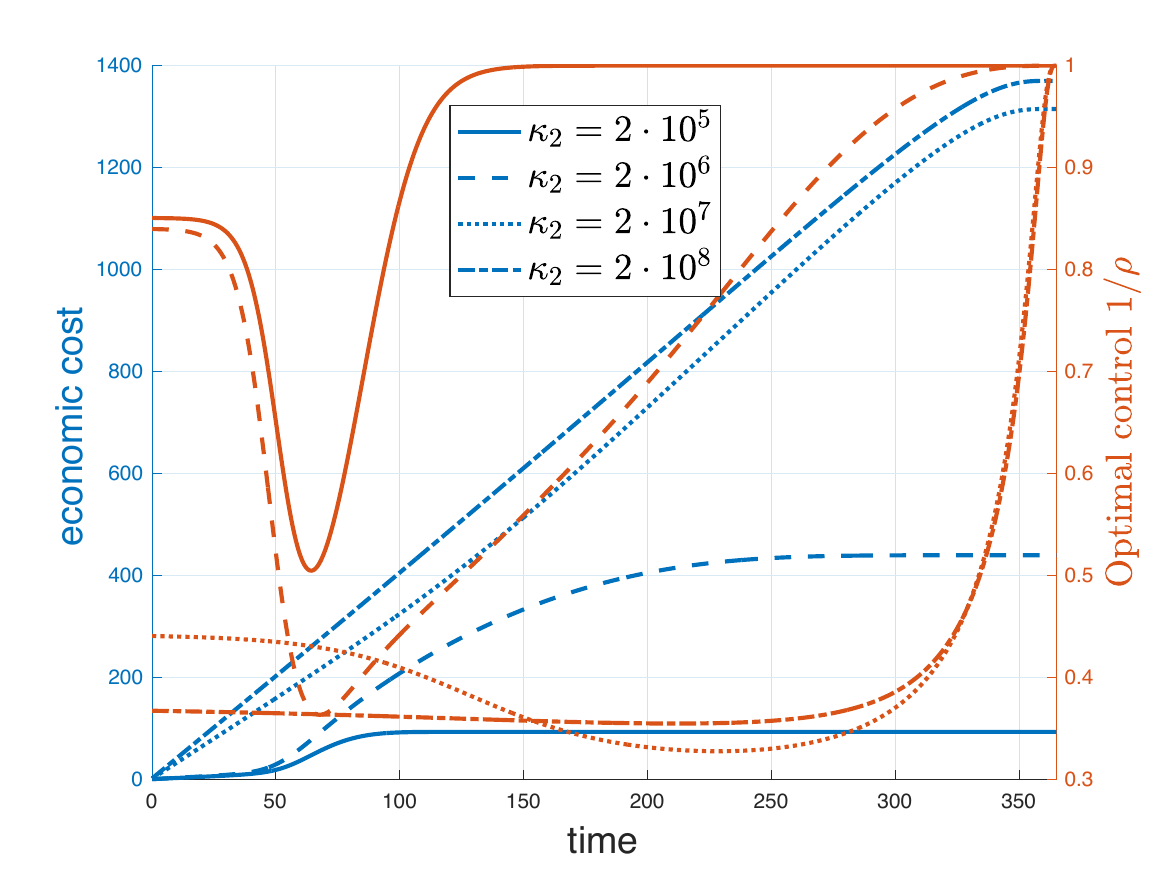}&
				\includegraphics[width=0.64\columnwidth]{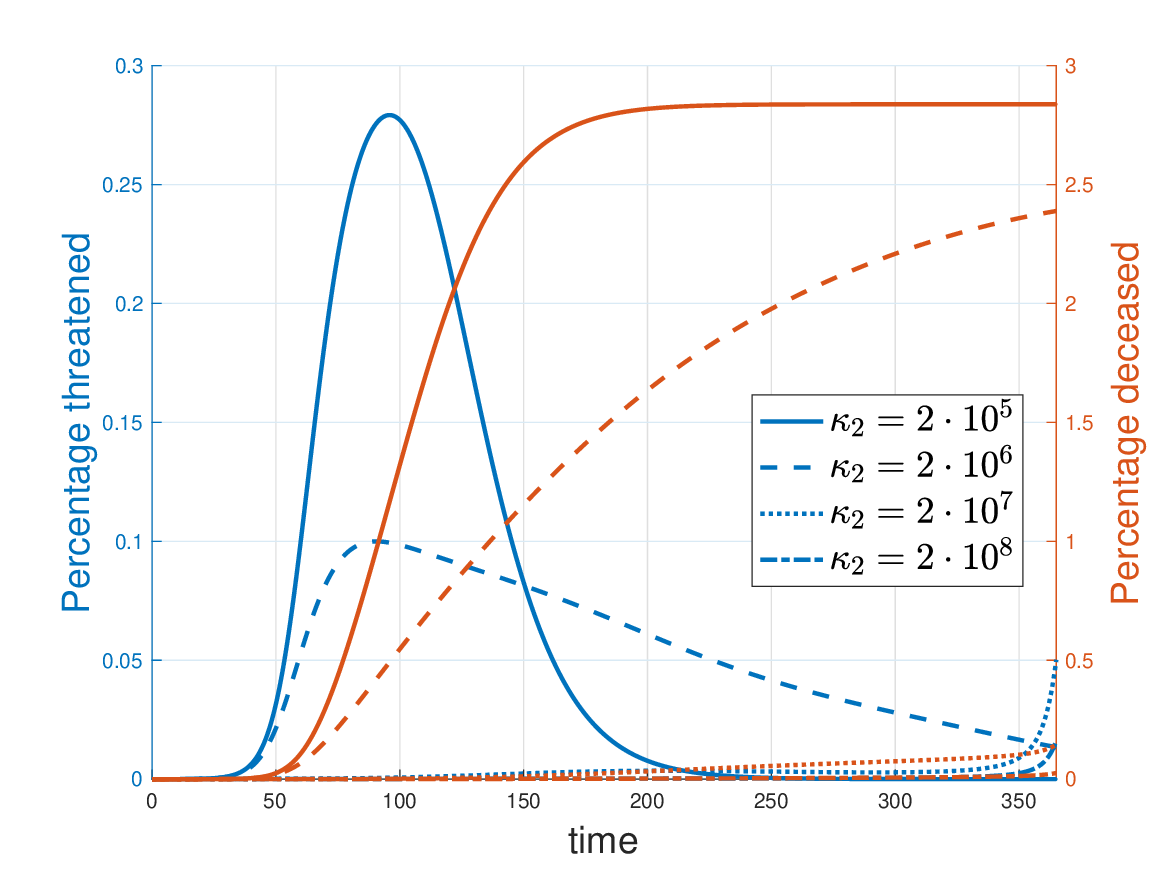}&
				\includegraphics[width=0.64\columnwidth]{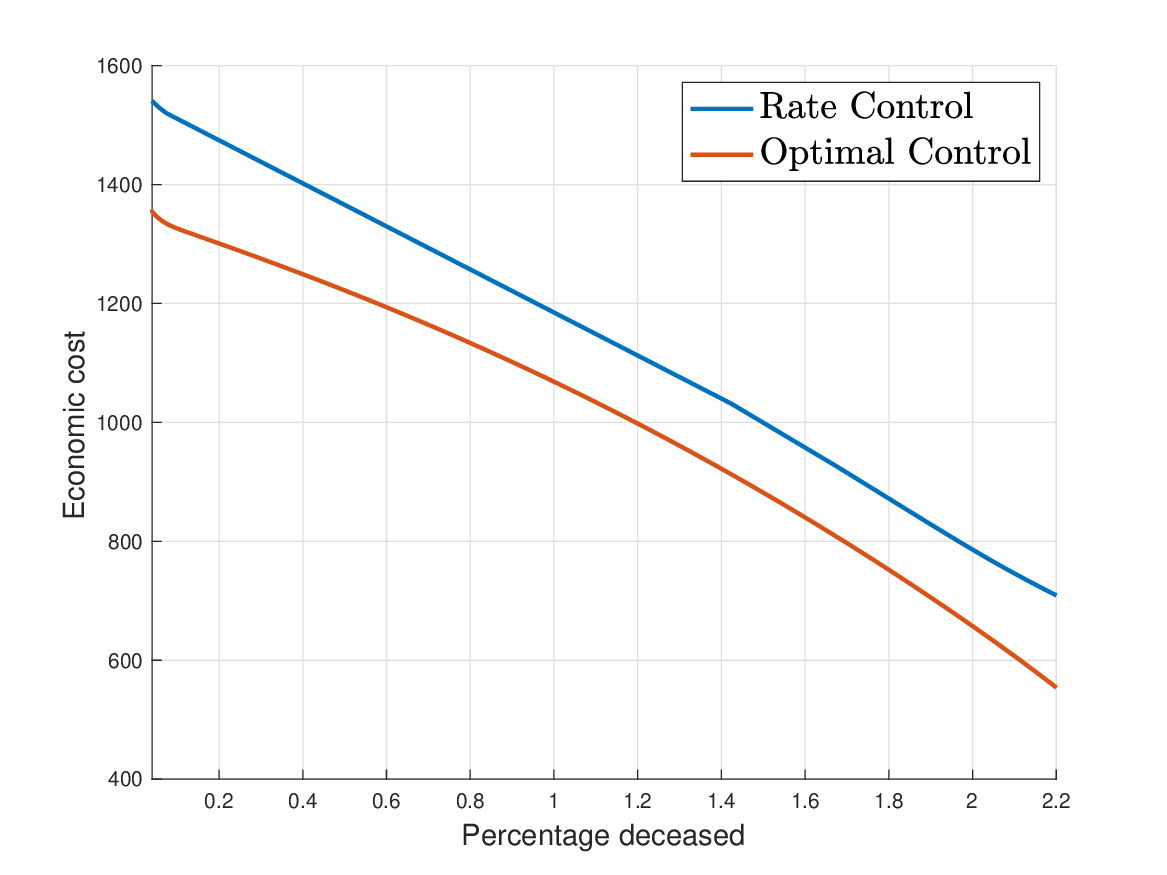}\\
				(a)&(b)&(c)
			\end{tabular}
		\end{center}
		\caption{(a) Economic cost and control effort via optimal control (b) Percentage of threatened and deceased individuals via optimal control implementation. (c) Comparison of optimal control and Rate control  strategies}\label{fig_optimal_control_a}
	\end{figure*}
	\begin{figure*}[h!]\begin{center}
			\begin{tabular}{ccc}
				\includegraphics[width=0.64\columnwidth]{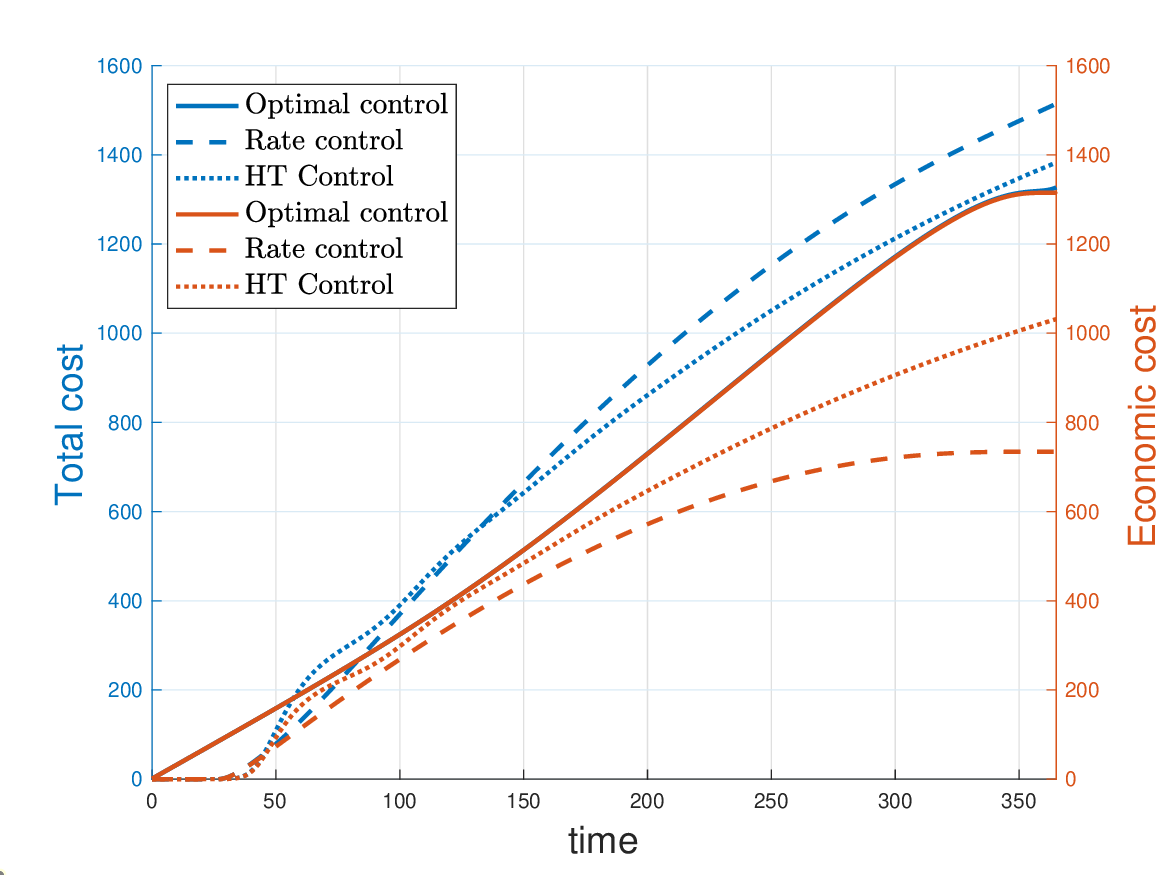}&
				\includegraphics[width=0.64\columnwidth]{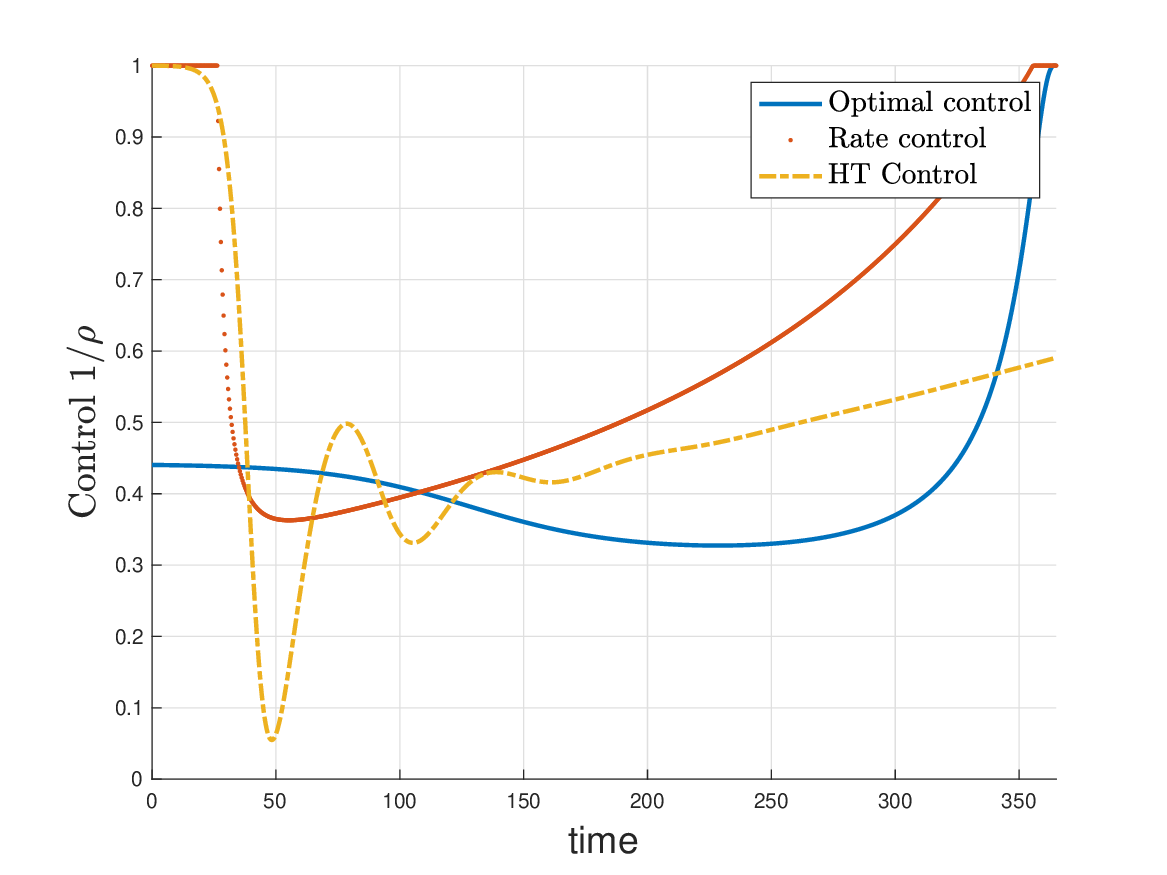}&
				\includegraphics[width=0.64 \columnwidth]{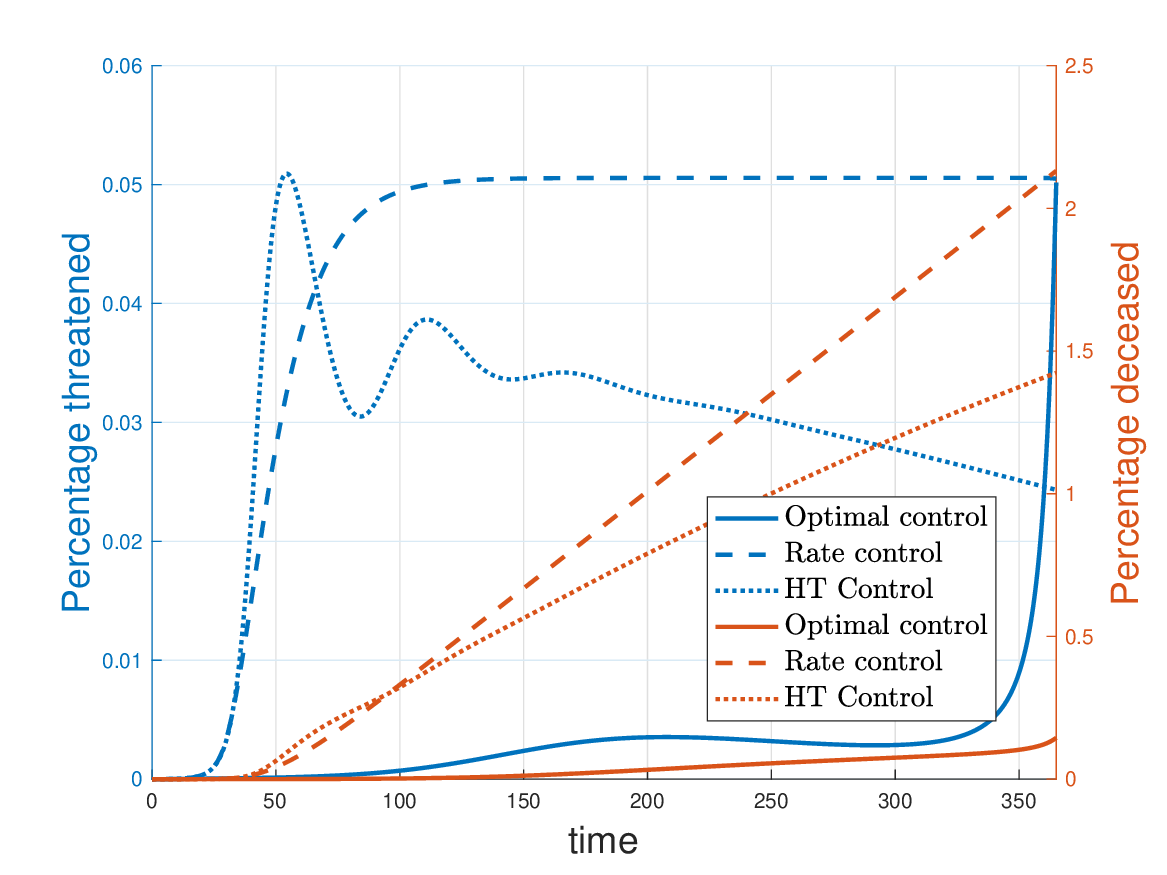}\\
				(a)&(b)&(c)
			\end{tabular}
		\end{center}
		\caption{(a) Comparison of different control strategies: (a) Overall costs and economic costs; (b) Control function. (c) Percentage of threatened and deceased individuals.}\label{fig:add_a}
	\end{figure*}
		
	In this section, we compare the Rate and HT controllers with optimal control during the first phase of the pandemic in which vaccines are still unavailable. All experiments refer to the single-class version of the model in \eqref{eq:dyn_sys}.
		
	In \cite{Kasis2022}, various government intervention strategies are compared against a specific percentage of the deceased population while employing different control policies. We replicate similar experiments and present the numerical solutions obtained via the optimal control approach within a time horizon $t_{\max}=365$ days. We fix  $\kappa_1=200$, $\kappa_3=20$, and $\zeta=4$, and we let $\kappa_2$ to vary from $10^5$ to $10^8$.

	The considered economic cost corresponds to the teal curve depicted in Figure \ref{Fig:1}. 
	\textcolor{black}{Such cost has been chosen non-convex, on purpose,  to put the Rate control strategy in the most challenging conditions (indeed  Proposition \ref{prop:minimizing_cost} and \ref{prop:2} do not hold). 
	Observe that, with our choice of parameters, the first term of the objective function in \eqref{eq:obj} is typically small with respect to the third, and therefore the choice of $\kappa_2$ becomes fundamental to determine the proper trade-off between economic cost and death/ICU occupancy, where the last two metrics are highly correlated, since by tightly controlling ICU occupancy, we exert tight control on the deaths and vice versa. }

	As discussed in the Appendix, the average sojourn time in state I has been set equal to 8 days and that in states H and T to 16 days. Therefore $\gamma=1/8$, $\phi=\tau=1/16$. The transition probabilities between compartments I, H, T, and D, satisfying constraint (2), are set for simplicity as follows: $p^{IH} = p^{HT} = \hat{p}^{TD}= p^{1/3}$ with overall mortality rate $p=0.01$. The assumed value for the basic reproduction number has been fixed to $\mathcal{R}_0=3$, as reported in \cite{yuan2020monitoring}, and the healthcare capacity parameter to $3.33\cdot10^{-3}$. 
	
	In Figure \ref{fig_optimal_control_a}(a), the economic cost  (in blue) and the control (in red) are shown as a function of time.
	For lower values of the parameter ($\kappa_2\in[10^5,10^6]$), the control measures are moderate and remain relatively constant for a brief period of approximately 50 days. After this initial phase, the control is tighter, reaching its maximum level of restriction. During this period, stringent measures are implemented to contain the epidemic effectively. Subsequently, as the situation improves or specific goals are achieved, the control is gradually relaxed, allowing for a more lenient approach to managing the epidemic. This sequential pattern of moderate-tightened-relaxed control measures aims to strike a balance between mitigating the spread of the disease and minimizing the socioeconomic impact on the population.
	The stringent initial intervention effectively disrupts the early exponential growth of the epidemic, leading to a dampened peak number of infections. 
	\textcolor{black}{In both previous cases at the end of the observation window, i.e., for $t \approx t_{\max}$, the population reaches herd immunity.}
	As we increase the value of $\kappa_2$, as expected, the control becomes more stringent and is kept constant for most of the time (around 350 days). At the same time, the healthcare system experiences less stress, and the number of deaths decreases at the expense of higher economic costs. Notably, when setting $\kappa_2$ above $2\cdot10^7$, we can confidently guarantee that the peak number of patients requiring intensive care remains below around $30k$ (see Figure \ref{fig_optimal_control_a}(b) where blue curves refer to ICU occupancy and red curves to cumulative deaths). This behavior highlights the importance of appropriately calibrating control parameters to achieve optimal outcomes in managing the epidemic and preventing overwhelming pressure on the healthcare infrastructure. 
	\textcolor{black}{ Observe that at the end of the observation window, i.e., for  $t \approx t_{\max}$, the control is always completely released, i.e.,  $\rho(t_{\max})=1$. 
	\textcolor{black}{This effect is a by-product of the optimal control approach, which does not account for what happens when $t>t_{\max}$.}
	Indeed, as $t$ approaches $t_{\max}$,  releasing the control leads to an instantaneous reduction of the economic cost, while, due to the delay, the resulting increase in ICUs and deaths is negligible (as it will take place after $t_{\max}$ ).}

	{\color{black}{The analysis in Figure \ref{fig_optimal_control_a}(c) highlights the trade-offs between economic cost and human lives achieved by optimal control and  Rate control, respectively. The curves have been obtained by varying parameter $\kappa_2\in[10^5,10^8]$ for optimal control, $\lambda\in[1000,700000]$ for Rate control. \textcolor{black}{ We have disregarded the healthcare stress cost, using the total number of deaths as a proxy of it}.
	It should be noted that the optimal control strategy proves to be the most effective, outperforming the rate control strategy. However, if we fix the number of deaths, for example, to 0.2\%  (i.e., 100000 deaths), the rate control strategy exhibits only a slightly worse economic cost. The difference between the economic cost curves of the two strategies is not substantial, with a modest 7\% increase obtained by the rate control strategy.
	Despite the increase in economic costs resulting from the rate control strategy, the difference is relatively small, indicating that both strategies remain competitive in managing the epidemic. 
				
	It is worth remarking that  \textcolor{black}{approximately the same value  of the  overall objective  function}  in \eqref{eq:obj}, which takes into account economic cost, deaths, and healthcare stress in the optimal control strategy, can be achieved through fairly different approaches. }
	{\color{black}  Figure \ref{fig:add_a} provides a comparison among the optimal control strategy with 
		$\kappa_2=2\cdot 10^7$,  the Rate control  with  $\lambda=120k$, and the linear HT  control
		with 
		$\rho_H(H)=\min\left(15, \frac{H_{0}}{H_{0}-H}\right)$ and   $\rho_T(T)=\min\left(15, \frac{T_{0}}{T_{0}-T}\right)$, with and  $H_0=480k$ and $T_0=300k$. 
	While the overall cost for the three strategies is approximately the same, the different components of the cost are significantly different.}
	For what concerns the economic cost, the optimal control strategy appears to be the least favorable, resulting in the highest economic burden compared to the other strategies (the economic cost for the optimal control strategy is hardly distinguishable from the overall cost).
	On the other hand, the rate control strategy is the most efficient in minimizing economic costs, offering a more economically sustainable approach. The HT control strategy falls in an intermediate position, achieving a balance between cost-effectiveness and epidemic management. Regarding the number of deaths, the optimal control strategy demonstrates its strength, resulting in the lowest fatality rate among the three strategies. It effectively minimizes the loss of life during the epidemic. Conversely, the rate control strategy shows the highest number of deaths, indicating that this approach is less effective in preventing fatalities. The HT control strategy lies in between, offering an intermediate level of protection against the loss of life compared to the other two strategies.
	\textcolor{black}{In conclusion, adopting the overall cost as the unique driver for the choice of $\rho(t)$  turns out to be not particularly appealing to decision-makers because it does not allow them to exert direct control on the different components of the cost.}			
	}}

\subsection{Motivating the heterogeneity} 
	In this section, we question the importance of introducing in the model a stratification based on distribution $f_{r,p}$, given that, at least when $S_{r,p}(t) \approx N f_{r,p}$, our multi-class (stratified) model is equivalent to a single-class (non-stratified) model with a transmission rate modified by a factor $\beta=\frac{\sum_{r,p} r^2 f_{r,p}}{\sum_{r,p} r f_{r,p}}$.
	
	We now show that, besides being necessary to evaluate prioritized vaccination strategies, our stratification is fundamental also to compute the cost of pure non-pharmaceutical interventions. To this purpose, we consider a simple SIR model under perfect control of new infections, i.e., where $\rho(t) = \max\left\{1,\frac{\lambda_U(t-\varepsilon)}{\lambda}\right\}$ with $\varepsilon=1/100$ day.
	
	\begin{figure} [htb]
		\centering
		\includegraphics[width=.7\linewidth]{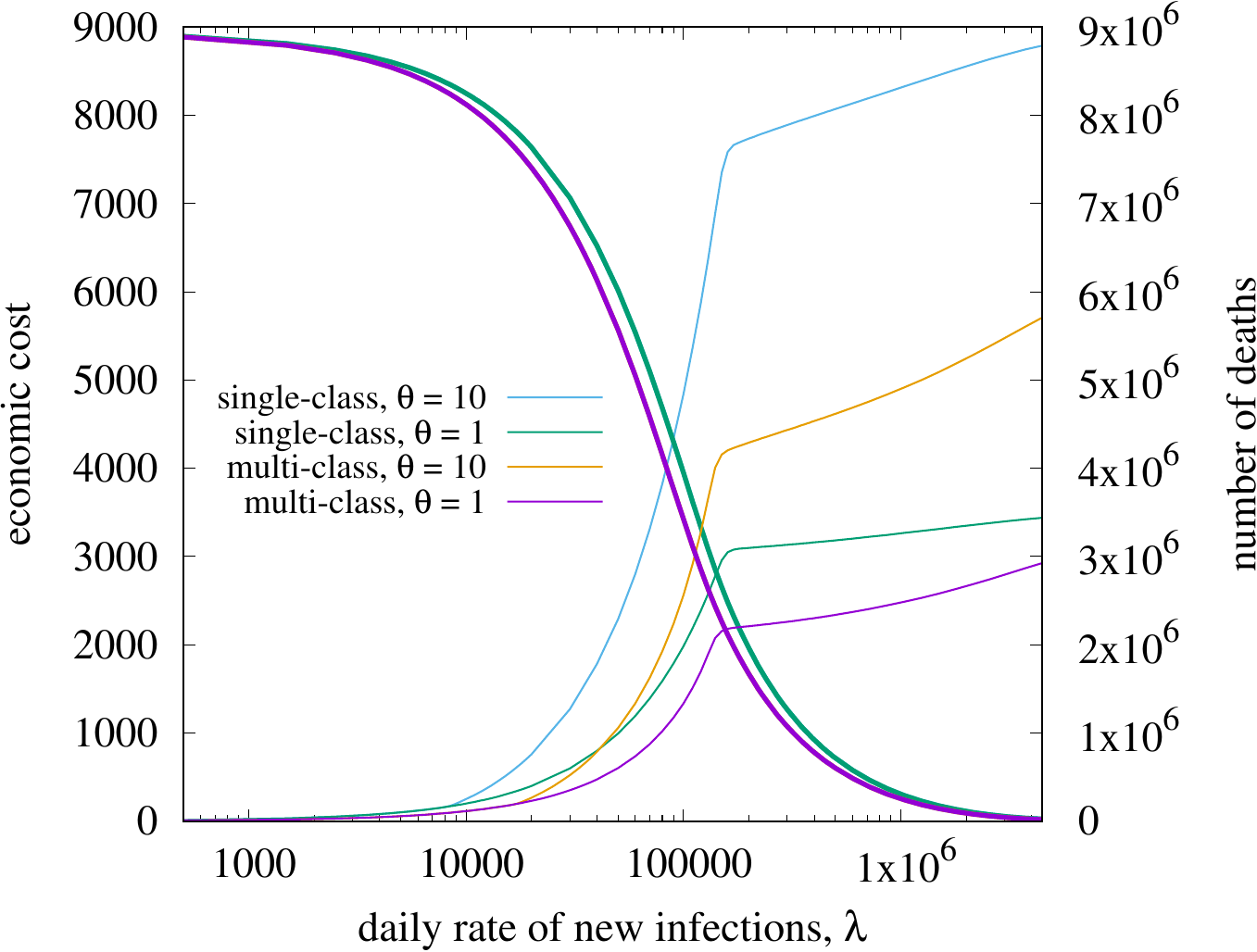}
		\caption{Comparison between single-class and multi-class models,
			and impact of parameter $\theta$, on economic and social costs,
			as a function of controlled rate $\lambda$ of new infections (log $x$ scale).} 
		\label{fig2}
	\end{figure}
	
	\begin{figure} [htb]
		\centering
		\includegraphics[width=.7\linewidth]{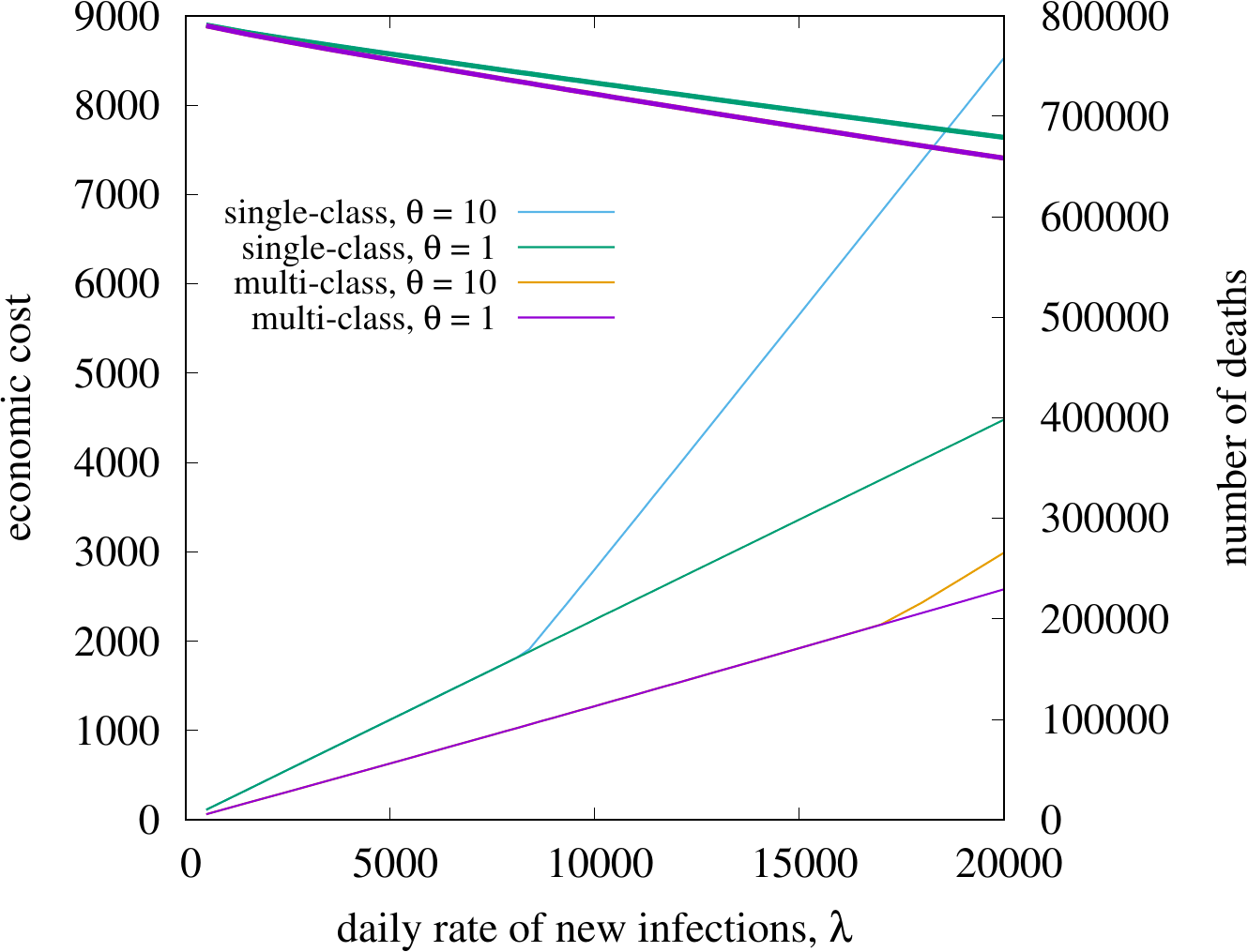}
		\caption{Same results as in Fig. \ref{fig2}, zooming in on the more 
			reasonable control regime of small $\lambda$.} 
		\label{fig3}
	\end{figure}
	
	In Fig. \ref{fig2}, we compare the economic and social costs derived by our model with those computed by a single-class model in which the contact rate and mortality of all individuals are set equal to mean values $\EX[r]$ and $\EX[p]$, respectively.
	In this experiment, the time horizon is $t_{\max} = 1$ year.
	
	We also compare the case in which saturation of intensive therapy capacity does not affect mortality ($\theta = 1$) and the case where mortality is severely increased by ICU saturation ($\theta = 10$).
	Fig. \ref{fig2} uses a log horizontal scale to encompass also large values of $\lambda$, corresponding to an almost uncontrolled epidemics, i.e., the attempt to quickly reach the naturally-acquired herd immunity (which occurs at the \lq knees' of the curves for the number of deaths). Notice that such a strategy to achieve herd immunity produces a dramatic number of deaths (2 million out of a population of 60 million) even in the most favorable case (multi-class model with $\theta = 1$).  
	We remark that the case fatality rate varies by time and location, and its measurement is affected by well-known biases exacerbated during the COVID-19 pandemic.
	Hence, the mortality in the numerical results might be overestimated, as the case fatality rate has been retrieved from data statistics that include both deaths {\it due to} or {\it with} COVID-19.
	
	Fig. \ref{fig3} reports identical results, but zooming in on the more reasonable control strategy in which 
	small $\lambda$ values are enforced.
	The economic cost (left $y$ axes) is approximately the same under single/multi-class models. This observation depends on the fact that, under the considered distribution $f_{r,p}$, $\EX[r] = 12.1$ and $\EX[r^2]/\EX[r] = 13.5$, are very close to each other. Moreover, the economic cost does not depend on $\theta$, as expected. In contrast, the predicted number of deaths (right $y$ axes) is highly diverse across different models and considered values of $\theta$. Even with $\theta = 1$, the single-class model predicts a much larger number of deaths.  
	This can be explained by the fact that under any realistic distribution $f_{r,p}$ (see Fig. \ref{fig:synthetic_populations}), contact rate $r$ and mortality $p$ are negatively correlated, such that most vulnerable individuals (elderly) have lower contact rate. Therefore, the single-class model, in which all individuals are identical, is more pessimistic in terms of deaths. More significant discrepancies between single- and multi-class models are observed with $\theta = 10$, since after saturation of ICU, occurring at the \lq bifurcation' points appearing in Fig. \ref{fig3}. Saturation effect of mortality probability (Eq. \equaref{pTD})  occurs in the multi-class scenario for some disadvantaged classes, but not in the single-class model. 
	
	\section{Experiments in a comprehensive scenario}\label{sec:reference}
	Our numerical results are obtained in a reference scenario roughly inspired by the actual evolution of COVID-19 in Italy during a period of 3 years, starting from the onset of the virus at the beginning of 2020. During this period, the dynamics of COVID-19 in Italy (and similarly in other European countries) have been characterized by three main phases, each spanning about one year:
	\begin{enumerate}
		\item {\bf first phase}: in this phase the most dangerous strains
		of the virus, e.g., the alpha and delta mutations, propagate
		in the absence of pharmaceutical interventions (vaccines), causing
		the majority of all deaths  attributed to COVID-19.
		\item {\bf second phase}: since the beginning of 2021, 
		vaccines started to be massively distributed to the population, 
		and almost all individuals (excluding no-vax people) completed
		the vaccination cycle (by receiving one or two doses) by the end of the second year.
		\item {\bf third phase}: since the beginning of 2022, 
		with the onset
		of the omicron variant, less dangerous but more virulent
		strains became prevalent, substituting the initial strains.
		Vaccines originally developed for the alpha and delta mutations
		also protected people against the omicron variant, though with 
		reduced efficacy.
	\end{enumerate}
	To capture the above dynamics, we made some simplifying
	approximations to limit the model complexity: we assume
	that a single variant (strain 1), with basic
	reproduction number $R_0^1 = 6$, propagates during the first
	2 phases, after which a new variant (strain 2) appears
	with higher $R_0^2 = 12$ and reduced mortality (by factor $q_{21}$ with
	respect to the mortality of strain 1, for each class of people).

	The parameters of our reference scenario are summarized in Table {\color{black} IV of the Appendix.} 
	Although our model and parameters can only roughly describe the actual dynamics of COVID-19 in Italy, they provide a realistic scenario in which different virus mitigation strategies can be compared, offering valuable insights.
	Of course, in our model for the reference scenario, we stratify the population using the $f_{r,p}$ distribution computed for Italy, as explained in \cite{SM}.
	Transition probabilities between compartments $I$,$H$,$T$,$D$ satisfying constraint \equaref{prodp} are set for simplicity as follows:
	$ p_{r,p}^{IH} = p_{r,p}^{HT} = \hat{p}_{r,p}^{TD} = p^{1/3} $.
	
	Strain 1 starts at time 0 with 1 initially infected individual.
	Similarly, strain 2 starts at time $t_2$ with 1 initially infected individual.
	We consider the economic cost function:
	$
	\mathfrak{C}(\rho) = (\rho-1)^{\alpha} 
	$ which satisfies the assumptions of Proposition \ref{prop:minimizing_cost} for $\alpha \geq 1$, and allows us to explore the impact of costs caused by more substantial non-pharmaceutical interventions by varying the single parameter $\alpha$.
	We emphasize that the resulting scenario is not specific to Italy: similar assumptions and parameters could describe equally well, at a high level, the dynamics of COVID-19 in other mid-size European countries or a single US state with a comparable population size.      
	At last, while each of the first two phases lasted approximately one year, in our analysis to have a complete view of the potential impact of different control approaches, we have also considered cases in which no effective treatments have been available for several years. When the epidemic spread out at the beginning of 2020, and the first decisions had to be made, no one could predict how long it would have taken to have effective vaccines/treatments available.
	\textcolor{black}{In the following, for the sake of simplicity, we neglect the term associated with the healthcare system stress by taking into consideration only social (deaths)  and economic costs, this  corresponds to set $\kappa_2=0$.}

	In the following three subsections, we start by considering only non-pharmaceutical interventions (corresponding to the \textit{first phase}). Then, we incorporate vaccinations (\textit{second phase}) into the model, and at last, we consider a comprehensive scenario with all three phases.
	
\subsection{First phase - Epidemics without vaccinations}
	Here we investigate the behavior of the system as a function of the control parameters, highlighting trade-offs of different regimes.
\subsubsection{Rate Control: assessing the impact of $\lambda$}
	\textcolor{black}{When controlling the infection rate, the suppression strategy, i.e., minimizing $\lambda$,} appears to be the most reasonable choice since it minimizes the number of deaths incurring an almost constant economic cost for, e.g., all values of $\lambda < 10000$. Indeed note that, once the system is stabilized around a fixed infection rate\footnote{Further, note that with proper control, the cost incurred during the transient phase necessary to bring the system to operate at a given $\lambda$ is negligible with respect to the long-term accumulated cost.} $\lambda^*$, the economic cost is the same for any $\lambda^*$, as long as $S(t) \approx N$. 
	
	In the case of COVID-19, some countries, e.g., China, have adopted the suppression strategy, which is particularly effective when restrictions can be geographically localized to small areas with limited impact on the national economy.
	Of course, this cannot be a solution in the long term unless the virus is totally eradicated or conditions change, e.g., herd immunity is reached through vaccinations.
	Indeed, note that all results discussed so far refer to a fixed time horizon $t_{\max} = 1$ year.
	
	\begin{figure} [htb]
		\centering
		\includegraphics[width=.75\linewidth]{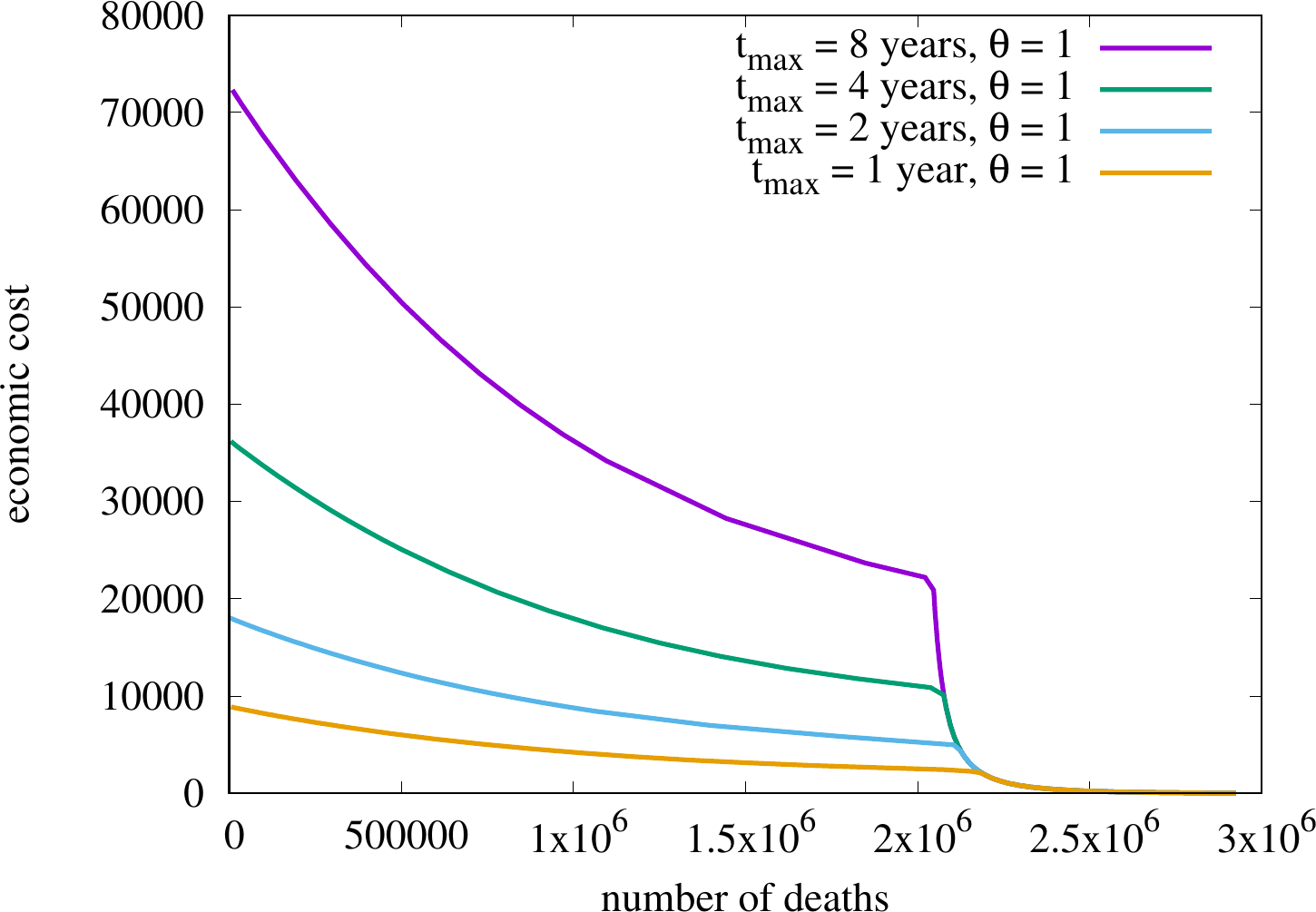}
		\caption{Parametric curves of economic cost vs social cost, as we vary $\lambda$, for different time horizons $t_{\max}$. Multi-class SIR model with $\theta = 1$.} 
		\label{fig4}
	\end{figure}

	To understand how the optimal strategy might change as we increase the time horizon $t_{\max}$, it is convenient to look at the plot in Fig. \ref{fig4}, showing parametric curves of economic cost vs. social cost, as we vary $\lambda$, for $t_{\max} = 1,2,4,8$ years.
	These results have been obtained by running the multi-class model with $\theta = 1$, putting us in the most favorable conditions (i.e., in the presence of unlimited healthcare facilities) to decide to abandon the suppression strategy.
	Clearly, under the suppression strategy, the economic cost increases linearly with time, so for $t_{\max}$ large enough, this strategy becomes necessarily suboptimal\footnote{It should be noticed, however, that a finite population model like ours is not adequate to describe a system running for more than, say, a few years.}.
	
	\begin{figure} [htb]
		\centering
		\includegraphics[width=.75 \linewidth]{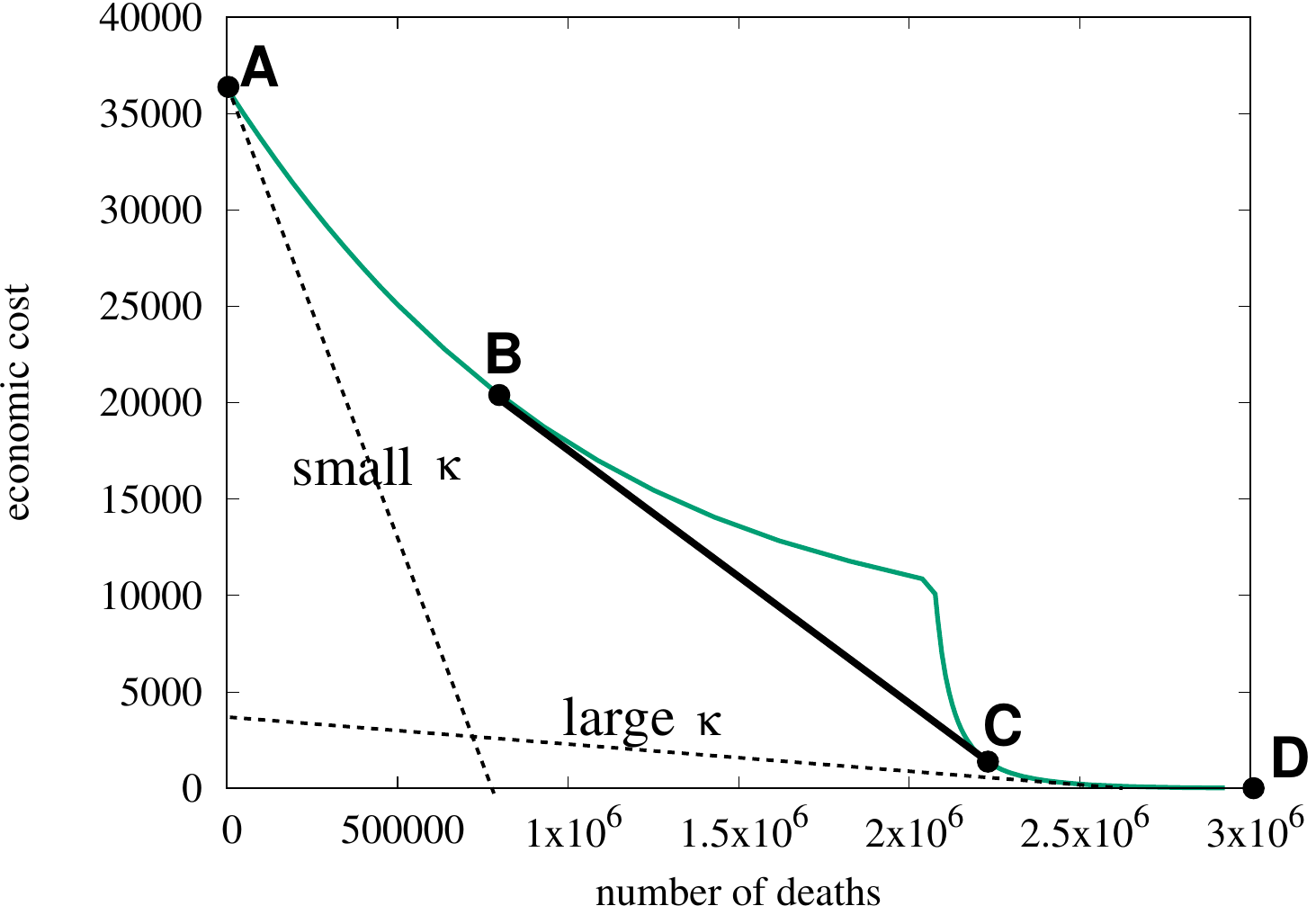}
		\caption{Pareto frontier of the multi-objective function \equaref{obj} in the
			case $t_{\max} = 4$ years. Multi-class SIR model with $\theta = 1$.} 
		\label{fig8}
	\end{figure}   
	
	Interestingly, curves shown in Fig. \ref{fig4} can be split into two convex parts connected at the point where the population reaches natural herd immunity (the knee). The consequences of this behavior on the multi-objective function \equaref{obj}, for $\kappa_2=0$ and $\kappa_3=1$, which is linear with respect to trade-off factor $\kappa_1$, are illustrated in Fig. \ref{fig8} for the case $t_{\max} = 4$ years.  
	We observe that all points between B and C are not Pareto-efficient, hence cannot be optimal solutions for the optimization problem \equaref{obj}. The optimal strategy exhibits a phase transition with respect to $\kappa_1$: for small values of $\kappa_1$ (social cost much more important than economic cost), the best strategy is total suppression (point A), whereas for large $\kappa_1$ we end up operating beyond the herd immunity knee.
	Intermediate solutions between A and B also exist, but only for a very small, particular range of $\kappa_1$ values.
	
	\begin{figure} [htb]
		\centering
		\includegraphics[width=.75\linewidth]{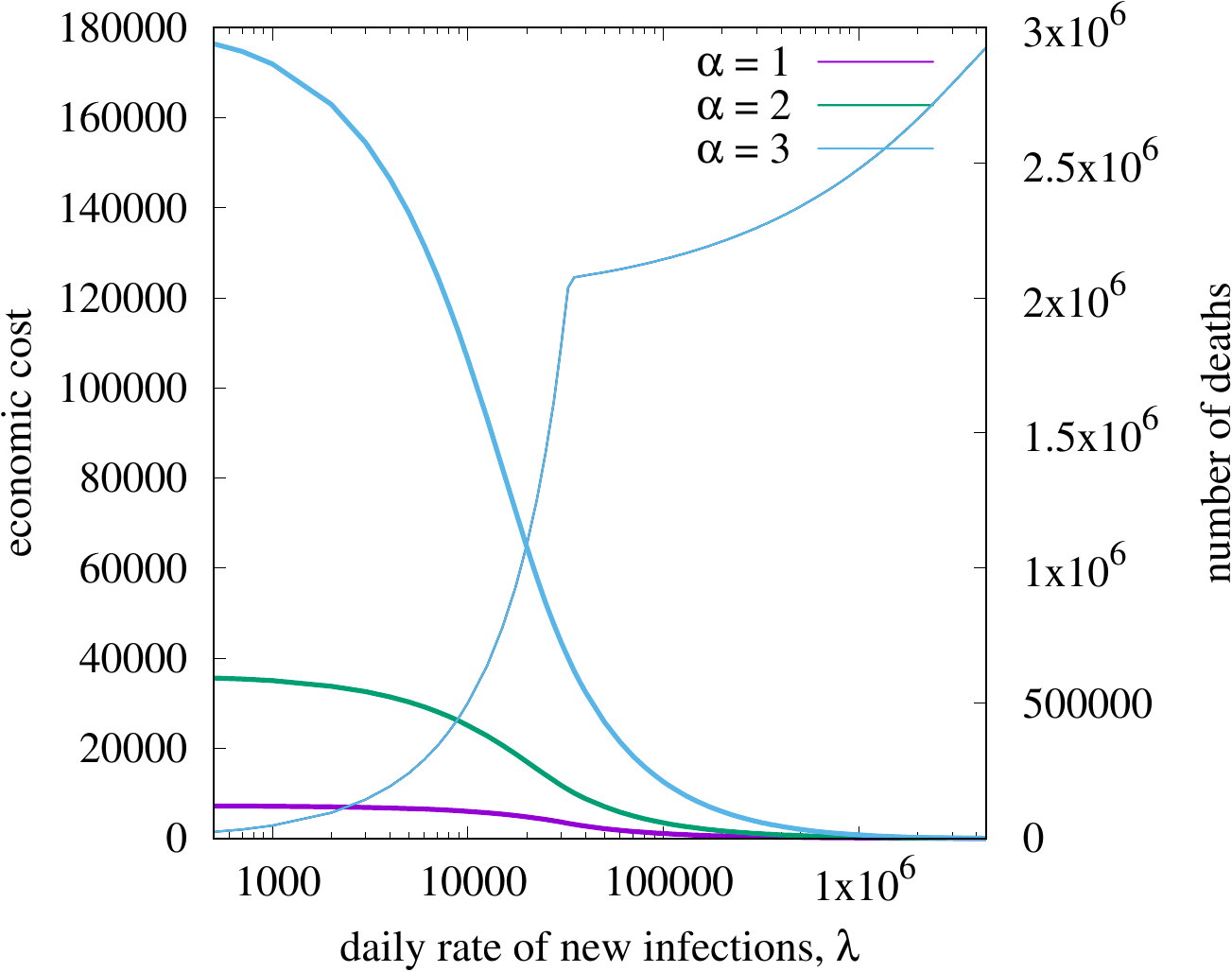}
		\caption{Economic and social costs as a function of 
			controlled rate $\lambda$, for fixed $t_{\max} = 4$ years and different values
			of $\alpha$.} 
		\label{fig6}
	\end{figure}
	
	The particular value of $\kappa_1$ at which the phase transition occurs, in addition to the time horizon $t_{\max}$, depends crucially on the exponent $\alpha$, as one can intuitively understand from Fig. \ref{fig6}, which shows economic and social costs as a function of the controlled rate $\lambda$, for fixed $t_{\max} = 4$ years, and different values of $\alpha = 1,2,3$: while the social cost is the same for all $\alpha$, the economic cost depends dramatically on $\alpha$. Note that $\alpha = 1$ is the extreme case for the validity of Proposition \ref{prop:minimizing_cost}.
	
	Proper values of $\alpha$ to be used in the model are difficult to set. However, the general conclusion remains the same: unless one considers considerably long (but unlikely to be significant) time horizons, the best option always appears to be the minimization of $\lambda$. With the parameters of COVID-19, and in particular, for the delta variant, the opposite \lq let it rip' strategy in which one tries to achieve the natural herd immunity (while still controlling $\lambda$ to avoid ICU saturation) produces an unreasonable social cost in terms of deaths. Some countries (like the UK) initially considered this option at the onset of the pandemic but quickly switched back to the suppression strategy after a few months. 
	
	Another reason why the 'let it rip' strategy considered so far is perilous is that it relies on the assumption that recovered people are immune forever, i.e., $\mu = 0$. In the case of COVID-19, natural immunity is progressively lost over time, so reinfections are possible about six months after recovery. Even assuming that reinfected people are much less likely to develop a severe form of the disease, we expect a significantly higher social cost when $\mu > 0$.
	This observation is confirmed by results in Fig. \ref{fig9}, showing economic and social costs for $t_{\max} = 4$ years, mortality reduction after the first exposure  $q_{\texttt{post}} = 10$, and different values of the average sojourn time in the immune state, equal to 6 months (as estimated for COVID-19), 1 year, 2 years, in addition to the optimistic hypothesis $\mu = 0$.

	\begin{figure} [htb]
		\centering
		\includegraphics[width=.75\linewidth]{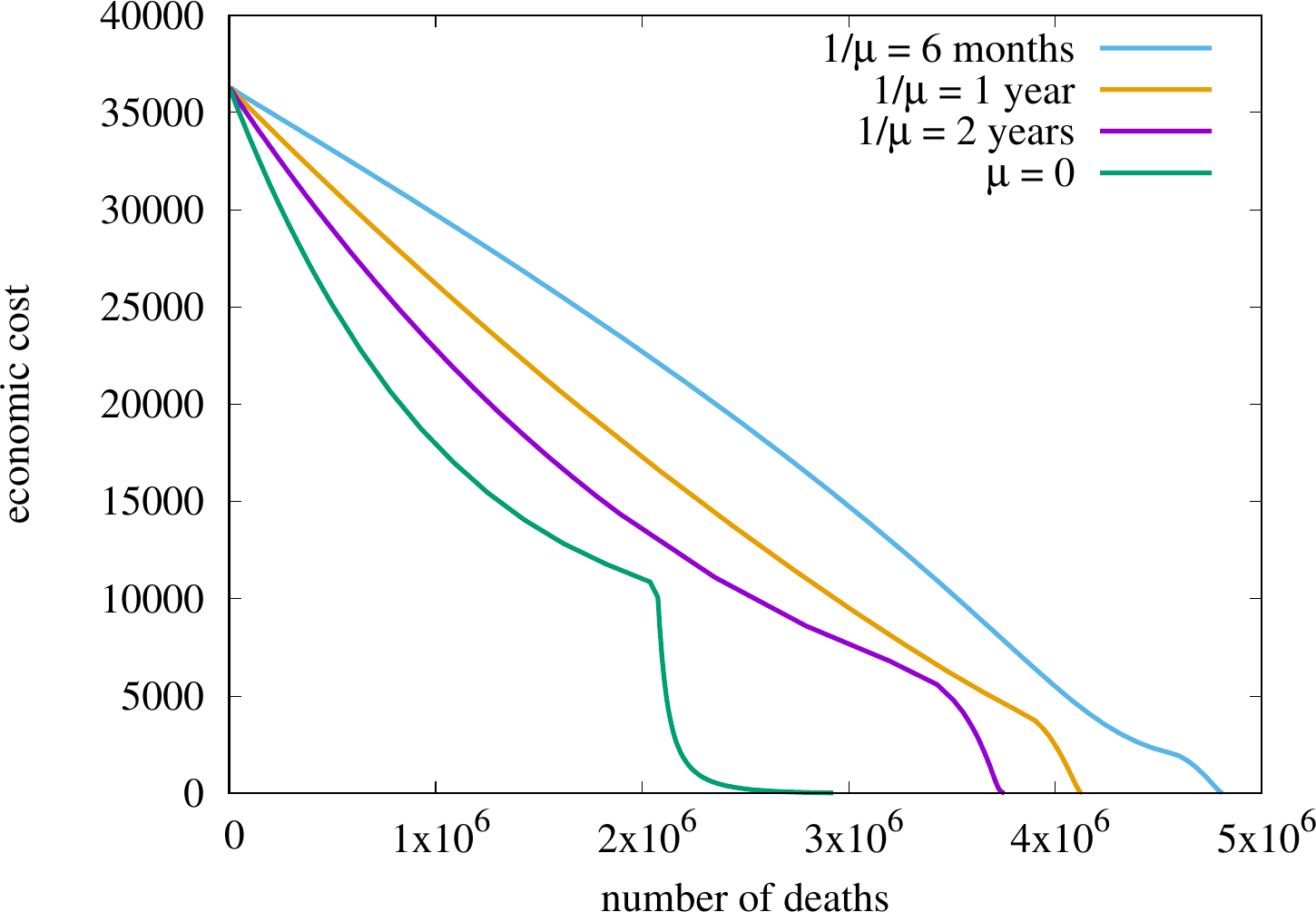}
		\caption{Parametric curves of economic cost vs social cost, as we vary $\lambda$,
			for fixed $t_{\max} = 4$ years and different   values of $\mu$. Mortality rate after reinfection is reduced by $q_{\texttt{post}} = 10$. } 
		\label{fig9}
	\end{figure}

	Note that at the beginning of the pandemic, the decision of which strategy to adopt was daunting because nobody knew the characteristics of the novel virus and whether effective vaccines could be developed, and after how much time.
	It was also unknown when and which mutations of the original virus would have replaced the original strain.
	In later sections, we will bring into the picture these two fundamental factors that have steered the pandemic's evolution after the first year.
	
	\medskip
\subsubsection{HT Control: assessing the impact of $H_{\max}$ and $T_{\max}$}

	We start analyzing the impact of parameters $H_{\max}$ and  $T_{\max}$ on the system dynamics.
	Both the implemented controllers are linear.
	
	In all the cases we have set $\widehat{H}=50k$ and $T_{\max}= \widehat{T}$.
	The choice $T_{\max}= \widehat{T}$ is justified by our previous analysis,  according to which  the maximization of ratio $T_{\max}/H_{\max}$  favors  system local stability around the equilibrium point.
	Note that our choice of parameters guarantees  local stability also in cases in which  the tightest control at the equilibrium point is exerted  by intensive therapies occupancy (indeed $1/\gamma= 1/\phi+1/\tau$).  Finally  note that, since in our scenario $\frac{\phi}{\tau} \frac{\EX[r \, p_{r,p}^{IH} \, p_{r,p}^{HT} ]}{\EX[r \, p_{r,p}^{IH}]}=0.331$,  we should enforce  $ T_{\max}/H_{\max}< 0.331$ to guarantee that at  the equilibrium point, the  tighter control is exerted by hospitalizations.   
	
	Figure \ref{fig:emilio} reports some result.  
	First, we have fixed  $T_{\max}= 10k$ and we let $H_{\max}$ vary. In particular we have chosen: $H_{\max} = 20k$ (top left plot),  $H_{\max} = 30k$ (top right plot), $H_{\max} = 50k$ (bottom left plot). 

	Only the first choice for $H_{\max}$ satisfies condition \eqref{condICU}.
	Note that by reducing  $H_{\max}$, we significantly reduce oscillations since the control on hospitalization becomes reactive. Periods in which the tightest control is exerted by hospitalizations/intensive therapy occupancy are highlighted in the figures. In no cases saturation of intensive treatment facilities is observed.
	Table \ref{CD-table}  complements the previous figure by reporting economic costs (with $\alpha=1,2,3$) and deaths for all scenarios.   In general, more conservative choices of $H_{\max}$ lead to significant reductions in the number of deaths, and in some cases also in the economic cost, as an effect of the reduction of oscillations.
	
	We have also tested, reporting results in Table \ref{CD-table}), situations in which $T_{\max}/H_{\max}$ is kept fixed equal to two (so to guarantee the satisfaction of condition \eqref{condICU}, while $T_{\max}$ is set respectively to $5k$, $10k$ and $20k$. Note that we obtain different trade-offs between economic cost and number of deaths. 
	In general, by increasing $T_{\max}$, we reduce the economic cost and increase the number of deaths.
	Evolution of metrics for the case   $T_{\max}=5 k$,   $H_{\max}= 10 k$  is  shown in Figure  \ref{fig:emilio} (bottom right plot).  
	In this case, contrarily to the case  $T_{\max}=10 k$ and   $H_{\max}= 20 k$, intensive therapy control exerts the tightest control for a given short period.
	
	At last,  Table \ref{CD-table}  reports results for the case  $T_{\max}=10 k$,   $H_{\max}= 10 k$.
	Observe that the performance of this last case is almost indistinguishable from the case $T_{\max}=5 k$,   $H_{\max}= 10 k$ (which requires just half of the intensive therapy facilities) both in terms of deaths and economic cost.
	
	In conclusion, in our scenario keeping the ratio $T_{\max}/H_{\max}\approx 2$ appears to be the best choice, as it guarantees that the tightest control is essentially always exerted by hospitalizations in dynamic conditions.
	Then   $T_{\max}$ (and consequently $H_{\max}$) should be chosen instead to achieve the desired trade-off between deaths and economic cost (as previously observed, deaths are more sensitive to parameters than economic costs).  
	In our analysis, we have neglected the costs related to the creation/maintenance of sanitary facilities (which are typically small with respect to general economic costs due to restrictions) to limit the number of free parameters. However, extending the model to include such costs would be relatively immediate.
	\begin{table}
		\caption{costs and deaths}\label{CD-table} 	
		\centering
		\begin{tabular}{cccccc}
			\hline
			\hline
			$T_{\max}$ & \!$H_{\max}$ &\!\text{cost} $(\alpha=1)$  
			&\!\!\text{cost} $(\alpha=2)$ & \!\!\text{cost} $(\alpha=3)$ & \!\!\text{deaths} \\
			\hline
			5$k$   & 10$k$   &  2.03$k$ & 13.0 $k$  &  102 $k$ &    13.2  $k$  \\
			10$k$ &  10$k$  &    2.03$k$     &    13.0 $k$                    &   101  $k$         &    13.2 $k$            \\
			10$k$  &       20$k$     &   1.94$k$  &      11.1   $k$   &  71.3  $k$  &  25.3    $k$                            \\
			10$k$   &   30$k$   & 1.92$k$  & 10.9 $k$ &   66.6 $k$  &   35.9  $k$   \\
			10$k$   & 50$k$   &  2.12$k$ & 14.6 $k$ &  115 $k$ &    42.2   $k$  \\
			
			20$k$  &  40$k$   & 1.88$k$     & 10.2 $k$  & 59.3 $k$ & 49.0 $k$   \\
			\hline
			\hline
		\end{tabular}
	\end{table}
\begin{figure*}[tb]
	\centering
	\begin{tabular}{cccc}
		\includegraphics[width=0.47\columnwidth]{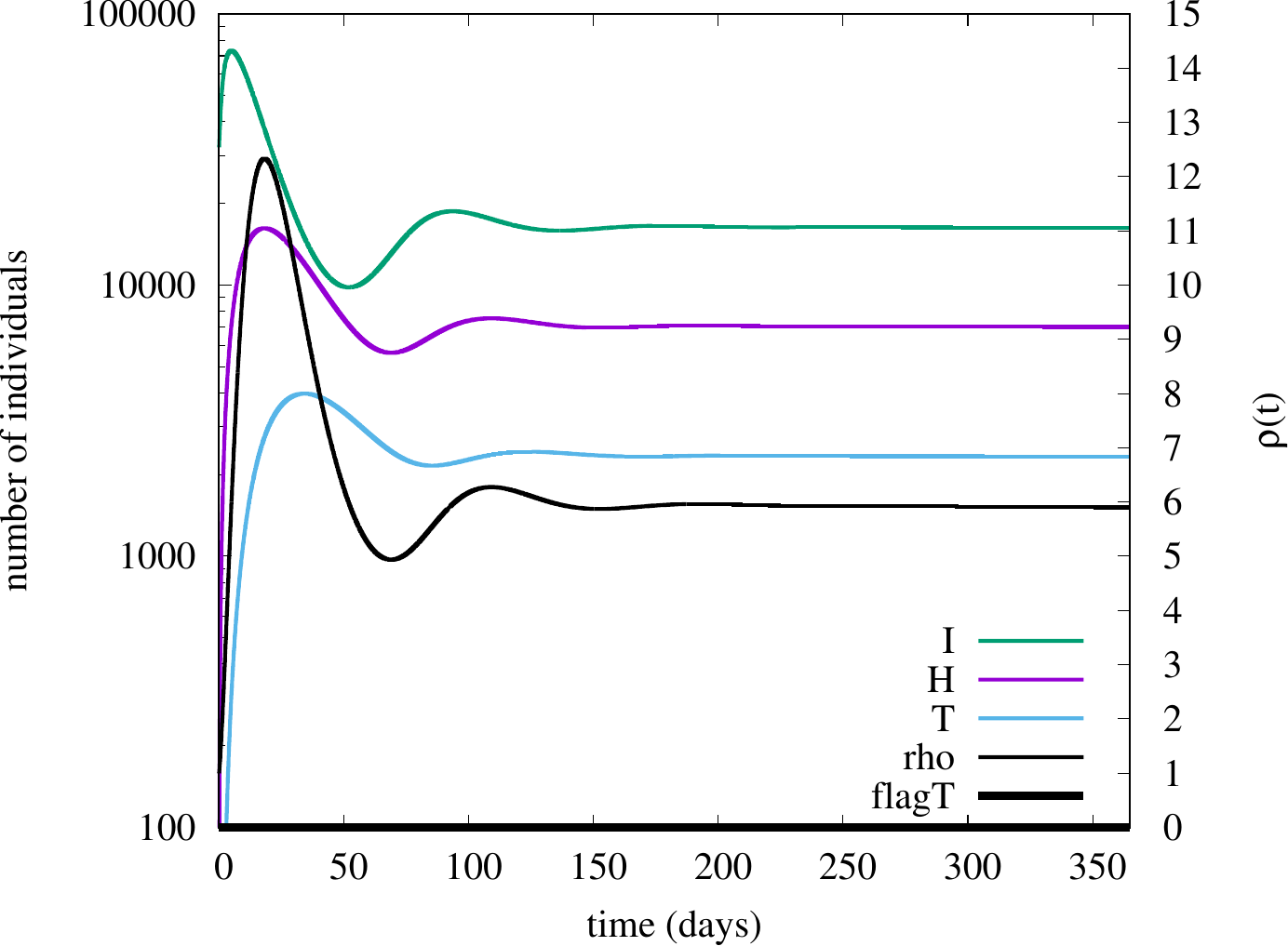}&
		\includegraphics[width=0.47\columnwidth]{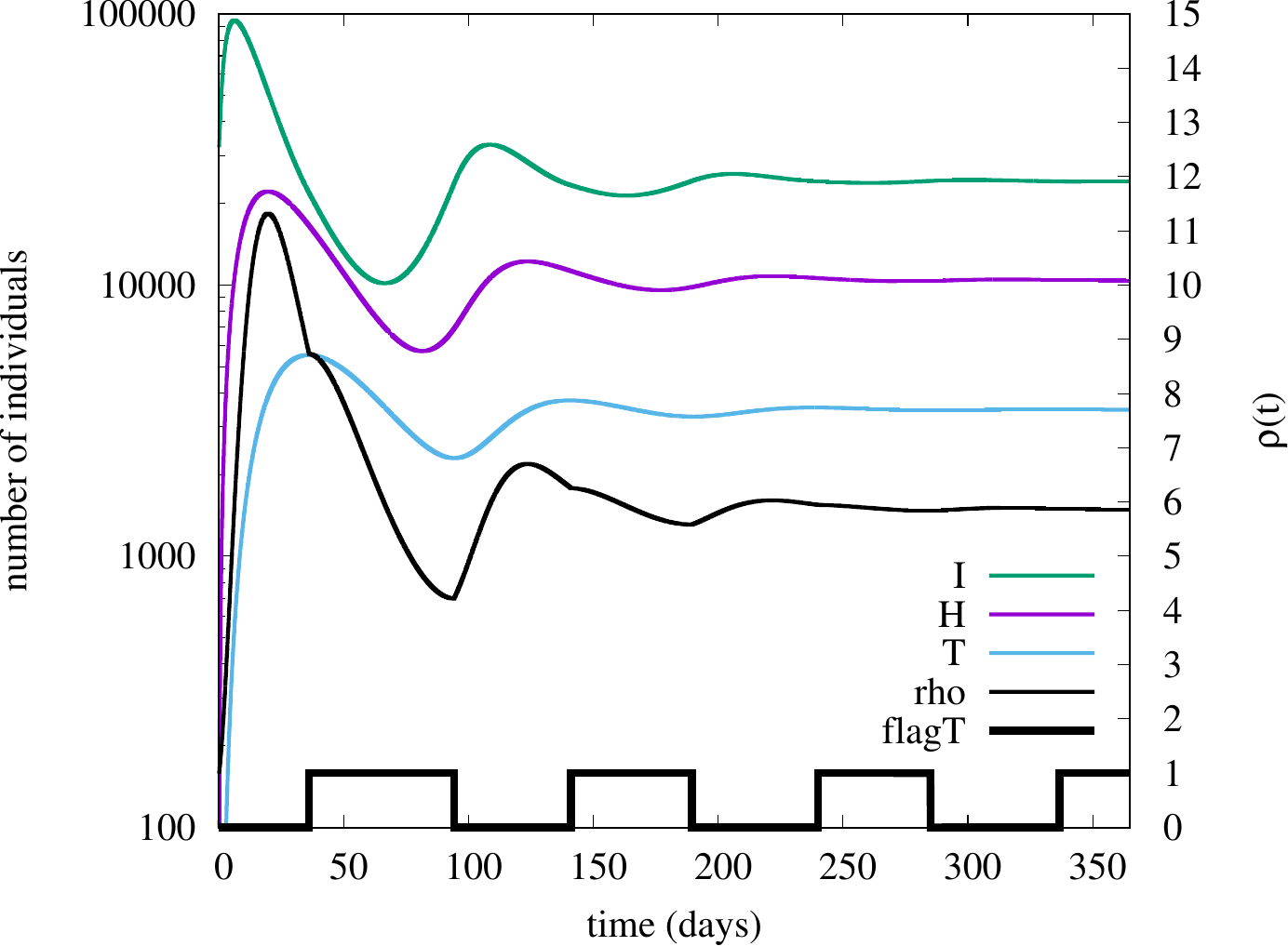}&
		\includegraphics[width=0.47\columnwidth]{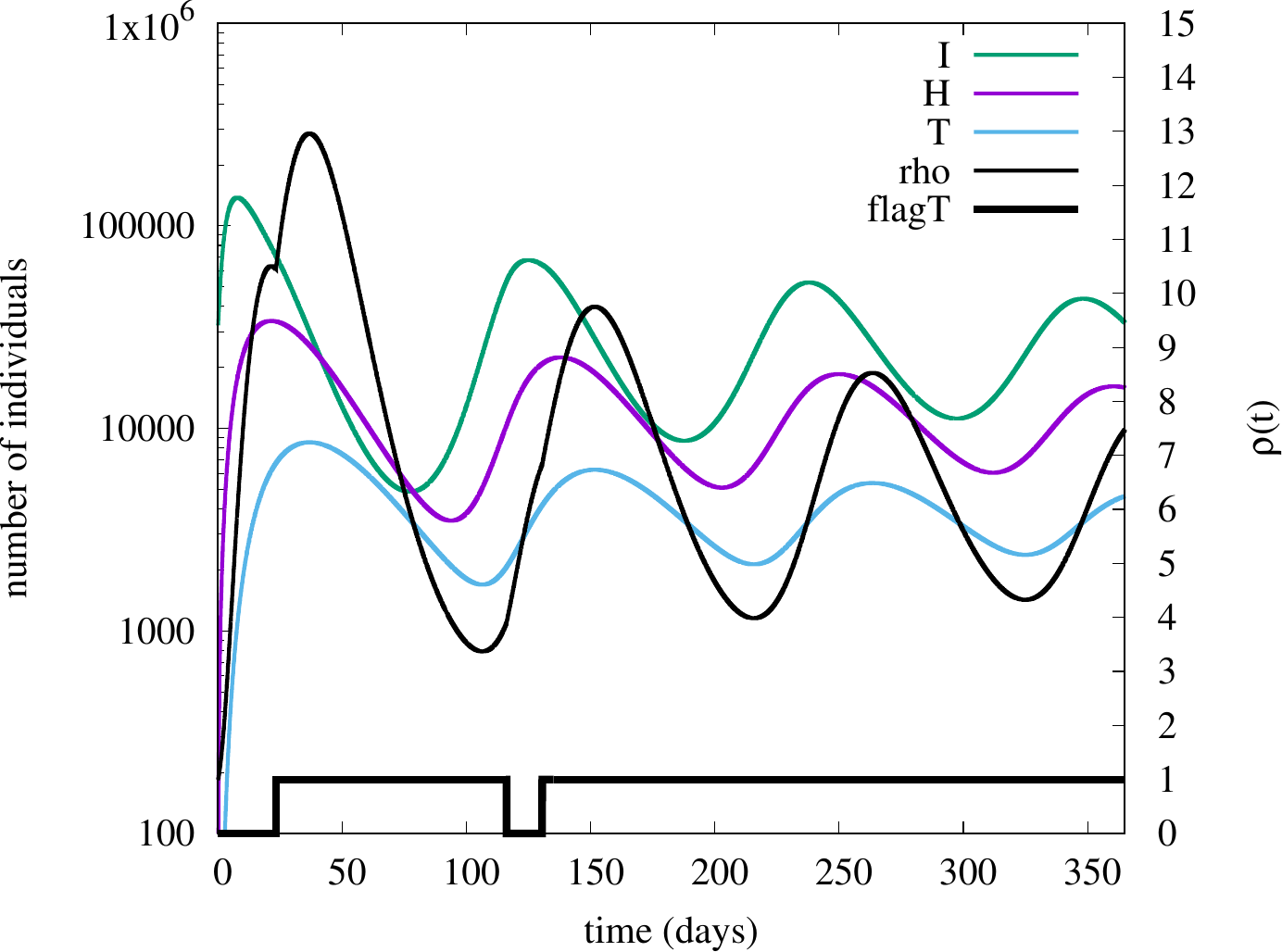}&
		\includegraphics[width=0.47\columnwidth]{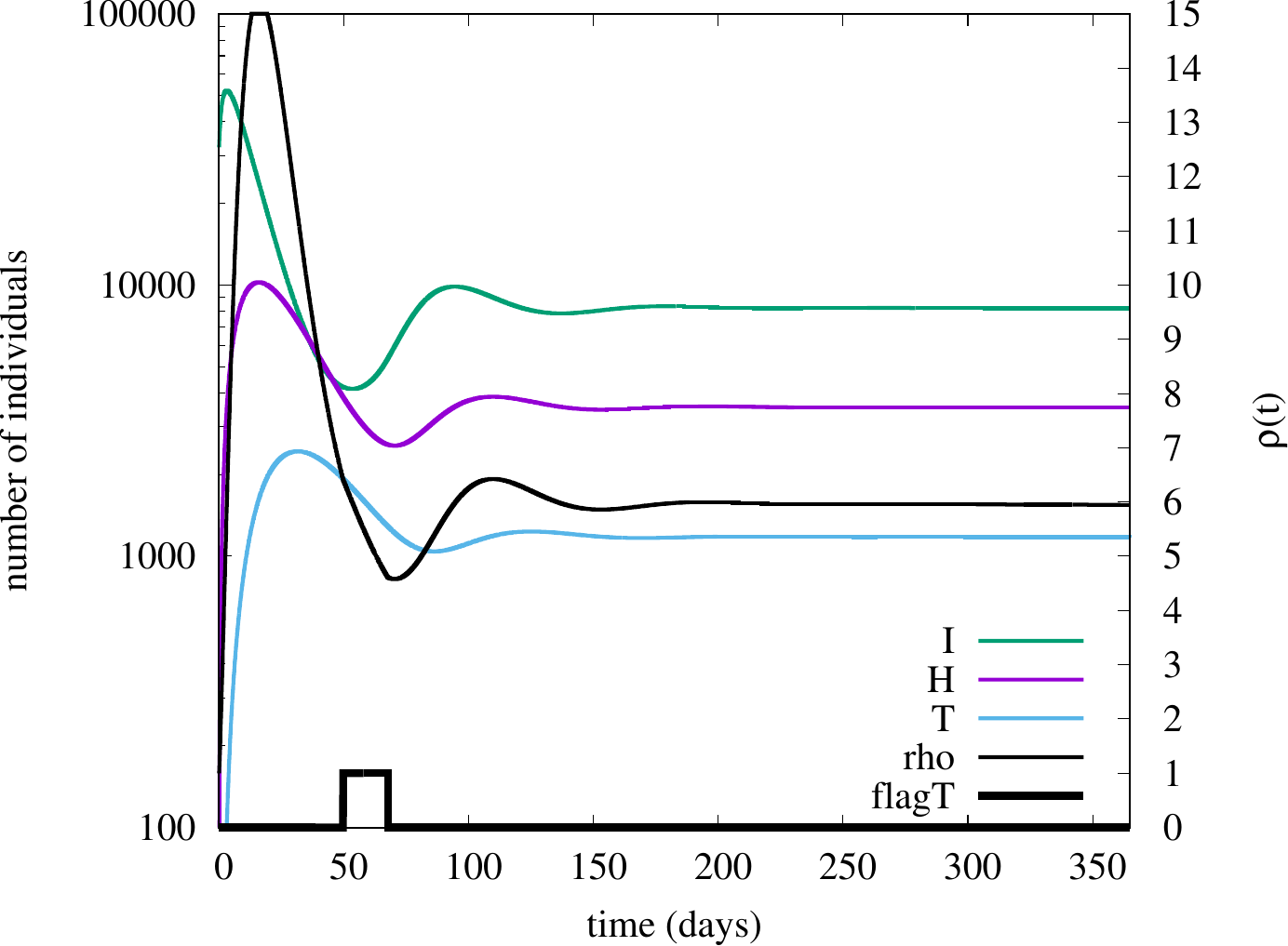}
		\\
		(a)&(b)&(c)&(d)\\
	\end{tabular}
	\caption{Evolution of $I(t)$,$T(t)$ (left $y$ axes) 
		and $\rho(t)$ (right $y$ axes), for different combinations
		of $T_{\max} = \widehat{T}$ and $H_{\max}$, 
		and fixed $\widehat{H}=100k$. 
		The rectangles at the bottom of the plots indicate periods in which 
		$\rho(t)$ is determined by $T(t)$.
	}\label{fig:emilio}
\end{figure*}
	
{\color{black}\subsection{Second phase - Mobility restrictions and vaccinations}\label{subsec:vaccinations}}
	We now add the \lq second phase' of our reference scenario, considering the joint impact of
	vaccination policies and control strategies during the second year of the pandemic.
	Recall from Sec. \ref{subsec:priority} that we focus on two extreme vaccine prioritization
	policies: Most Vulnerable First (MVF) and 
	Most Social First (MSF). 
	
	We will consider a single type of vaccine to be administered in two doses separated by a variable interval of $\Delta$ days. In this way, we can address an issue raised in some countries, e.g., the UK, when vaccines started to be available for mass distribution, i.e., whether it is better to follow the recommended protocol ($\Delta = 21$ days) or to give one dose to the largest possible population, before administering the second dose.
	The latter policy, which aims at partially immunizing a vast portion of the population, corresponds to choosing $\Delta = 135$ days.
	In our investigation, we assume the vaccination rate to be constant and such that the entire population can receive two doses after 9 months (270 days).
	
	No vaccine is available during the first year (first phase). 
	To better compare our two control strategies, we initially start the system at the equilibrium point ($I^*,H^*,T^*$), disregarding the transient needed to reach such equilibrium.\footnote{A comprehensive analysis of the complete scenario also comprising the initial transient will be presented later in Sec. \ref{Sect:finalresults}.}
	Under the HT strategy, we assume that control is always determined by the occupation of regular hospitals, rather than ICU, by adequately setting the ratio $T_{\max}/H_{\max}$.
	Moreover, note that the parameters of the HT strategy can be tuned to achieve the desired number $I^*$ of infected people at the beginning of the pandemic. This allows us to compare the trade-offs achievable by our two control policies.
	
	{\color{black}{Given the current understanding of COVID-19 vaccines, one limitation of the approach is the uncertainty surrounding the specific efficacy of different vaccines and their effectiveness against emerging variants. Vaccine efficacy can vary depending on age, underlying health conditions, and individual immune response. Additionally, the duration of vaccine-induced protection and the potential for waning immunity over time are still being studied.
	As a result, the parameters related to vaccine prioritization, such as the efficacy rates and the duration of protection, are subject to a range of values rather than precise estimates. The lack of comprehensive knowledge about these parameters restricts the ability to determine an optimal vaccination strategy with certainty.
	Therefore, the study may need to consider a range of plausible values for vaccine-related parameters and perform sensitivity analyses to assess the robustness of the results under different scenarios. Given the considerations above, we introduce variability in the efficacy ratio between the first and second doses of the vaccine. Specifically, we examine two different values for this ratio, denoted as $\mathsf{VE}^{1}/\mathsf{VE}^{2}$, namely 0.3 and 0.6. Meanwhile, we keep the efficacy of the second dose fixed at $\mathsf{VE}^{2}=0.9$. By incorporating this range of values for the efficacy ratio, we account for the uncertainty surrounding the relative effectiveness of the two vaccine doses.
	}}
	
	\begin{figure*}[tb]
		\centering
		\begin{tabular}{cccc}
			\includegraphics[width=0.47\columnwidth]{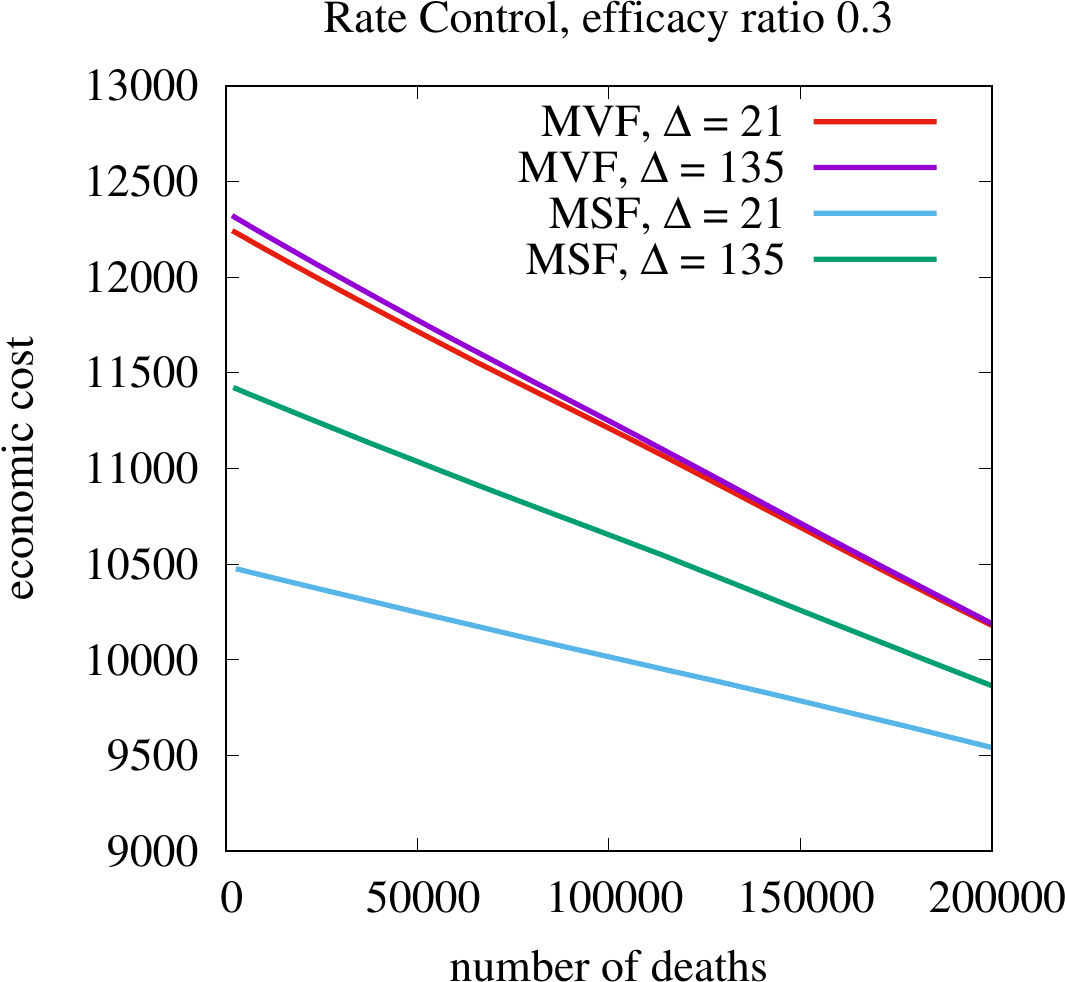}&
			\includegraphics[width=0.47\columnwidth]{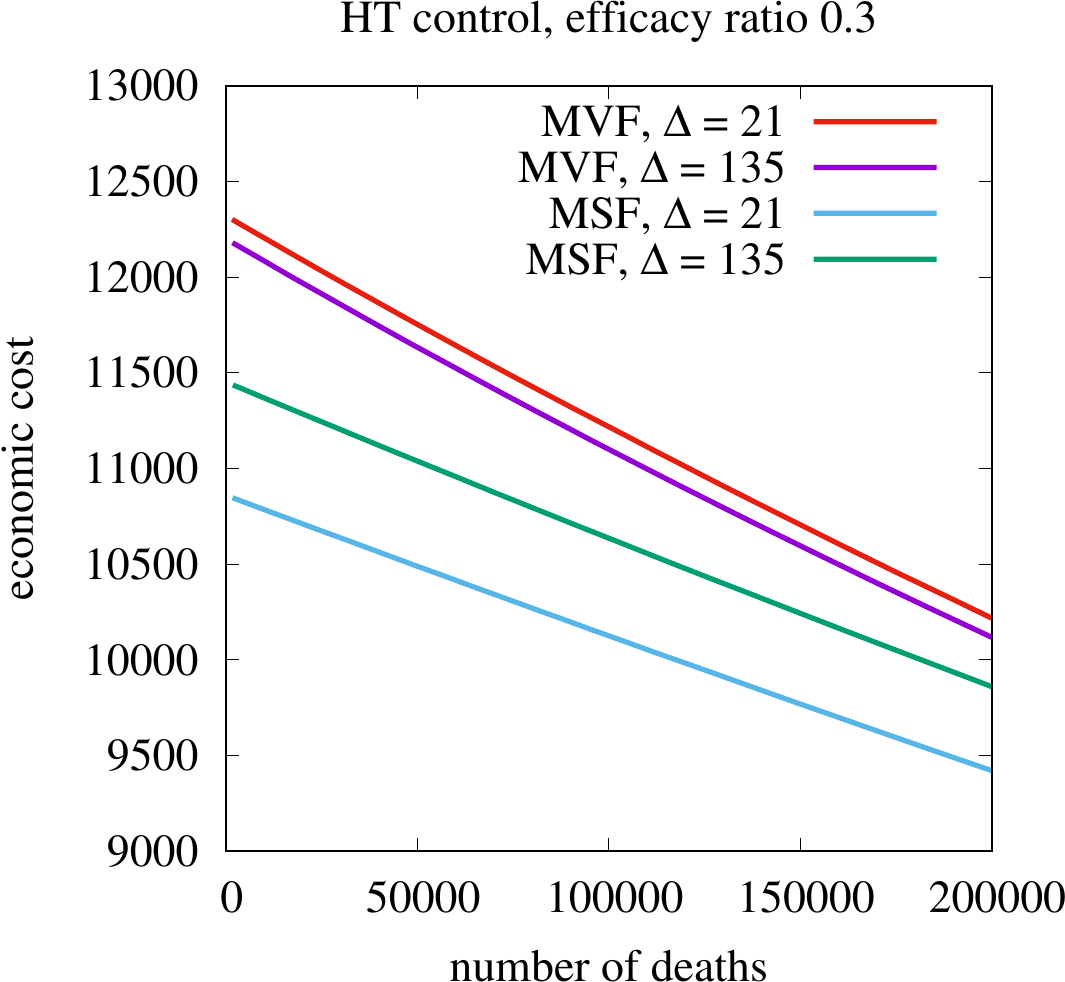}&
			\includegraphics[width=0.47\columnwidth]{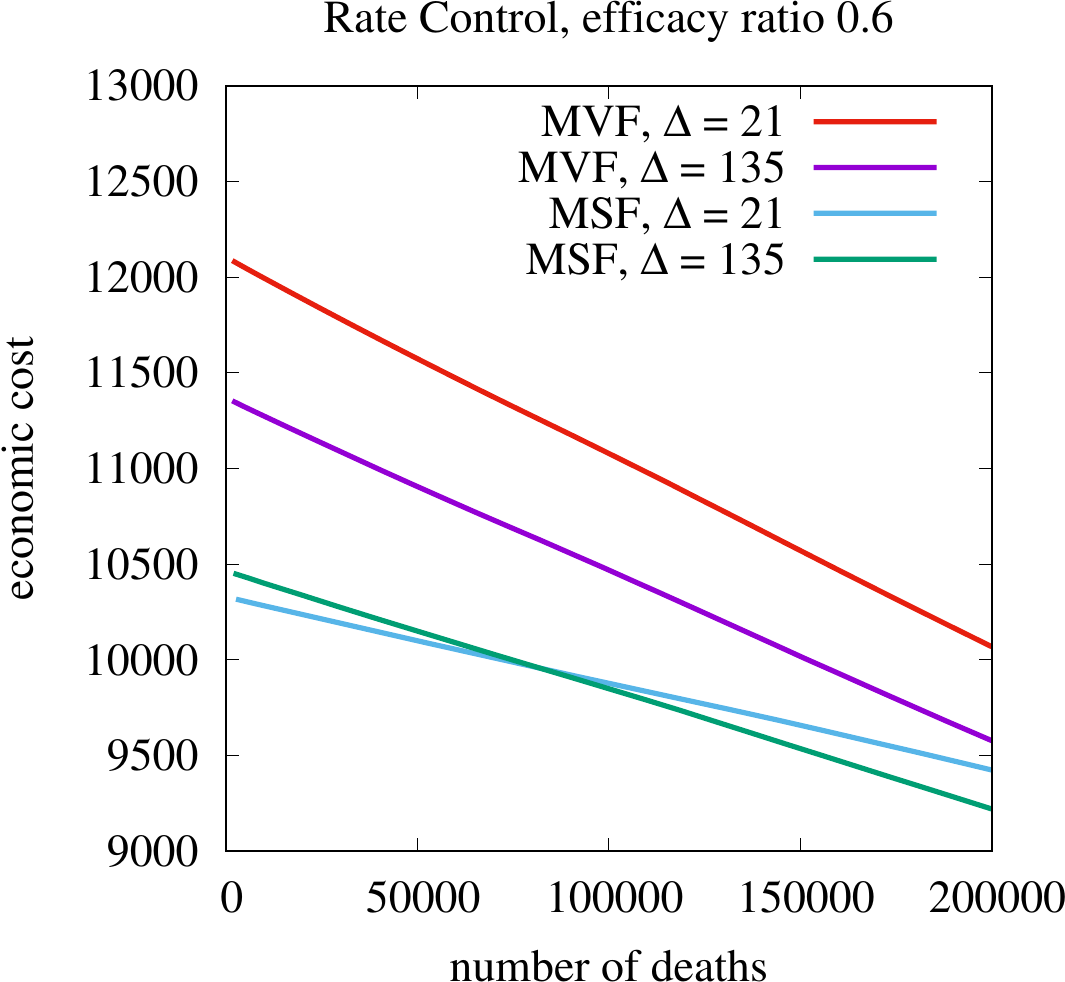}&
			\includegraphics[width=0.47\columnwidth]{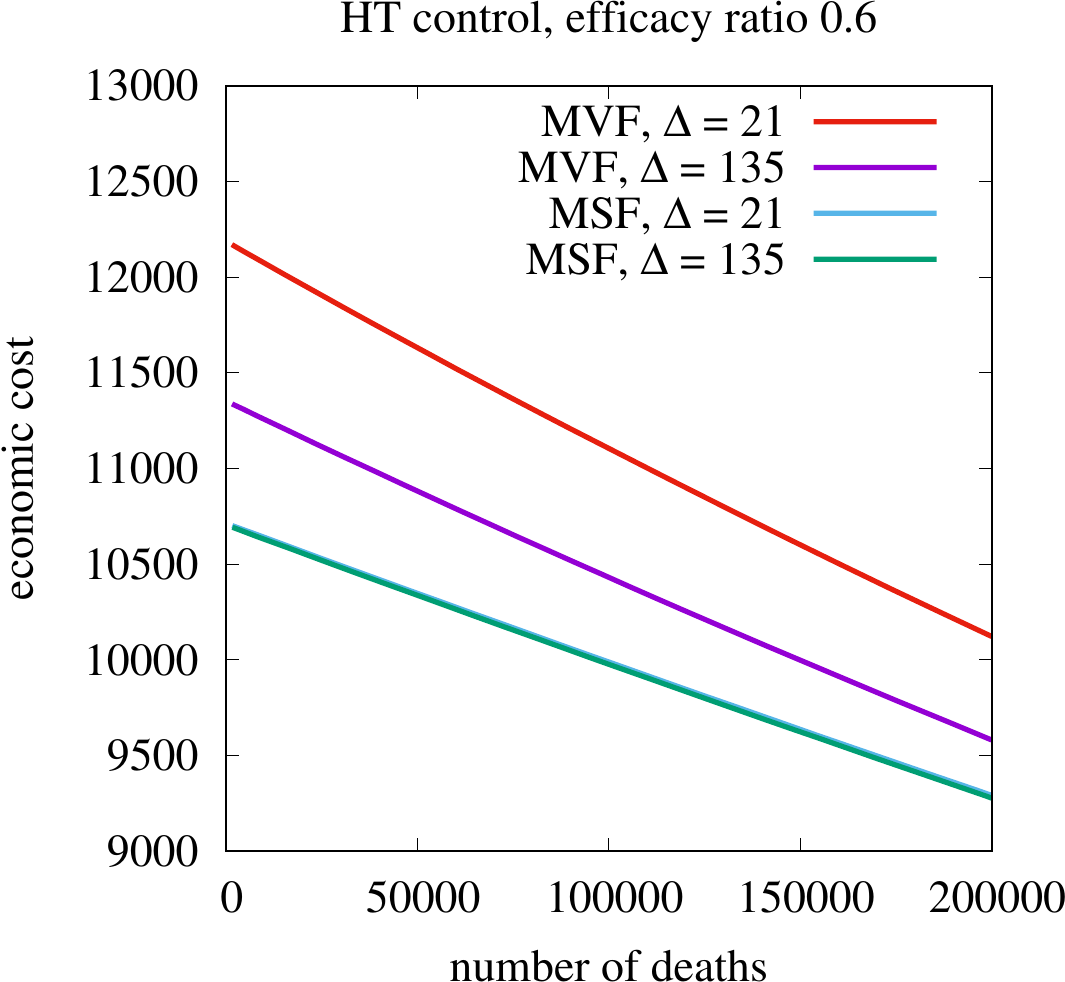}
			\\
			(a)&(b)&(c)&(d)\\
		\end{tabular}
		\caption{Impact of vaccination policies and control strategies on deaths and economic cost. All individuals are vaccinated in 270 days. } 
		\label{fig:vaccination_control_strategies}
	\end{figure*}
	
	The achievable trade-offs between economic cost and number of deaths, measured at the end of the second year, are shown in plots (a),(b),(c), and (d) of Fig. \ref{fig:vaccination_control_strategies}, for the four combinations arising from the two considered control policies and the two considered efficacy ratios (see plot titles). Each plot contains four curves related to the four combinations of vaccination policies (MSF vs. MVF, $\Delta = 21$ vs. $\Delta = 135$).
	
	Several observations are in order. First, the MSF policy (green and blue curves) generally outperforms MVF (red and purple curves). This fact is not trivial and depends crucially on the extent of the negative correlation between $r$ and $p$ in the population distribution $f_{r,p}$. Note that the MSF policy is hardly implementable in practice. Indeed, only the MVF policy has been deployed in many countries, by simple age prioritization, except for special categories of workers (e.g., healthcare workers) who have also received the vaccine in advance due to their exposition to the virus.
	
	Second, as expected, the efficacy ratio of 0.6 leads to better outcomes than the efficacy ratio of 0.3. In particular, delaying the distribution of the second dose ($\Delta = 135$) is not advisable if the first dose is relatively ineffective (efficacy ratio 0.3).
	
	Third, the impact of different control strategies is fairly small, with rate control slightly outperforming HT control. The best possible trade-offs, i.e., the lowest possible curves, are generated by the rate control, MSF, and a properly tuned $\Delta$ (note the crossing between blue and green curves on plot Fig. \ref{fig:vaccination_control_strategies}(c)).
	
	The effect of the two control strategies, combined with different vaccination policies,  can be better understood by looking at temporal dynamics shown in Fig. \ref{figcompareICU} for rate and HT control. In both cases, we assume an initial number of infected people $I^* = 32,000$ (corresponding to $\lambda_C=4,000$) while restricting ourselves to an efficacy ratio of 0.6.       
	
	\begin{figure} [h!]
		\centering
		\includegraphics[width=.7\linewidth]{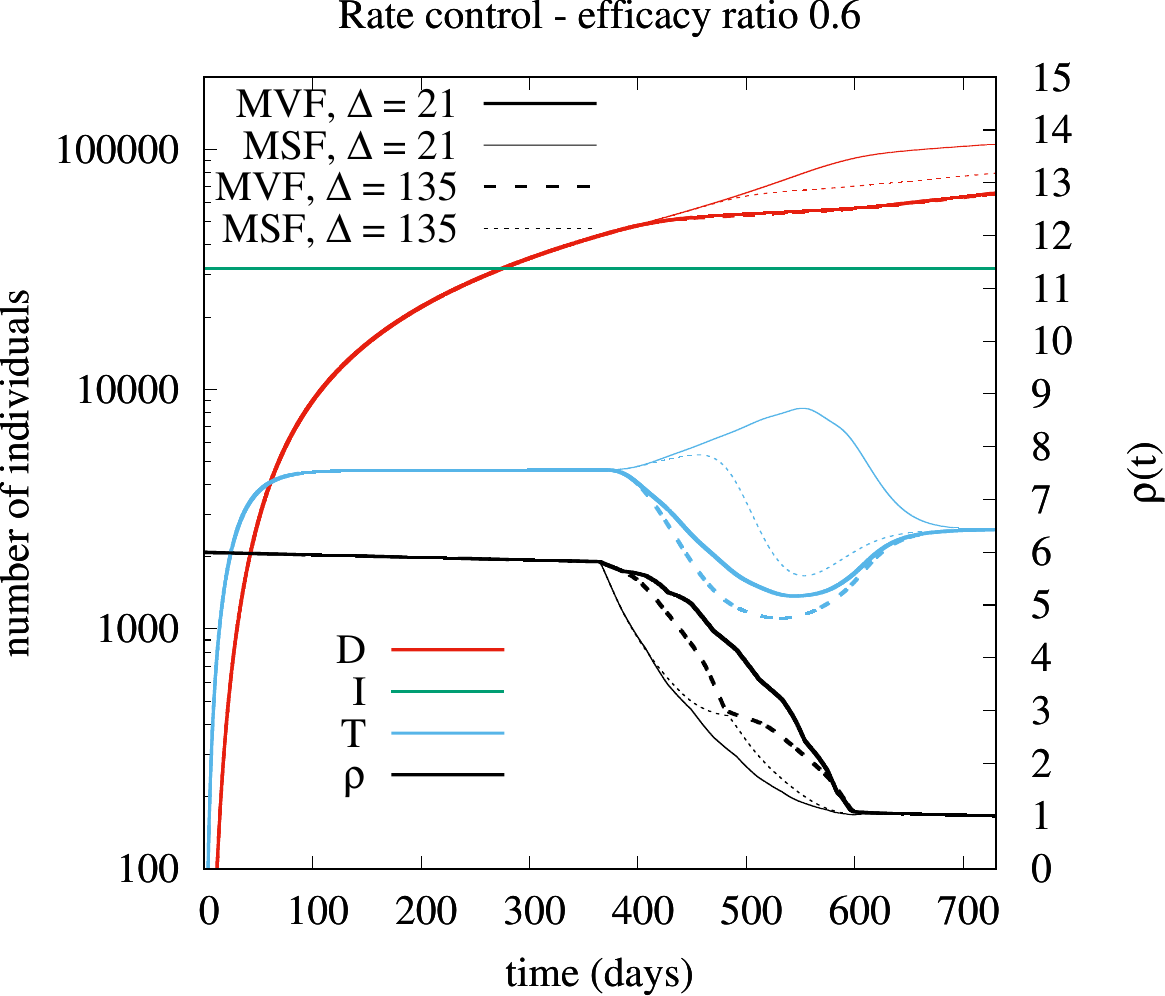}
		\vspace*{1.5mm}
		\includegraphics[width=.7\linewidth]{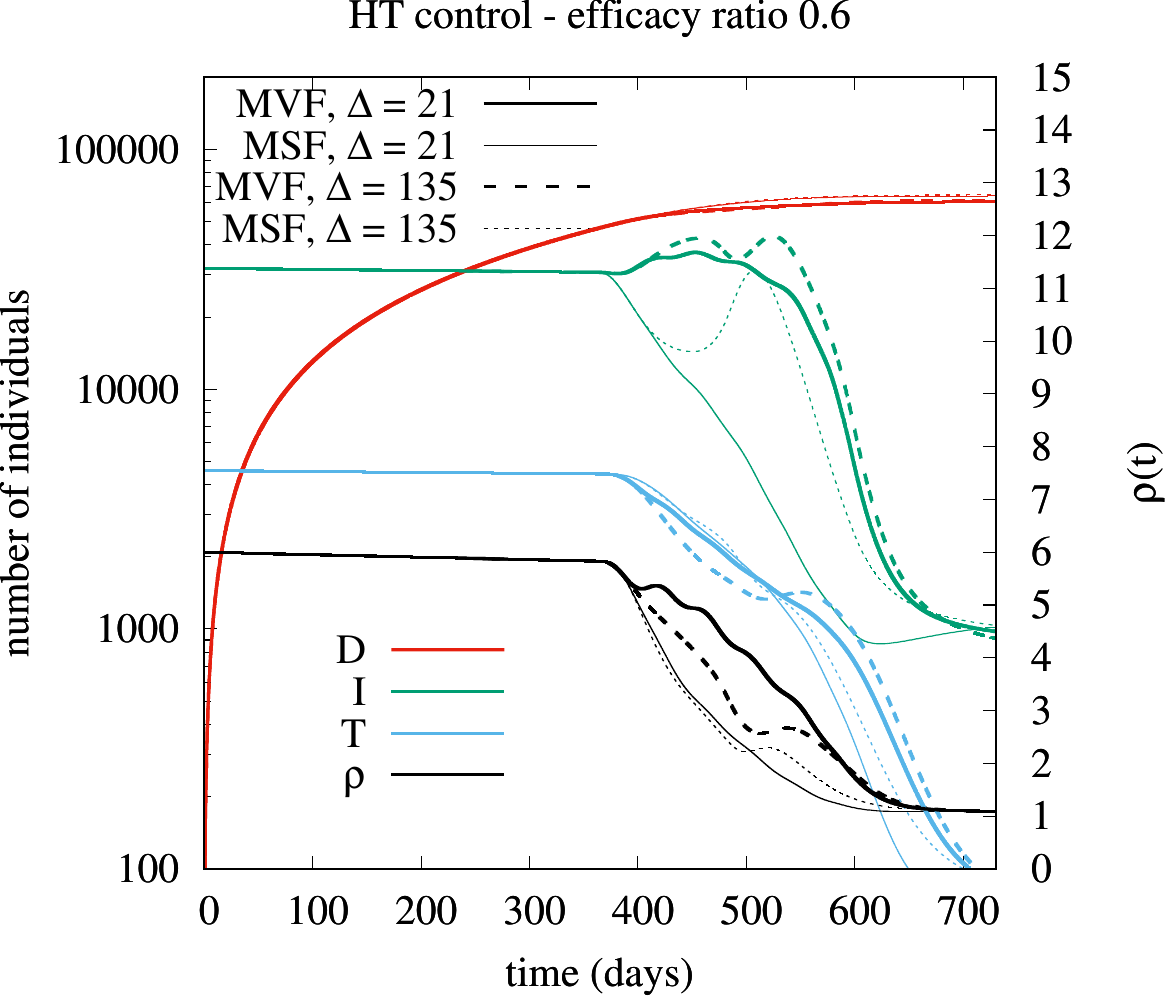}
		\caption{Evolution of $I(t),D(t),T(t),\rho(t)$ in the case of Rate control (left) and HT control (right), efficacy ratio 0.6, and different vaccination policies (different line styles of the same colour). } 
		\label{figcompareICU}
	\end{figure}
	
	The evolution of $D(t)$, $I(t)$, $T(t)$, $\rho(t)$ in Fig. \ref{figcompareICU} is shown by curves of different colors,
	respectively red, green, blue, and black. Thick (thin) lines correspond to MVF (MSF).
	Solid (dashed) lines correspond to $\Delta = 21$ ($\Delta = 135$).
	Let us start with the simpler case of rate control in Fig. \ref{figcompareICU}.
	Here, $I(t)$ is maintained constant through the entire period of two years. When vaccinations start (day 365), two extreme behavior for $\rho(t)$ arise, as expected, by MSF with $\Delta = 21$ (thin dashed black line) and MVF with $\Delta = 21$ (thick dashed black line), with the other curves (related to $\Delta = 135$) lying in between these two.  
	MSF with $\Delta = 21$ allows us to release social restrictions more quickly, lowering the economic cost at the expense of more deaths. 
	The case of HT control in Fig. \ref{figcompareICU} is more complex, since here $I(t)$ is not constant and, in fact, decreases drastically during the second year thanks to the self-adaptive nature of HT control.
	
	The fact that better trade-offs are achieved by the not self-adaptive rate control at the end of the second year may appear counter-intuitive. 
	Note, however, that such better trade-offs are only possible under a carefully tuned MSF policy, and they are thus hardly achievable in practice.
	At last, observe that in a more realistic setting, one might not arbitrarily choose the rate of new infections. For example, if one cannot operate below $\lambda_C=4,000$, from Fig. \ref{figcompareICU}, the best option would likely be MVF, which produces significantly fewer deaths at the expense of a tolerable and largely justifiable increase of the economic cost.
	Interestingly, in this case, $\Delta = 135$ would produce a significantly lower penalty in the economic cost with respect to $\Delta = 21$ while generating an almost identical number of deaths. 
	\smallskip
\textcolor{black}{\subsection{Third phase - Control in a comprehensive scenario}\label{Sect:finalresults}}
	At last, \textcolor{black}{we consider a scenario encompassing all three epidemic phases, spanning over three years, as described in Sect. \ref{sec:reference}}. The MVF-$\Delta=21$ vaccination policy was chosen in light of the fact that many countries have largely adopted this policy. The ratio between the first and second doses' efficacy has been set to $0.6$.
	
	\begin{figure*}[tbh]\centering
		\begin{tabular}{ccc}
			\includegraphics[width=0.62\columnwidth]{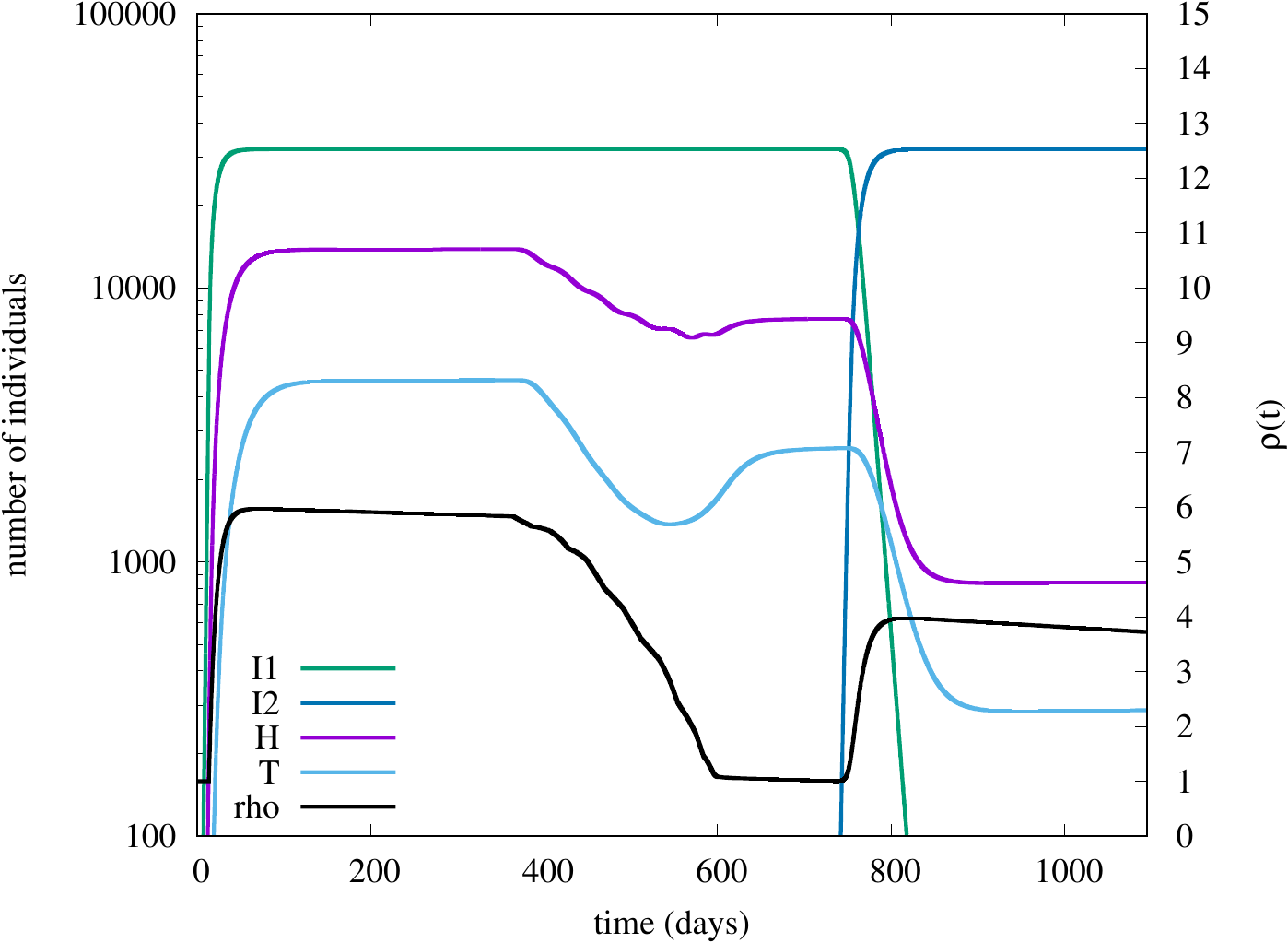}&
			\includegraphics[width=0.62\columnwidth]{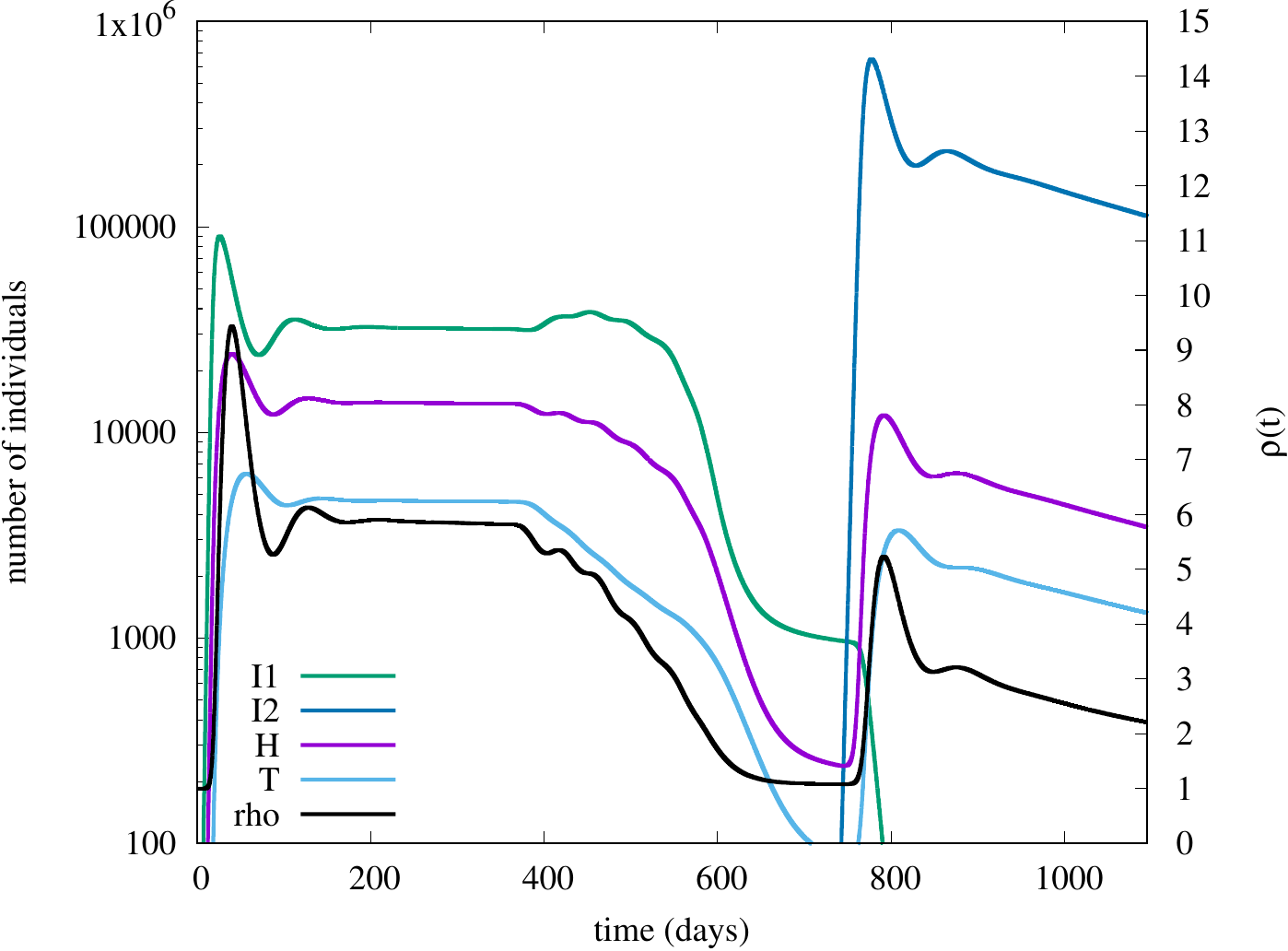}&
			\includegraphics[width=0.62\columnwidth]{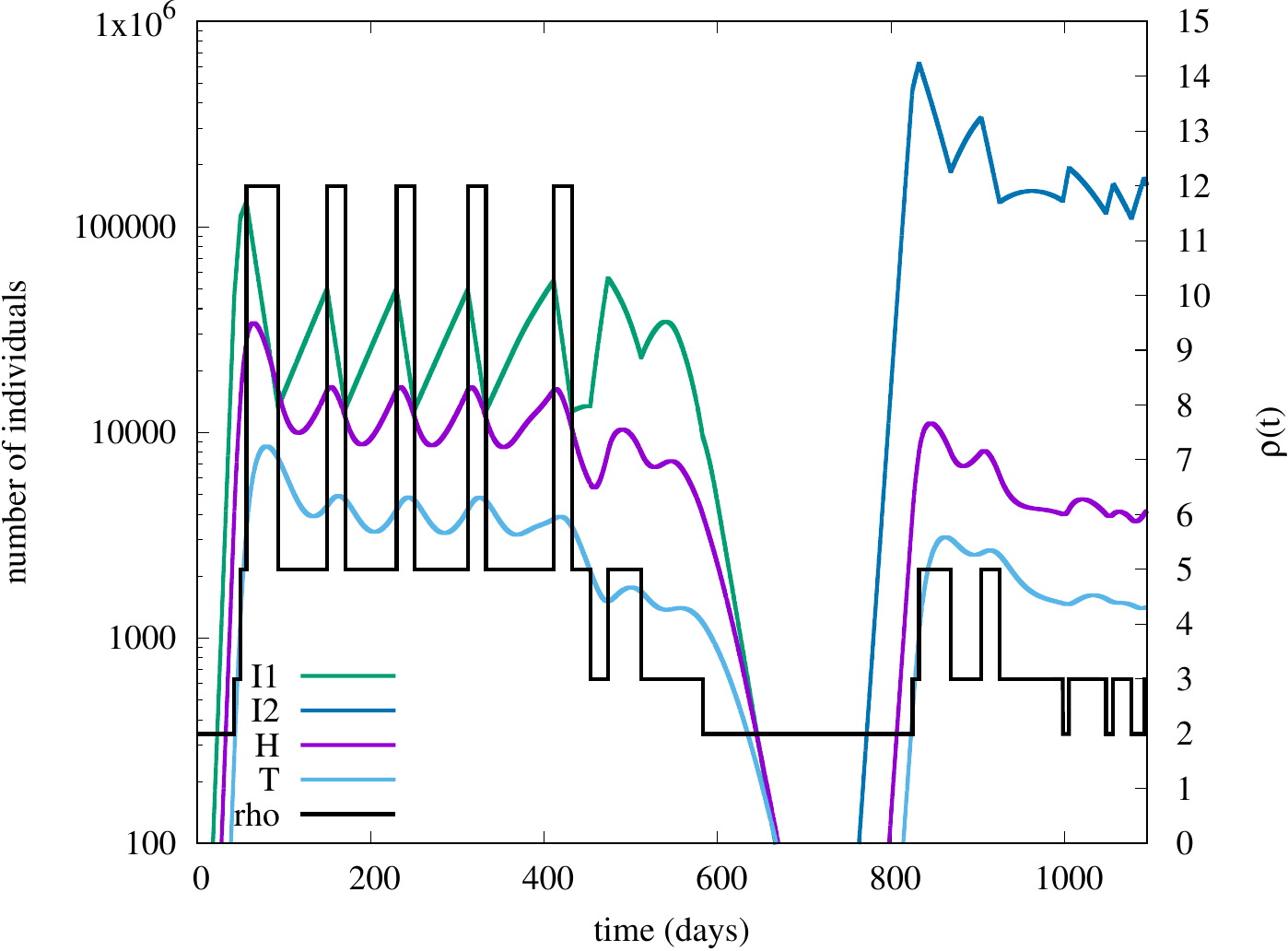}
			\\
			(a) Rate Control, $ \lambda_C=4000$.&
			(b) HT control: $T_{\max}=\widehat{T}=20k$;&
			(c) {Imperfect} HT control.\\
			& $H_{\max}=40k$, $\widehat H=100k$. &
		\end{tabular}
		\caption{Evolution of $I(t)$, $H(t)$, $T(t)$ (left $y$ axes) 
			and $\rho(t)$ (right $y$ axes) in the comprehensive  scenario.} 
		\label{fig:comprehensive}
	\end{figure*} 
	
	Figure \ref{fig:comprehensive}(a) and Figure \ref{fig:comprehensive}(b) report the evolution of the metrics, respectively, for the case in which the control is on the rate of new infected ($\lambda_C=4,000$)  and the HT (with $T_{\max}=20k$ and $H_{\max}=40k$).   Parameters have been set so that the two controls operate around approximately the same operational point during the first year.
	
	Rate control appears more reactive in the early phase of the epidemic. As already observed, due to its intrinsic delay in the control ring,  HT control exhibits some initial oscillations, which are not observable when rate control is applied. Therefore, it should not be surprising that rate control leads to better performance indices at the end of the first year, as shown in Table  \ref{table2}. \textcolor{black}{Note that costs are expressed in arbitrary units, while deaths are expressed in thousands.}
	However, when the second variant starts spreading,  the rate-control strategy may overreact, forcing the system to work in over-restricted conditions for quite a long time (note that at the end of the three-year period, rate-control is far from being completely relieved).
	Instead, HT control can automatically adjust its operational point as an effect of the mutated environmental conditions, i.e., a smaller intrinsic lethality index of the variant and a significant fraction of vaccinated individuals who are protected against severe outcomes.
	
	\textcolor{black}{We remark that these strategies, which tightly and precisely control either the infection rate or the hospitalization/ICU occupancy, are hardly implementable. However, they provide valuable insights. To shed light on more practical controls, we examine an implementable rough version of the HT control,  denoted as Imperfect HT (IHT).} Figure \ref{fig:comprehensive}(c) shows the evolution of the epidemic when the IHT strategy is adopted. In this case, the control dynamically selects the current alert level from the following finite set {\it green, white, yellow, orange, red, purple}. A different set of non-pharmaceutical restrictions corresponds to every alert level, determining a corresponding value of $\rho(t)\in \{1,2,3,5,12,15\}$  (note that intermediate values of $\rho(t)$  corresponding to different alert levels, do not need to be perfectly known).
	Every week a simple threshold mechanism is implemented to establish the current alert level for the following week,  with normalized thresholds (with respect to $H_{\max}$ or  $T_{\max}$) set respectively to $\{ 0.01, 0.1, 0.2, 0.4, 1.0 \}$.  Any alert level must be maintained for at least three weeks before it can be decreased. Despite the behavior of IHT  does not significantly deviate from HT, a high extra economic cost is paid for the effect of unavoidable oscillations between consecutive alert levels, especially for large values of $\alpha$.

	\begin{table}
		\centering
		\caption{Comparison of control strategies in a  comprehensive scenario: ecomomic costs and deaths}
		\label{table2}
		\scalebox{0.8}{%
			\begin{tabular}{|c|ccc|c|ccc|c|}
				\hline\hline
				&  \multicolumn{4}{c|}{1st year}  &  \multicolumn{4}{c|}{ three years } \\
				&  \multicolumn{3}{c|}{ Cost}   &  Deaths& \multicolumn{3}{c|}{ Cost}  &  Deaths\\
				\hline
				$\alpha$ 	& 1 & 2 & 3  &   & 1 & 2 & 3 &    \\
				\hline
				\text{rate}   &  1.68 &  8.18 & 39.9       & 41.1 	  &  3.35  & 13.5 &  58.0    &  70.5    \\
				\text{HT
				} &   1.76 &  9.22 &   50.2 &   45.8     & 3.19  & 13.4  &     64.6     &   78.8  \\
				\text{IHT} &  2.02  & 15.5&   146 & 40.6   & 3.71& 22.1  &  187  &   68.7 	  \\
				\hline
				\hline 
		\end{tabular}}
	\end{table}

	\medskip
\section{Conclusions}\label{sect:conc}
	{\color{black}{Drawing inspiration from the COVID-19 pandemic, our study utilized the standard compartmental approach of mathematical epidemiology, incorporating a data-informed population stratification. The objective was to investigate the planning problem associated with implementing pharmaceutical and non-pharmaceutical interventions in order to minimize both economic costs and deaths within a country-sized community.
			
			Through our analysis, we discovered that control strategies based on either the infection rate or the current load of hospitalizations/ICU can be highly unstable, particularly in non-ideal and realistic conditions. Even under ideal and stable conditions, the complexity of the problem is evident due to the intricate interplay of multiple factors and the influence of various parameters that are often unknown to policymakers.
			
			Our study emphasizes the need for more robust control strategies that can effectively tackle future pandemics. The complexity of the problem, combined with the uncertainty surrounding key parameters, highlights the importance of designing control strategies that are resilient and adaptable in the face of unforeseen challenges. By employing a data-informed population stratification and considering both pharmaceutical and non-pharmaceutical interventions, our work contributes to the broader understanding of epidemic planning.
	}}


\appendices
	
\section{Motivation}

	This appendix complements the main article with additional details and discussion. {\color{black} We begin with a review of some basic epidemiological models and provide a non-comprehensive but detailed literature overview of epidemiological models. The main article provides a more concise yet comprehensive literature review. Then,} we describe in more detail the structure of the proposed model, {\color{black}present its various formulations, and} highlight the main distinguishing features from existing compartmental models describing the dynamics of COVID-19. 
	We also explain how we obtained the required distributions to be fed into the model  {\color{black}from real data. We also discussed the choice of parameters used to describe the dynamics of the COVID-19 epidemic.}
	We may repeat some concepts and formulas introduced in the main article to make the appendix self-contained.

\section{Base model}
	{\color{black}\subsection{SIR-like models - A brief literature review}}
	The so-called SIR model \cite{KermackMcKendrick1927} is paradigmatic in epidemiology, and it has been widely adopted to model infectious diseases for which recovered individuals acquire lasting or at least sufficiently durable immunity. In particular, the SIR model and its extensions have proven useful in modeling the dynamics of epidemic diseases such as seasonal influenza \cite{SIR_seasonal_flu} and swine flue \cite{Aldila2014}. Since the outbreak of COVID-19 in late 2019, it has been an effective tool for studying the spread of the novel coronavirus. This model is the prototype of a broader class of models that partition the population according to disease status, called compartmental models \cite{Hethcote2000}. One of the keys to the success of the SIR model is its simplicity, considering only three compartments: Susceptibles {\color{black}$S(t)$}, Infected {\color{black}$I(t)$} and Removed {\color{black}$R(t)$}, {\color{black} whose dynamics are described by the following system of ordinary differential equations:}
	{\color{black}
		\begin{align}
			\frac{dS(t)}{dt} &= -\beta \cdot \frac{S(t) \cdot I(t)}{N} \nonumber\\
			\frac{dI(t)}{dt} &= \beta \cdot \frac{S(t) \cdot I(t)}{N} - \gamma \cdot I(t) \\
			\frac{dR(t)}{dt} &= \gamma \cdot I(t) \nonumber
		\end{align}
	}
	{\color{black}$N$ in the system of equations represents the total number of individuals in the population. $\beta$ is a fundamental parameter and indicates the average number of contacts per person per time. This factor multiplies the term $\frac{S(t)I(t)}{N}$, which is linked with the probability of a virus transmission event (i.e., an infectious individual infects a susceptible one), assuming homogeneous mixing of the population. The parameter $\gamma$ indicates the rate at which an individual exits the infectious state, either by recovering from the disease or dying. These \textit{compartmental} models are better represented by block diagrams which highlight the transitions among states. For example, Figure \ref{fig:SIR_SEIR_SIMD} depicts the SIR model (A) and some of its extensions (B-C) which we discuss in the following.}
	
	\begin{figure} [h!]
		\centering
		\includegraphics[width=.8\linewidth]{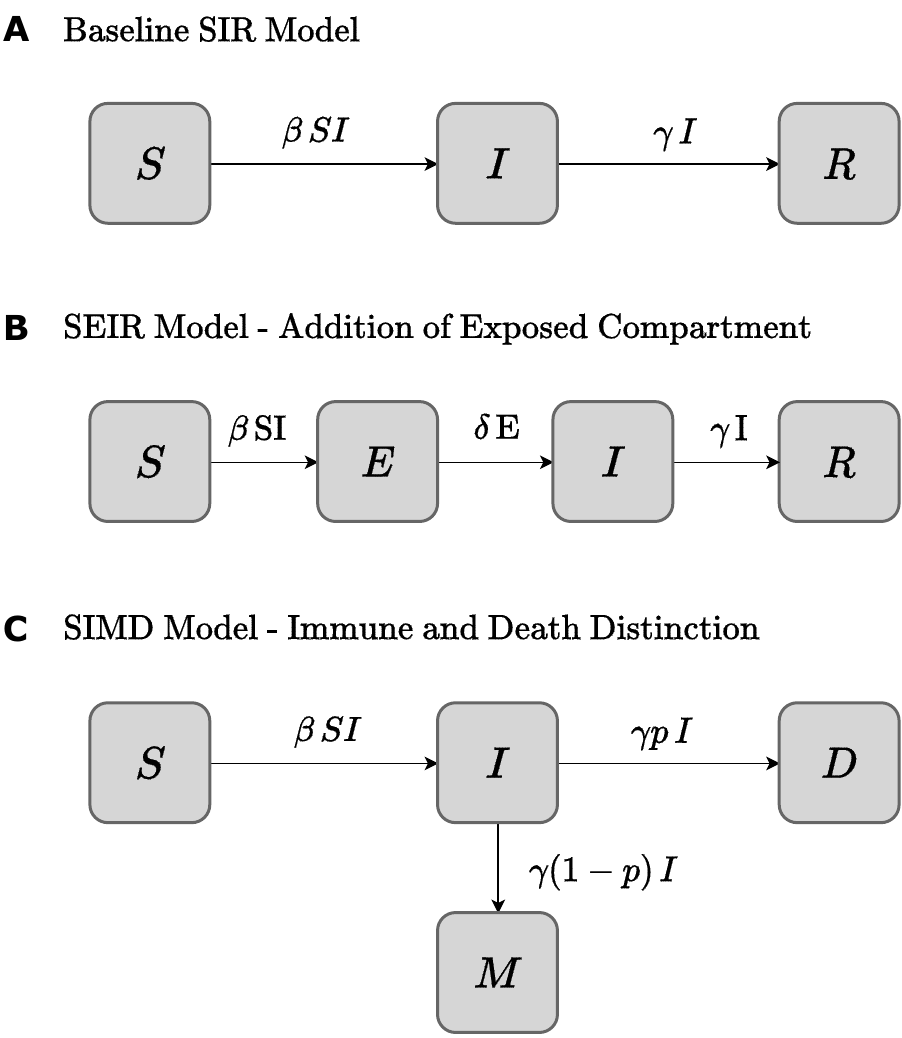}
		\caption{Diagrams highlighting the transitions among compartments for the SIR, SEIR and SIRD models introduced in this section.}
		\label{fig:SIR_SEIR_SIMD}
	\end{figure}
	
	However, the model's simplicity comes at the expense of oversimplifying the complexities of the disease processes \cite{JAMA_Tolles2020}. For example, it assumes a homogeneous mixing of individuals and therefore does not consider the correlation between daily contacts and {\color{black}other specific characteristics of the population (e.g., age).}
	In \cite{Vespignani_contacts}, age-specific contact patterns are derived for different countries, which allows for a more accurate description of the population's interactions. The population is no longer considered uniform but grouped according to the individual's age. Moreover, the population is assumed to be closed, with no in-migration or out-migration. {\color{black}Accounting for this would not be complicated. It is sufficient to have information regarding a particular population's birth and death rates, together with immigration (individuals coming to the population from another population) and emigration information. In most cases, these effects almost balance out, leaving the overall number of individuals approximately the same. Indeed, the} closed population assumption is reasonable as long as the time horizon is not too long (a few years).
	In addition, the SIR model does not account for the period during which an individual has been exposed to the virus but does not yet have sufficient infectious levels for transmission to others. A straightforward extension is to {\color{black}incorporate an additional compartment}, usually denoted with E (standing for Exposed), which defines what is known as the SEIR model \cite{LIN2020211}, Figure \ref{fig:SIR_SEIR_SIMD} (B). {\color{black}This only represents a delay in the dynamics for those individuals who have been in contact with the virus resulting in an infection and will, after $\frac{1}{\delta}$ units of time on average, become infectious for others.} Another straightforward extension of the basic SIR model is {\color{black}is to distinguish in the \textit{Removed} compartment between those who die as a result of the virus (Deaths) and those who acquire immunity against the virus (iMmune). Thus the model is referred to as SIMD, Figure \ref{fig:SIR_SEIR_SIMD} (B). It is sometimes called SIRD in the literature, depending on the convention used to indicate immune individuals, i.e., $M$ from iMmune or $R$ from Removed. The model extension is relatively simple and only relies on knowing the probability $p$ with which an infected individual dies after infection. As it is clear from the block diagram, this is achieved by simply multiplying the rates by the \textit{probability of dying}.} Models that account for both aspects have long been used in the literature. Much of the most recent works on \mbox{COVID-19} \cite{10.1371/journal.pone.0230405}\cite{9143144} use SIR-like models and even adds additional compartments, such as one for asymptomatic individuals, i.e., infected individuals who do not manifest any COVID-19-specific symptoms{\color{black}, and are potentially more dangerous than symptomatic individuals who may reduce their contacts rate due to the insurgence of symptoms}.
	
	\begin{figure} [h!]
		\centering
		\includegraphics[width=0.85\linewidth]{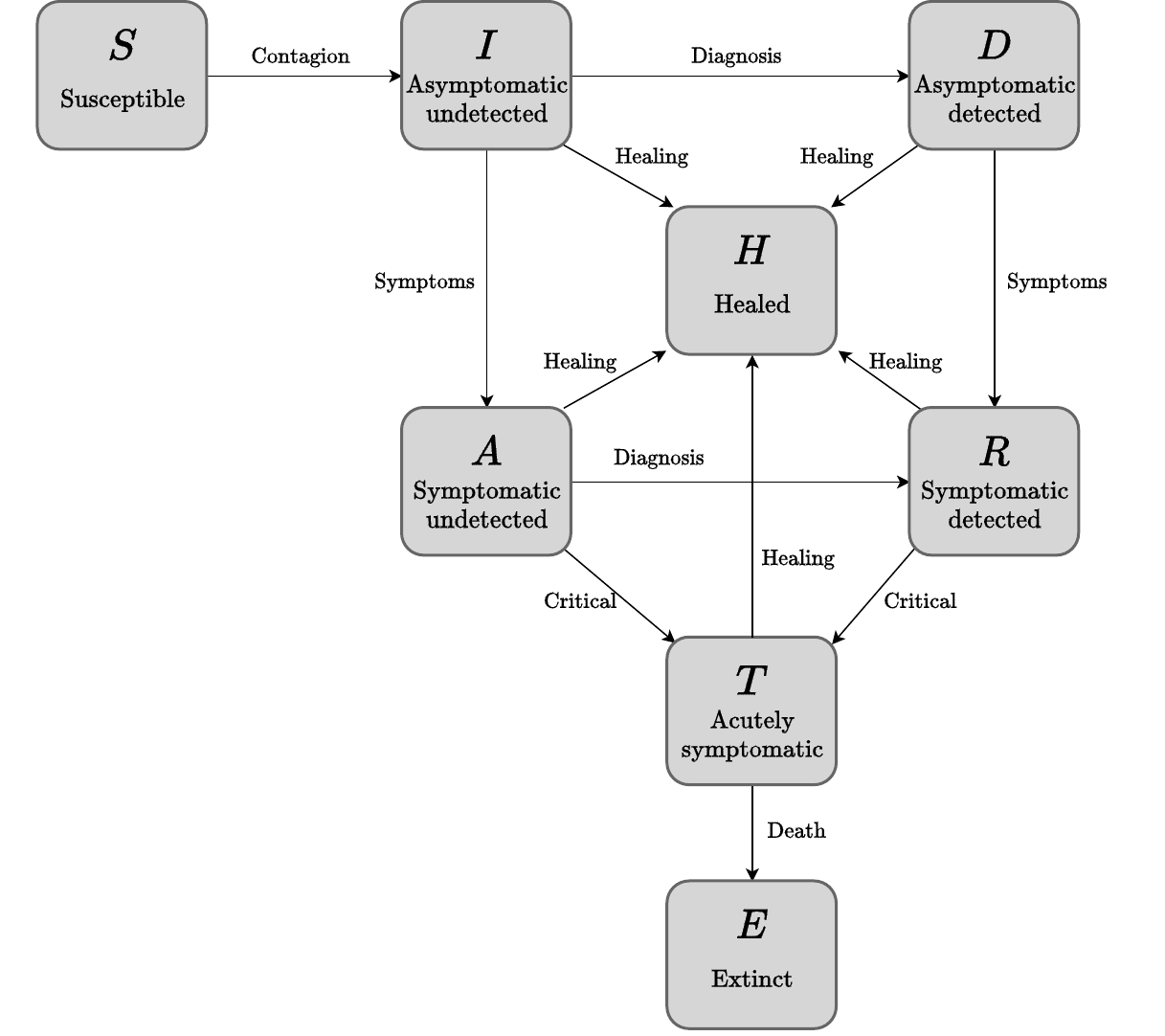}
		\caption{\textcolor{black}{Schematic representation of the SIDARTHE model, similar to that presented in \cite{Giordano2020ModellingTC}. We did not report the transition rates as we do not show the model's equations, the word description over the arrows provides more intuition.}}
		\label{fig:SIDARTHE}
	\end{figure}
	
	{\color{black}
		One of the most comprehensive models recently appeared in the literature to model COVID-19 \cite{Giordano2020ModellingTC} considers eight states that distinguish between detected and undetected infectious cases, varying severity of illness (symptomatic and asymptomatic cases), non-life-threatening cases and potentially life-threatening cases requiring ICU admission. It is called the SIDARTHE model, and we present its block diagram in Figure \ref{fig:SIDARTHE}. The states are indeed susceptible (S), infected (I), diagnosed (D, which represents detected asymptomatic cases ), ailing (A, more severe, i.e., symptomatic, cases which have not been detected), recognized (R, symptomatic cases detected), threatened (T, acute symptomatic detected), healed (H) and extinct (E, death as a result of severe infection). This model has been purposefully developed to capture the peculiarities of COVID-19 (e.g., the distinction made between detected and undetected cases) and fits well data related to COVID-19 spread in Italy. With respect to the model we present in the next section, this approach does not consider heterogeneity in the population and the epidemiological states are specific for COVID-19, our approach aims at being more general. 
		
		Another fundamental aspect that has not been considered in the models above is the effect of vaccinations on the dynamics. In the case of COVID-19, vaccines started to become available towards the end of 2020, and great debate sparked around how vaccines should be prioritized. Our model also aims to answer this question and provide valuable insights into the possible trade-offs the decision-maker could face. Clearly, in the literature, there already exist models that consider vaccinations. For example \cite{Bubar2021} distinguishes between vaccinated who are protected by the vaccine, vaccinated without protection against COVID-19, and unvaccinated because of a positive serotest (note that upon recovery, infected individuals acquire immunity towards the virus) or refusal to vaccinate (i.e., no-vax). We consider these aspects in our model in one of its extensions presented in the next section and depicted in Figure~\ref{fig:our_model_vax}.
		
	}
	
	\begin{figure}[h!]
		\centering
		\includegraphics[width=.7\columnwidth]{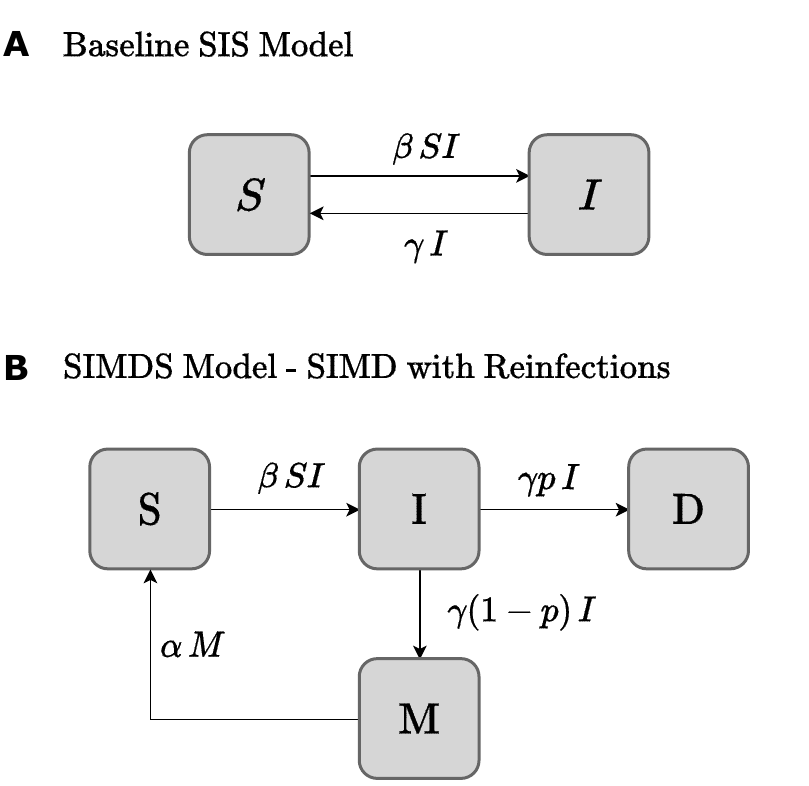}
		\caption{Schematic representation of models which consider reinfections, the classical SIS and a SIMDS model.}
		\label{fig:SIS_SIMDS}
	\end{figure}
	
	One final observation: the SIR (and SIR-like) underlying assumption of long-lasting immunity does not apply to all viruses. In the case of SARS-CoV-2, it is not yet clear how long immunity to the virus will last. Townsend et al. \cite{Townsend2021} claim that \lq\lq reinfection by SARS-CoV-2 under endemic conditions would likely occur between 3 months and 5 years  after peak antibody response". The susceptible-infected-susceptible (SIS, see block diagram in Figure \ref{fig:SIS_SIMDS} (B)) model is the most simple model when reinfection comes into play. Namely, infected individuals do not acquire immunity to the virus and move back into the susceptible compartment upon recovery. {\color{black}The SIS model assumes that no protection agains the virus is acquired upon infection and, at the end of the acute phase, an individual could be \textit{immediately} be reinfected again. The SIMDS model, which is a refined version of the SIS model and extends the SIMD model, encompass the possibility of a \textit{temporary} immunity against the virus. This is captured by the $M$ state (iMmune), where individuals do not contract the virus due to the obtained immunity. After an exponential time with parameter $\alpha$, an immune individual looses its protection and returns into the susceptible state. This framework seems to be more appropriate in the context of COVID-19, since it appears that the acuired immunity is only temporary.}
	
	\subsection{Modeling objectives}
	
	{\color{black}
		The epidemiological model we propose aims to provide a flexible framework that is general enough to be used for a large set of pandemics and also includes all the necessary ingredients discussed above, which are paramount to proposing and assessing both non-pharmaceutical and pharmaceutical interventions. We consider COVID-19 as a use case of the model but argue that it can represent other pandemics by appropriately setting the parameters. The proposed model considers different levels of severity of the disease, considering \textit{hospitalized} individuals ($H$) and those who need to be under intensive care ($T$). It considers the possible loss of immunity, with a rate of $\mu$, and, in its most general version (see Section \ref{sec:full_model}), the vaccination process, distinguishing between individuals who have undergone \textit{partial} and \textit{complete} vaccination. One of the distinguishing features of our model lies in the distribution $f_{r,p}$. It characterizes population heterogeneity in an even broader sense than what we will later use to construct the distributions from actual data. This distribution relates the \textit{risk exposure} $r$ and the \textit{mortality rate} $p$. While we will consider the risk exposure mainly as a function of age, this concept could incorporate other aspects such as occupation, personal habits, existing health conditions, etc. This choice depends on the data availability and the objective of the study. Indeed, in contrast to previous work \cite{Goldsteine21}, \cite{galvani09}, \cite{matrajt21}, \cite{Sandmann2021}, \cite{monod21}, \cite{SaadRoy2021},\cite{prio21}, we propose a simple modeling framework that explicitly represents the heterogeneity of risk exposure across different population segments.
		Moreover, our modeling framework describes a pandemic in a tightly controlled setting, in which non-pharmaceutical (e.g., social distancing, lockdowns) and pharmaceutical (i.e., vaccinations) interventions are naturally present in the framework to describe a pandemic in all its various phases effectively.
	}
	
	\subsection{Proposed dynamics}\label{sec:model}
	
	We propose a new compartmental model enriched by: i) population heterogeneity in terms of mortality rate and risk exposure, ii) closed-loop control of the epidemiological curve, and iii) progressive vaccination of individuals.
	We start by describing the base version of our compartmental model
	to describe the spread of a disease in a non-homogeneous population of size $N$ in the absence of any intervention (either pharmaceutical or non-pharmaceutical).
	
	We consider six epidemiological states {\color{black} for each population class $(r,p)$}: let $S_{r,p}(t)$, $I_{r,p}(t)$, $M_{r,p}(t)$, $H_{r,p}(t)$, $T_{r,p}(t)$, and $D_{r,p}(t)$ denote the number of individuals characterized by $(r,p)$ who at time $t$ are, respectively, susceptible, infected, immune, hospitalized, under intensive treatment and dead{\color{black}.The variable $r$represents the risk exposure:}
	individuals with high $r$ have a higher risk of infection and a higher probability of transmitting the disease. While $p$ represents the mortality rate and is associated with the \textit{vulnerability} to the considered disease of a certain segment of the population.
	
	The amount of time spent by an individual in the infected, hospitalized, intensive therapy, immune compartment is exponentially distributed 
	with mean value $1/\gamma$, $1/\phi$, $1/\tau$, $1/\mu$, respectively.
	
	The system dynamics are described by the following set of ordinary differential equations:
	\begin{align}\begin{split}\label{eq:dyn_sys_sm}
			\dot S_{r,p}(t) & = -{\sigma(t)}
			\left( \sum_{r',p'} r' I_{r',p'}(t) \right)
			\frac{r S_{r,p}(t)}{\sum_{r',p'} r' N f_{r',p'}} + \mu M(t) \\
			\dot I_{r,p}(t) & = {\sigma(t)}
			\left( \sum_{r',p'} r' I_{r',p'}(t) \right)
			\frac{r S_{r,p}(t)}{\sum_{r',p'} r' N f_{r',p'}}
			- \gamma I_{r,p}(t)  \\
			\dot H_{r,p}(t) & = \gamma p_{r,p}^{IH}  I_{r,p}(t) - \phi H_{r,p}(t)  \\
			\dot T_{r,p}(t) & = \phi p_{r,p}^{HT} H_{r,p}(t) -\tau T_{r,p}(t) \\
			\dot D_{r,p}(t)   &= \tau p_{r,p}^{TD}(t) T_{r,p}(t) \\
			\dot M_{r,p}(t)  &= 
			\gamma  (1-p_{r,p}^{IH}) I_{r,p}(t) 
			+ \phi (1-p_{r,p}^{HT}) H_{r,p}(t) \\
			& + \tau (1-p_{r,p}^{TD}(t)) T_{r,p}(t)
			-\mu M(t)
		\end{split}
	\end{align}
	
	where 
	$\sigma(t) \geq 0$
	is a function representing all exogenous (uncontrolled) factors that change the strength of the infection, such as seasonal effects.
	In this paper, we will assume for simplicity that $\sigma(t) = \sigma$ 
	is constant. The total number of susceptible people is $S(t) = \sum_{r,p} S_{r,p}(t)$, and, similarly, we introduce the total number of people in the other compartments: $I(t)$, $H(t)$, $T(t)$, $M(t)$, $D(t)$.
	Probabilities $p_{r,p}^{IH}$, $p_{r,p}^{ HT }$ and $p_{r,p}^{ TD }(t)$
	denote the probability that an individual of type $(r,p)$ moves between the two compartments indicated in the apex.
	We make probability $p_{r,p}^{ TD }(t)$ depend on $T(t)$, i.e., on the instantaneous total number of people in ICUs, since the death probability dramatically increases when ICUs are saturated.
	Denoted with $\widehat T$ the number of available ICUs,
	when $T(t) \le \widehat{T}$, the overall death probability of an infected person is assumed to be equal to $p$: 
	\begin{equation}\label{eq:prodp_sm}
		p_{r,p}^{IH} \cdot p_{r,p}^{HT} \cdot \hat{p}_{r,p}^{TD} = p \qquad 
		\text{if }T(t) \le {\widehat T}, 
	\end{equation} 
	where $\hat{p}_{r,p}^{TD}$ is the probability
	to transit from state $T$ to state $D$ in \textit{normal} conditions, i.e., when
	$T(t) \le 
	{\widehat{T}}$. Therefore, $p_{r,p}^{TD}(t) = \hat{p}_{r,p}^{TD}$ as long as
	$T(t) \le 
	{\widehat{T}}$.   
	
	\begin{figure*}[h!]
		\centering
		\includegraphics[width=0.7\textwidth]{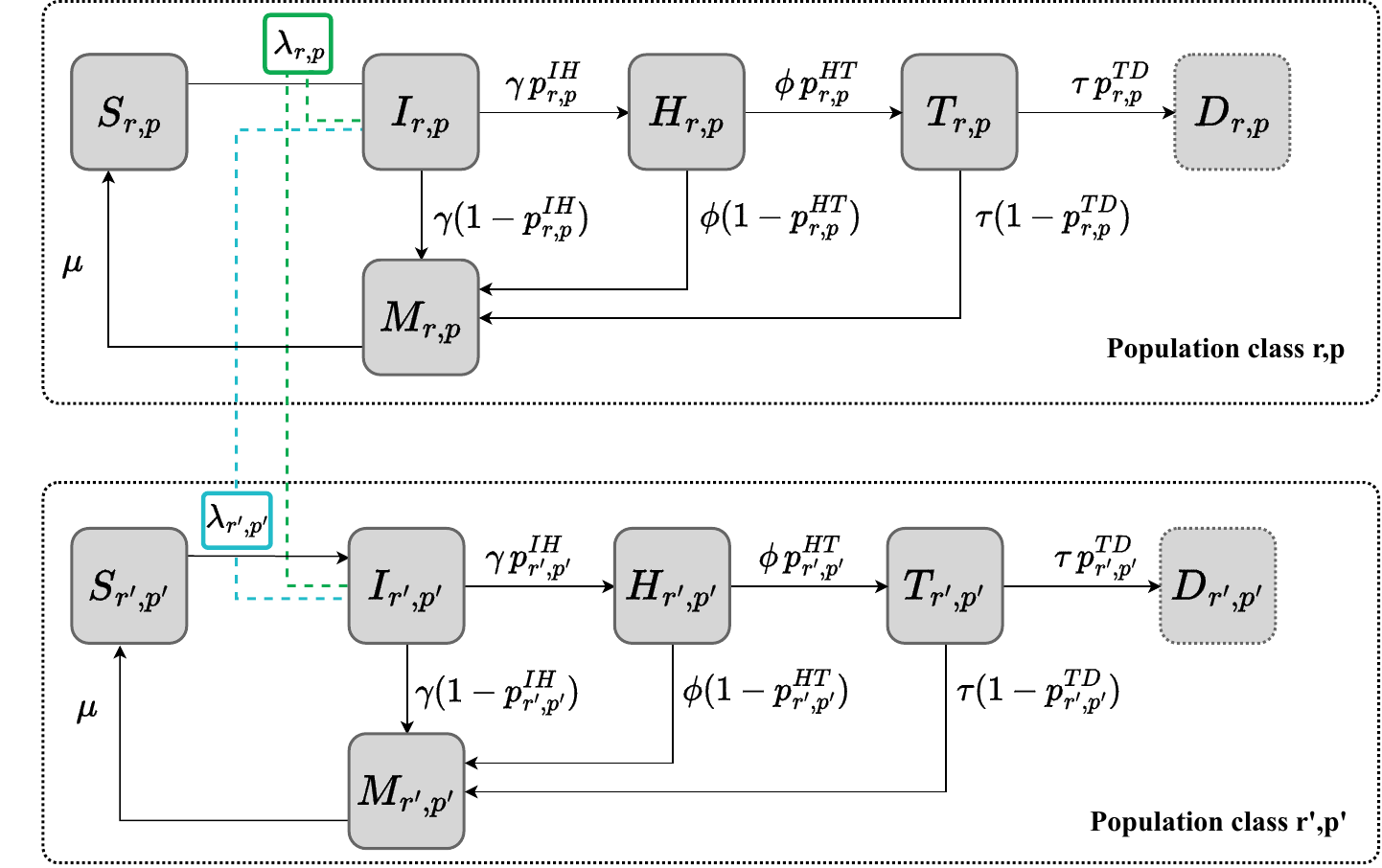}
		\caption{\textcolor{black}{Schematic representation of the proposed model without vaccination and \lq\lq additional" transitions from the states $I$ and $D$ to $H$. In the figure it si highlighted that the \textit{infection rates} in different population classes are dependent on each other (see system of equations (\ref{eq:dyn_sys_sm})).}}
		\label{fig:our_model_sm}
	\end{figure*}
	
	When $T(t) > 
	{\widehat{T}}$, we assume that the death probability of people who cannot receive the required intensive therapy is increased by a factor $\theta$,
	hence $p_{r,p}^{ TD }(t)$ is dynamically adjusted as follows:
	\begin{multline}
		p_{r,p}^{TD}(t) = \hat{p}_{r,p}^{TD} \frac{ {\widehat{T} }   } {T(t)} +  
		\min\{1,\theta \cdot \hat{p}_{r,p}^{TD}\} 
		\frac{T(t)-{\widehat{T}} }  {T(t)} .
			\label{eq:pTD_sm}
	\end{multline}

We consider the case in which the individuals might lose immunity with rate $\mu$, thus becoming susceptible again. 
It should be noticed that the mass preservation 
$\dot S_{r,p}(t) +\dot I_{r,p}(t) +\dot M_{r,p}(t) + \dot H_{r,p}(t) + \dot T_{r,p}(t) + 
\dot D_{r,p}(t)=0 $ is satisfied for all $t\geq0$.

{\color{black}
	The presented model is an extension of the standard SIR model, and this observation can become more evident by looking at the block diagram in Figure \ref{fig:our_model_sm}. To make it even more straightforward, it is possible to rewrite the first two equations in the system of equations \eqref{eq:dyn_sys_jump} defining $\beta_{r\leftrightarrow r'}~:=~\sigma(t) \frac{r'r}{\sum_{r'',p''} r'' N f_{r'',p''}}$:
	\begin{align}\begin{split}\label{eq:dyn_sys_jump_new}
			\dot S_{r,p}(t) & = -\sum_{r',p'} \beta_{r \leftrightarrow r'} I_{r',p'}(t) S_{r,p}(t) + \mu M(t) \\
			\dot I_{r,p}(t) & = \sum_{r',p'} \beta_{r \leftrightarrow r'} I_{r',p'}(t) S_{r,p}(t)
			- \gamma I_{r,p}(t)  \\
		\end{split}
	\end{align}
	
	where the newly introduced parameter $\beta_{r\leftrightarrow r'}$ represents the pattern of interaction between population segments with risk exposure $r$ and $r'$. 
	
}

The total (uncontrolled) rate of new infections is equal to
$$\lambda_{\textcolor{black}{U}}(t) = \sigma(t) \left( \sum_{r,p} r I_{r,p}(t) \right)
\frac{\sum_{r,p} r S_{r,p}(t)}{\sum_{r,p} r N f_{r,p}}, $$

{\color{black}
	Or, by employing the interaction parameter $\beta_{r \leftrightarrow r'}$:
	\begin{equation*}
		\lambda_U(t) = \sum_{r,p} \sum_{r',p'} \beta_{r \leftrightarrow r'} I_{r',p'}(t) S_{r,p}(t)
	\end{equation*}
	\subsubsection{First extension - Adding transitions to the base model}
	One could also consider the possibility that infected people die without being hospitalized or undergoing intensive treatment. This extension requires the specification of transition probabilities $p_{r,p}^{ID}$ and $p_{r,p}^{HD}$ associated with direct transitions from state $I$ (state $H$) to state $D$, respectively, representing the occurrence of premature death events. Previous probabilities $p_{r,p}^{IH}$ and $p_{r,p}^{ HT }$ are then redefined as transition probabilities conditioned to the event that such premature deaths do not occur.
	
	The model is depicted in Figure \ref{fig:our_model_jumps} and the modified system dynamics would be:
	
	\begin{align}\begin{split}\label{eq:dyn_sys_jump}
			\dot S_{r,p}(t) & = -{\sigma(t)}
			\left( \sum_{r',p'} r' I_{r',p'}(t) \right)
			\frac{r S_{r,p}(t)}{\sum_{r',p'} r' N f_{r',p'}} + \mu M(t) \\
			\dot I_{r,p}(t) & = {\sigma(t)}
			\left( \sum_{r',p'} r' I_{r',p'}(t) \right)
			\frac{r S_{r,p}(t)}{\sum_{r',p'} r' N f_{r',p'}}
			- \gamma I_{r,p}(t)  \\
			\dot H_{r,p}(t) & = \gamma (1-p_{r,p}^{ID})p_{r,p}^{IH}  
			I_{r,p}(t) - \phi H_{r,p}(t)  \\
			\dot T_{r,p}(t) & = \phi (1-p_{r,p}^{HD}) p_{r,p}^{HT} H_{r,p}(t) -\tau T_{r,p}(t) \\
			\dot D_{r,p}(t)   & = \tau p_{r,p}^{TD}(t) T_{r,p}(t) 
			+ \gamma p_{r,p}^{ID} I_{r,p}(t) + 
			\phi p_{r,p}^{HD} H_{r,p}(t) \\
			\dot M_{r,p}(t)  & = 
			\gamma  (1-p_{r,p}^{ID})(1-p_{r,p}^{IH}) I_{r,p}(t) \\
			& + \phi (1-p_{r,p}^{HD})(1-p_{r,p}^{HT}) H_{r,p}(t) \\
			& + \tau (1-p_{r,p}^{TD}(t)) T_{r,p}(t)
			-\mu M(t)
		\end{split}
		\vspace{-2mm}
	\end{align}  
	
	Note that in this case we must assure that the overall 
	death probability $p$ satisfies (for $T(t) \le {\widehat T}$)  
	\begin{equation}\label{eq:prodpjump}
		p_{r,p}^{ID} + (1- p_{r,p}^{ID}) p_{r,p}^{IH} \left[
		p_{r,p}^{HD} + (1- p_{r,p}^{HD}) p_{r,p}^{HT} \hat{p}_{r,p}^{TD}\right] 
		= p 
	\end{equation} 
	
	which replaces \equaref{prodp}.
}

\begin{figure*}[h!]
	\centering
	\includegraphics[width=0.75\textwidth]{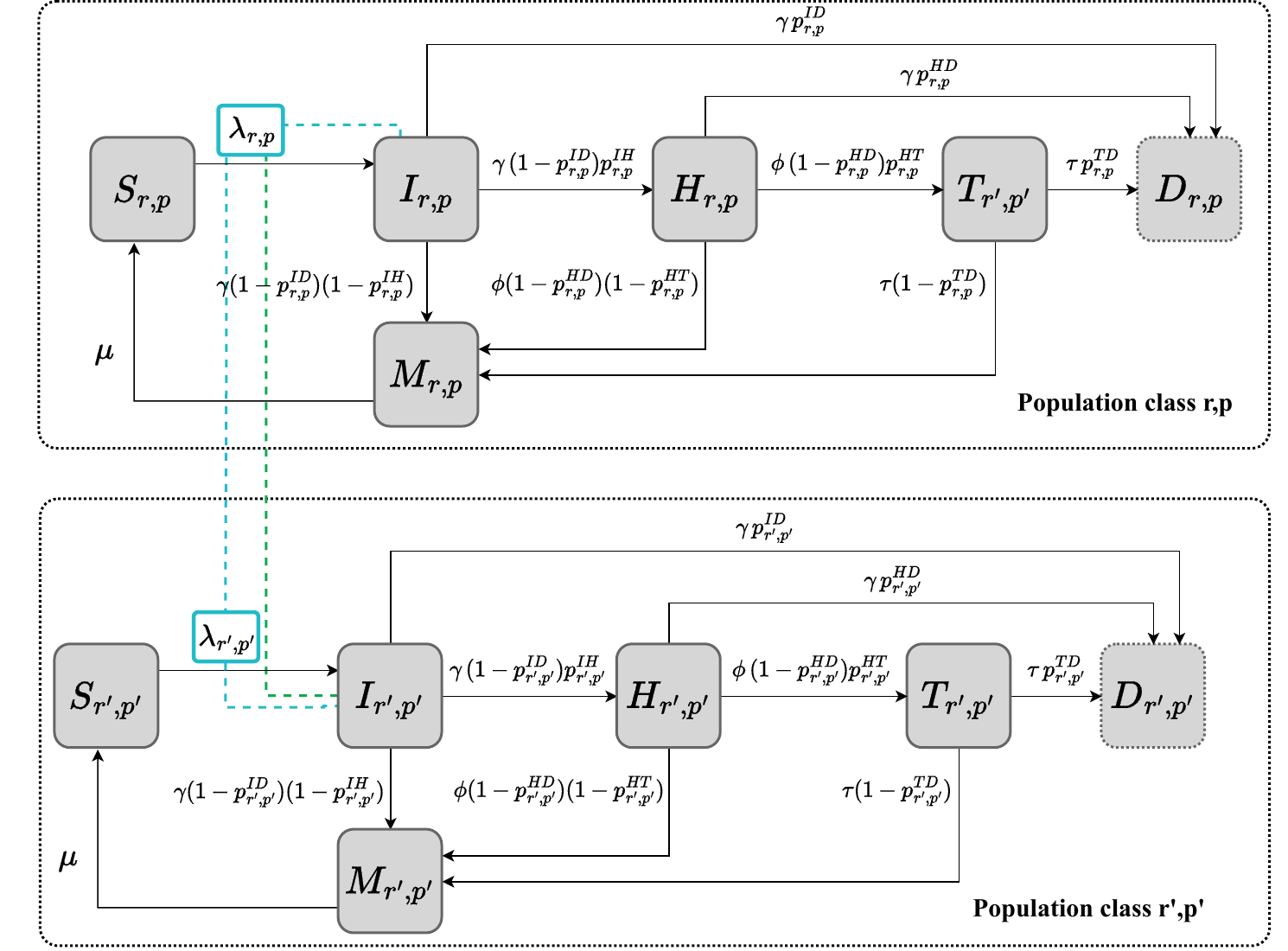}
	\caption{\textcolor{black}{Schematic representation of the proposed model without vaccination adding transitions from the \textit{Infected} ($I$) and \textit{Hospitalized} stets to the \textit{Death} ($D$) state, as detailed by the system of equations (\ref{eq:dyn_sys_jump}). This model is rather similar to that in Figure \ref{fig:our_model_sm}, only two transitions for each $\{r,p\}$ individual's class have been added.}}
	\label{fig:our_model_jumps}
\end{figure*}

{\color{black}
	As a final remark, we point out that while this extension is straightforward, it introduces the need for additional parameters which could be difficult to be determined. Moreover, models such as the SIDARTHE \cite{Giordano2020ModellingTC}, discussed above and explicitly developed for COVID-19, do not consider such transitions. Lastly, there is evidence that the majority of deaths due to COVID-19 have occurred in hospitals or care homes \cite{ONS2023}. Therefore, we employed the first version of the model in our simulations.
}

\section{Modeling population heterogeneity}
\label{secII}
\subsection{Data-driven Approach}
The model proposed in this work assumes the knowledge of some characteristics of the population involved in the epidemiological process. At the country level, populations present notable differences in their features, for instance, in terms of age distribution, overall health condition, or the daily contacts among individuals, which in turn depend on the country's customs and, more broadly, on the country's wealth and employment conditions. The heterogeneity in the population may play a role in the disease transmission process. Indeed, different contact patterns among individuals may result in a faster virus outbreak. To account for this heterogeneity, we characterized the population through a distribution function $f_{r,p}$ (equivalently $f_{r, a}$, where $a$ is the age, assuming a one-to-one correspondence between the age $a$ and $p$, see sec. \ref{sec:frp}) whose parameters represent the average daily contacts of an individual with other individuals in the population $r$ and the probability of dying $p$ as a consequence of the contraction of COVID-19. The parameters on which these variables depend are numerous (i.e., gender, occupation, individual habits, and pre-existing medical conditions, to name a few); a precise population characterization is outside this work's scope. We considered the individual's age as the main factor on which the average daily contacts $r$ and the probability of death $p$ depend. Thus, we stratified the population by age and derived country-specific distributions exploiting the contacts patterns presented in \cite{Vespignani_contacts} and the case fatality rate from \cite{CFR2020}, which have both been obtained for various age classes. The construction procedure of the function $f_{r,p}$ is discussed in depth in the following paragraphs.

\subsection{Chance of death due to COVID-19}\label{sec:p}

One of the characterizing parameters of our population is the probability of dying after infection $p$, for which it is crucial to have an estimation. To this end, we considered the Case Fatality Rate (CFR); see Table I for the values we used, which refer to Italy. The CFR represents the proportion of deaths due to a specific disease compared to the total number of people diagnosed with that disease in a certain period. The data we leveraged come from \cite{CFR2020} in which CFR values have been obtained for 10-year-wide age classes considering Italian COVID-19 epidemiological data by 18th August 2020.

An observation at this point is in order, the case fatality rate refers to \textit{confirmed} cases. In contrast, another popular indicator: the Infection Fatality Rate (IFR), considers the \textit{actual} number of infected cases. Since COVID-19 can give rise to infections with mild or no symptoms, the number of actual cases might be significantly higher than the confirmed cases. For these reasons, it has been debated that the CFR overestimates the probability of death. However, the actual number of cases is unavailable, and the studies considering the IFR relied on estimations. In this direction, it is worth mentioning the work of Ghisolfi et al. \cite{IRF_age_sex_etc} in which values for the IFR have been extrapolated accounting for age, gender, comorbidities, and health system capacity by discriminating between low-income and high-income regions. Considering the controversy around the CFR and IFR indicators, we decided to employ the CFR values to obtain an estimate for the mortality probability.

Moreover, for the characterization of our synthetic population, we have not considered other comorbidities (i.e., seropositive status, pre-existing medical conditions) nor gender differences (see \cite{IRF_age_sex_etc} for such differentiated values). Again, the scope of this work is not to faithfully describe the population but to capture some quantitative differences in population classes and retain the significant traits of a heterogeneous population. Many studies confirm the strong correlation between age and CFR (or IFR) \cite{CFR2020}\cite{IRF_age_sex_etc}\cite{AustraliaCFR}. The latter study considers data from 25th January to 10th December 2020 for the State of Victoria, Australia. In this period, the state experienced two waves of the virus, and by the end of the data series, the infection was eradicated from the Victorian population \cite{AustraliaCFR}. This observation made the writers conclude that their estimate of the CFR is not spoiled by the underestimation bias, which is customary for an ongoing outbreak. We report the observed CFR data in Table II. 
Even though the values for the CFR vary slightly among the cited studies, the trend is clear: older individuals are more \textit{fragile}, and the CFR drastically increases with age.

\begin{figure}[h!] 
	\centering
	\begin{minipage}[c]{0.45\linewidth} 
		\centering
		\begin{tabular}[t]{cr}
			\hline
			\hline
			Age Class & CFR \\
			\hline
			0-19 & 0.1\\
			20-29 & 0.1\\
			30-39 & 0.3\\
			40-49 & 0.9\\
			50-59 & 2.8\\
			60-69 & 10.9\\
			70-79 & 26.7\\
			$\geq$ 80 & 34.6\\
			\hline
			\hline
		\end{tabular}
		\caption*{Table I - Italy}
	\end{minipage}
	\begin{minipage}[c]{0.45\linewidth} 
		\centering
		\begin{tabular}[t]{cr}
			\hline
			\hline
			Age Class & CFR \\
			\hline
			0-9 & 0\\
			10-19 & 0\\
			20-29 & 0.02\\
			30-39 & 0.06\\
			40-49 & 0.04\\
			50-59 & 0.63\\
			60-69 & 2.16\\
			70-79 & 14.41\\
			80-89 & 31.90\\
			$\geq$ 90 & 40.03\\
			\hline
			\hline
		\end{tabular}
		\caption*{Table II - Australia}
	\end{minipage}
\end{figure}

The data from \cite{CFR2020} have been assumed as representative of the mortality probability for all the countries considered for the $f_{r,p}$ distribution construction. It is a simplification justified by the absence of an extensive and precise study concerning the case fatality rate at a country level. Figure \ref{fig:Italy_p} represents these data on the Case Fatality Rate as a function of age. The linear fitting of the empirical points is also shown in the figure, which allows the expansion of the number of age classes considered in the synthetic population.

\begin{figure}[h]\begin{center}
		\includegraphics[width=0.8\columnwidth]{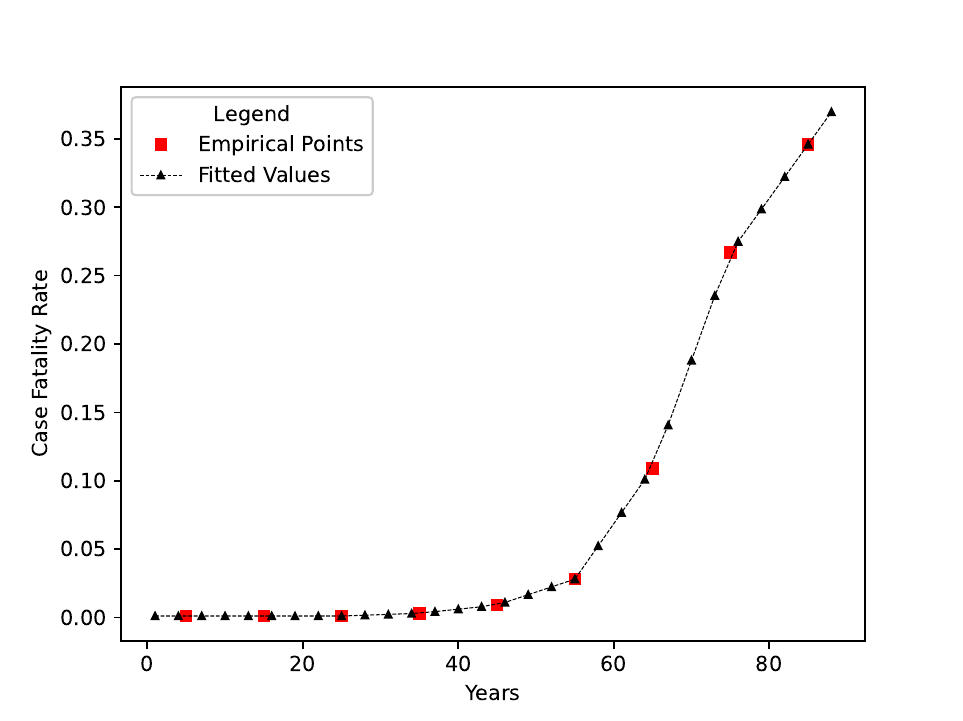}
	\end{center}
	\caption{Values for the Case Fatality Rate for the various age classes from \cite{CFR2020}. The graph shows the empirical points (in red) and the linearly fitted values of CFR in order to expand the number of available age classes for our population. The strong correlation between age and probability of dying is well depicted by this graph.}
	\label{fig:Italy_p}
\end{figure}

\begin{figure}[h]\begin{center}
		\includegraphics[width=0.48\columnwidth]{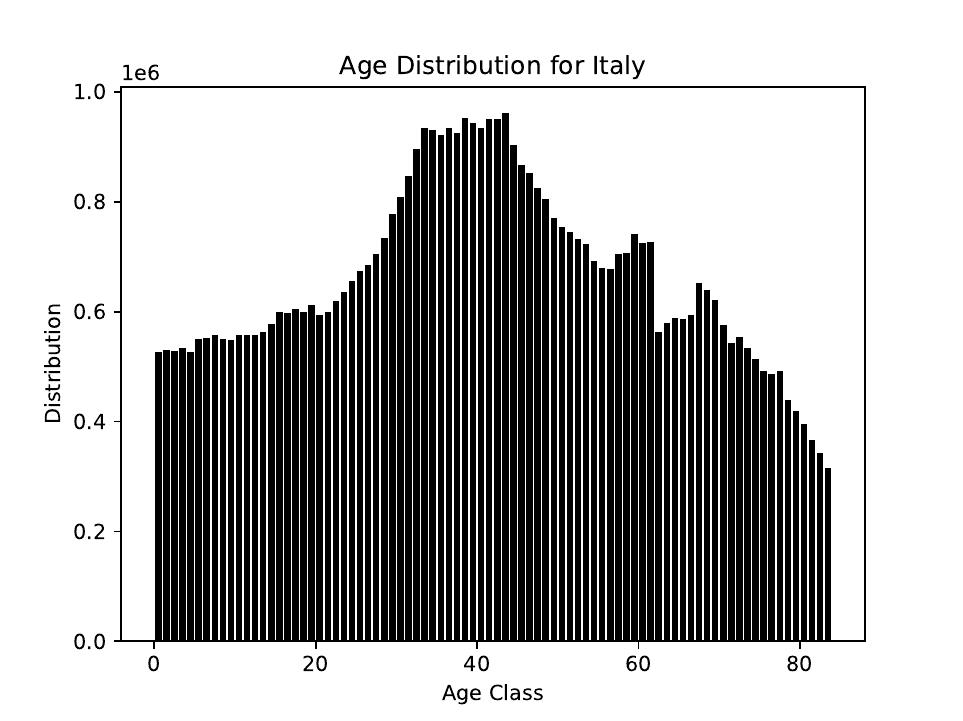}
		\includegraphics[width=0.48\columnwidth]{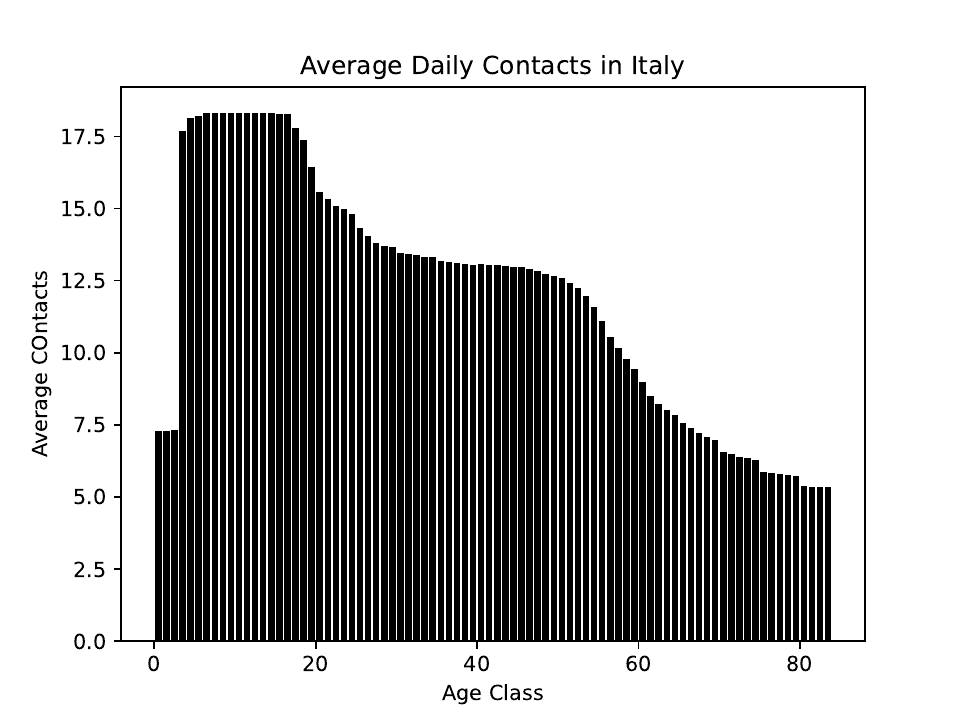}
	\end{center}
	\caption{Characterization of the Italian population in terms of the age distribution and the daily average contacts among individuals as reported in \cite{Vespignani_contacts}.}
	\label{fig:Italy_age_contacts}
\end{figure}

\subsection{Daily Number of Contacts}

\begin{figure*}[h!]
	\centering
	\includegraphics[width=\textwidth]{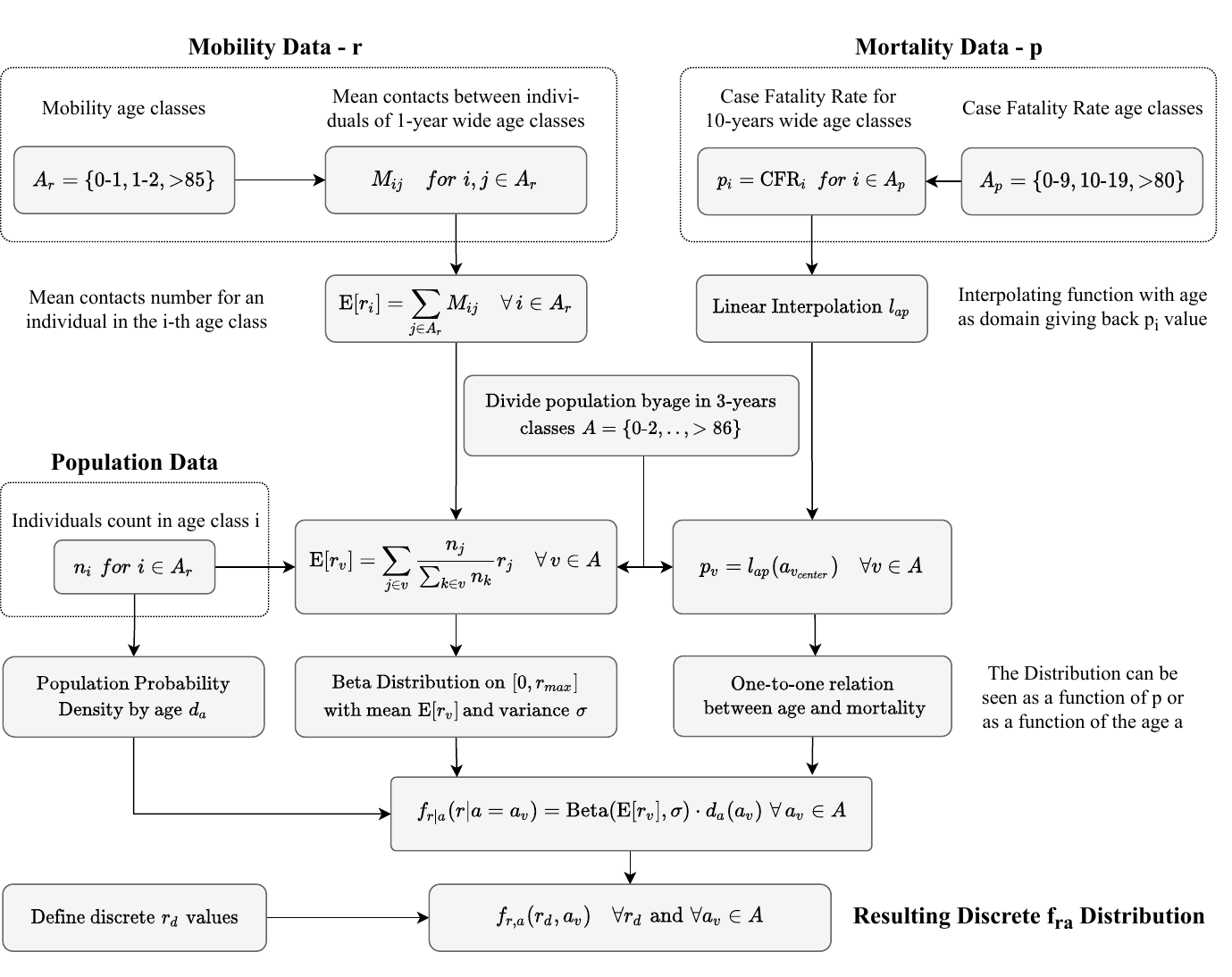}
	\caption{Pictorial representation of the construction's procedure for the synthetic distribution $f_{r,a}$. Starting from empirical data a discrete $f_{r,a}$ distribution is obtained for arbitrary values of the risk exposure $r$ and age classes $a$.}
	\label{fig:f_ra_flowchart}
\end{figure*}

Another crucial information for describing a population in epidemiological terms is the daily number of contacts of an individual. Indeed, an infection can occur when an infected individual encounters a healthy individual. The more individual-to-individual contacts, the higher the chance of an infection. In \cite{Vespignani_contacts}, age-specific contact matrices have been generated from detailed census and survey data on key-demographic features for 35 countries. Four different settings are considered (household, school, workplace, and community), producing as many specific contact matrices. A linear combination with appropriate weights of these matrices produces an overall contact matrix $\{M_{ij}\}$ which indicates the daily per capita number of contacts an individual of age $i$ has with individuals of age j (the interested reader is referred to \cite{Vespignani_contacts} for more details). For this work, the entire contact matrix $M$ is not necessary. We are only interested in the average number of contacts of an individual of age $i$ regardless of the age of the individual encountered. This value is easily computed from $M$ by summing over the columns of the matrix: $\mathbb{E}[r_i] = \sum_j M_{ij}$. In \cite{Vespignani_contacts} for each country considered in the study, the number of people belonging to 1-year-wide age classes is also provided, allowing for the definition of the age distribution $d_a(\text{age})$ of the population. Figure \ref{fig:Italy_age_contacts} reports the age distribution for Italy truncated at 84 years, together with the average daily contacts for each age class.

\begin{figure*}[h!]\begin{center}
		\begin{tabular}{ccc}
			\includegraphics[trim=1.9cm 0.6cm 1.9cm 0.8cm, clip, width=0.31\textwidth]{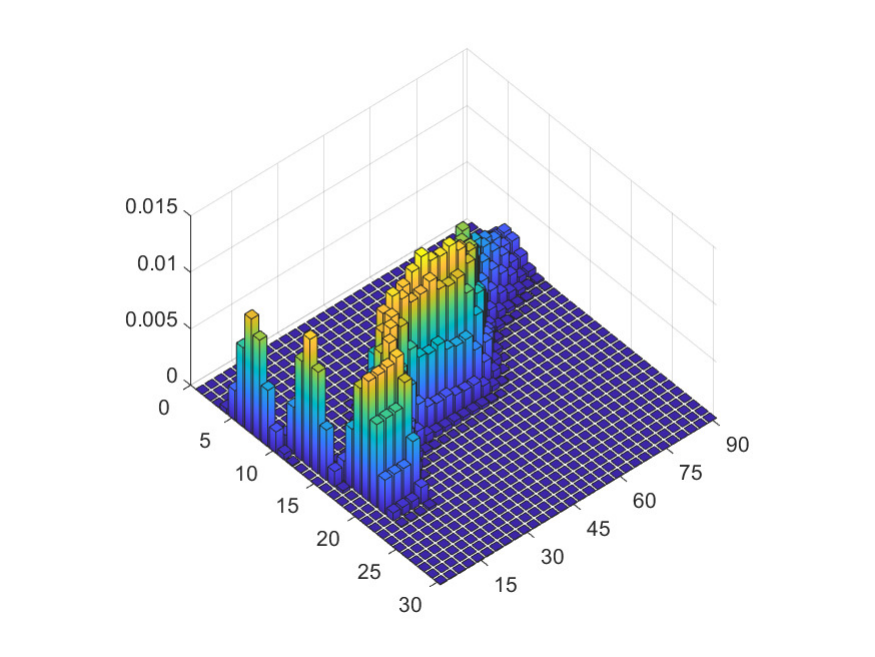}&
			\includegraphics[trim=1.9cm 0.6cm 1.9cm 0.8cm, clip, width=0.31\textwidth]{Australia_hist_3D_v2.pdf}&
			\includegraphics[trim=1.9cm 0.6cm 1.9cm 0.8cm, clip, width=0.31\textwidth]{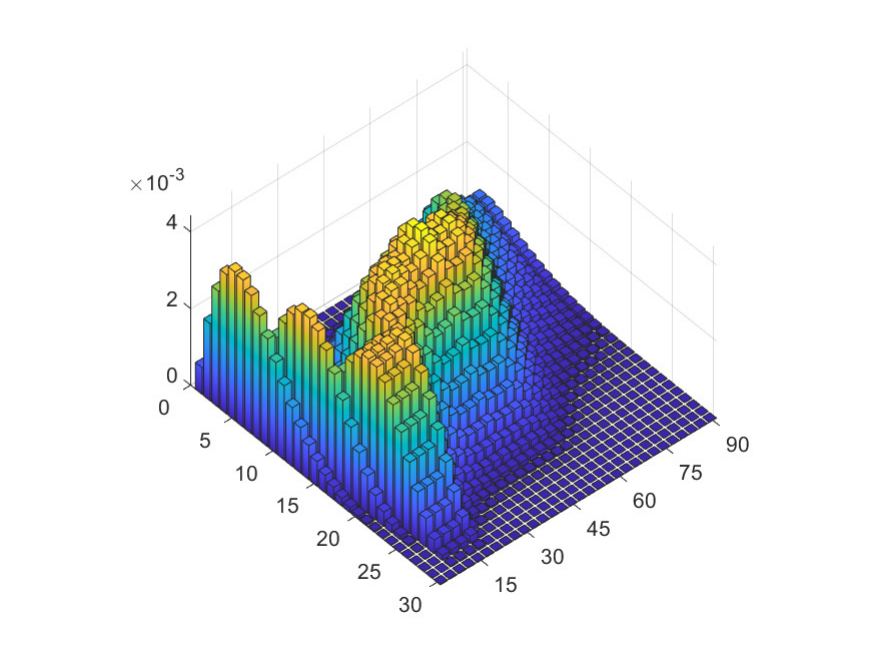}\\
			(a) Australia distrib. with $\textsf{Var}=0.002$ &(b) Australia distrib. with $\textsf{Var}=0.01$ &(c) Australia distrib. with $\textsf{Var}=0.02$
		\end{tabular}\\
		\begin{tabular}{ccc}
			\includegraphics[trim=1.9cm 0.6cm 1.9cm 0.8cm, clip, width=0.31\textwidth]{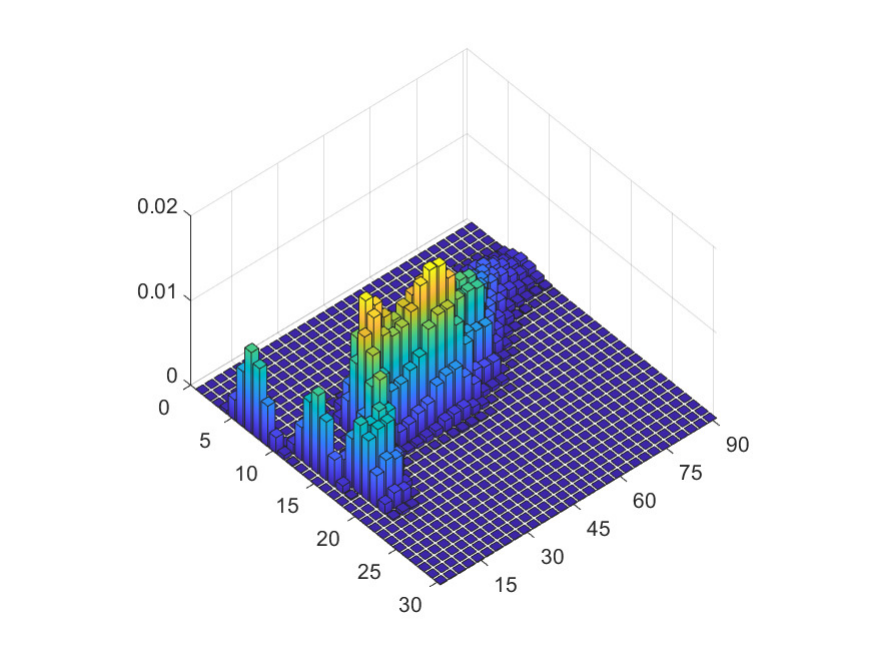}&
			\includegraphics[trim=1.9cm 0.6cm 1.9cm 0.8cm, clip, width=0.31\textwidth]{China_hist_3D_v2.pdf}&
			\includegraphics[trim=1.9cm 0.6cm 1.9cm 0.8cm, clip, width=0.31\textwidth]{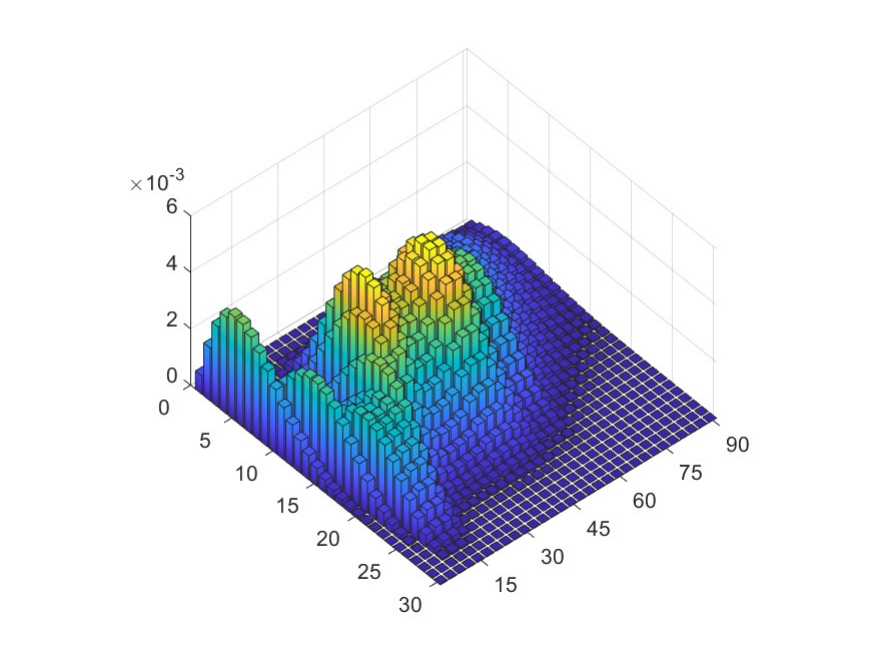}\\
			(d) China distribution with $\textsf{Var}=0.002$ &(e) China distribution with $\textsf{Var}=0.01$ &(f) China distribution with $\textsf{Var}=0.02$
		\end{tabular}\\
		\begin{tabular}{ccc}
			\includegraphics[trim=1.9cm 0.6cm 1.9cm 0.8cm, clip, width=0.31\textwidth]{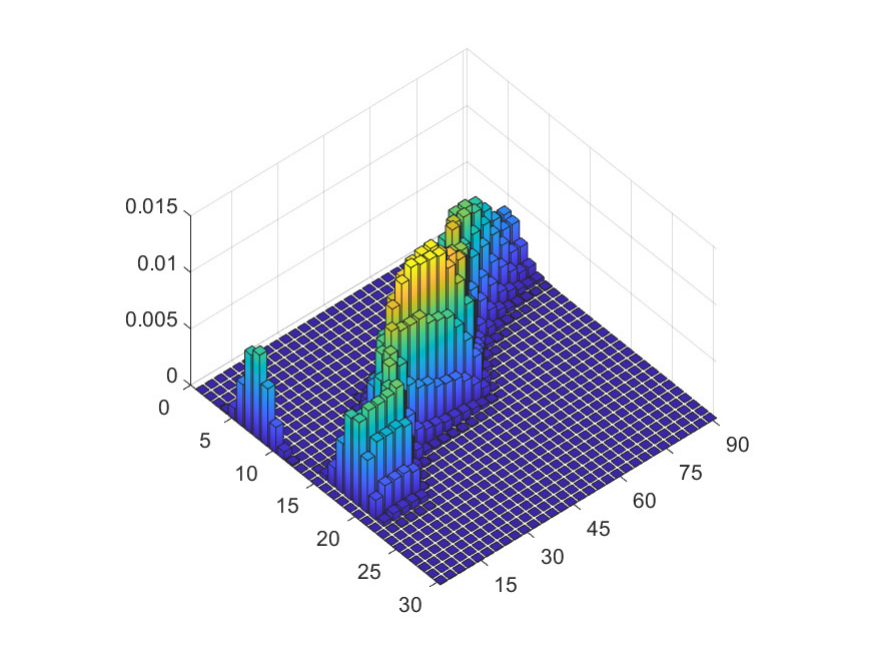}&
			\includegraphics[trim=1.9cm 0.6cm 1.9cm 0.8cm, clip, width=0.31\textwidth]{Italy_hist_3D_v2.pdf}&
			\includegraphics[trim=1.9cm 0.6cm 1.9cm 0.8cm, clip, width=0.31\textwidth]{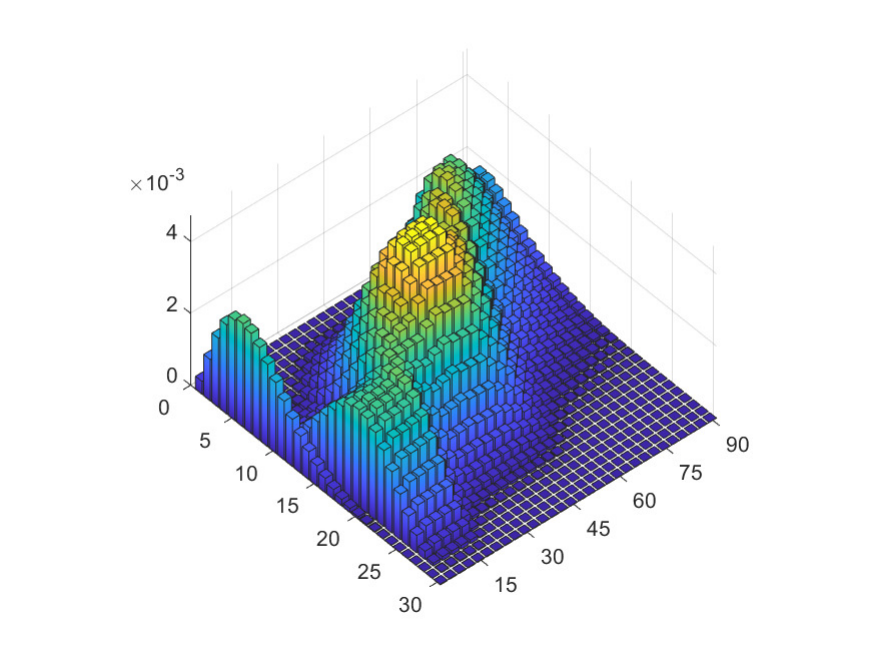}\\
			(g) Italy distribution with $\textsf{Var}=0.002$ &(h) Italy distribution with $\textsf{Var}=0.01$ &(i) Italy distribution with $\textsf{Var}=0.02$
		\end{tabular}
	\end{center}
	\caption{$f_{r,a}$ distributions for the countries that have been considered: Australia, China, Italy. The age bins are $A=\{\mbox{0-2},\mbox{3-5},..,\geq 87 \}$ the daily number of contacts $r \in [0,r_{max}=30]$.} \label{fig:synthetic_populations1}
\end{figure*}

\begin{figure*}[h!]\begin{center}
		\begin{tabular}{ccc}
			\includegraphics[trim=1.9cm 0.6cm 1.9cm 0.8cm, clip, width=0.31\textwidth]{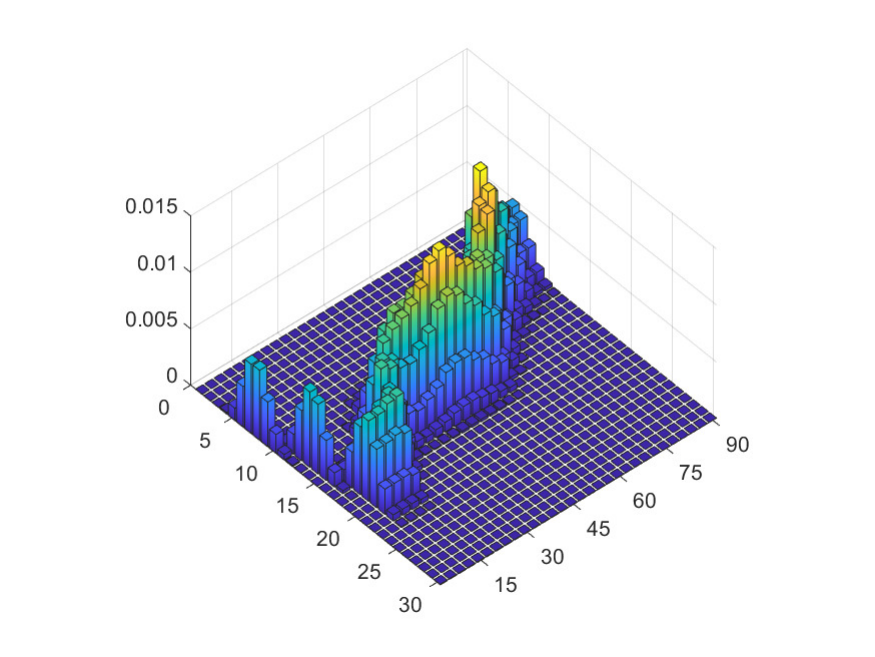}&
			\includegraphics[trim=1.9cm 0.6cm 1.9cm 0.8cm, clip, width=0.31\textwidth]{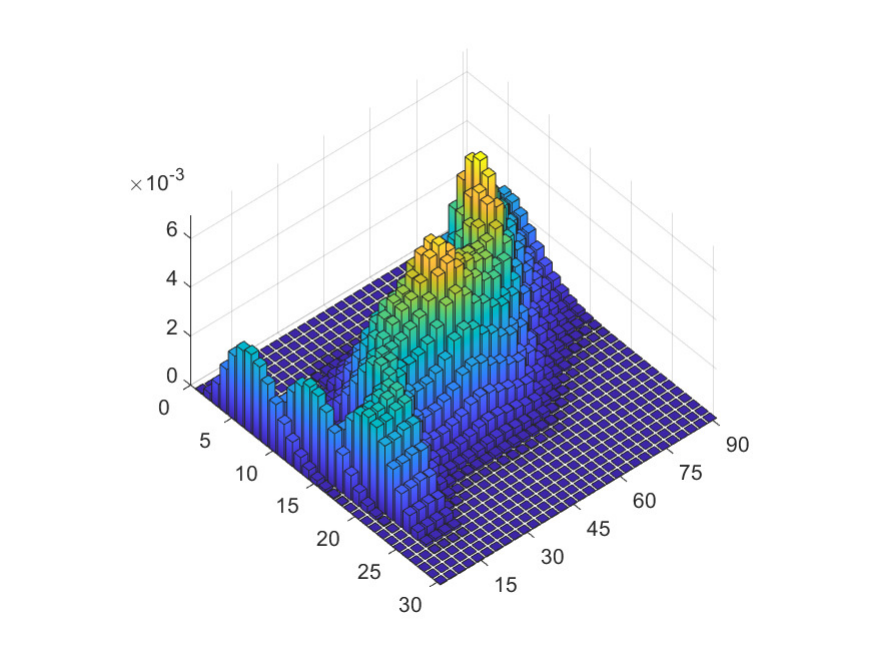}&
			\includegraphics[trim=1.9cm 0.6cm 1.9cm 0.8cm, clip, width=0.31\textwidth]{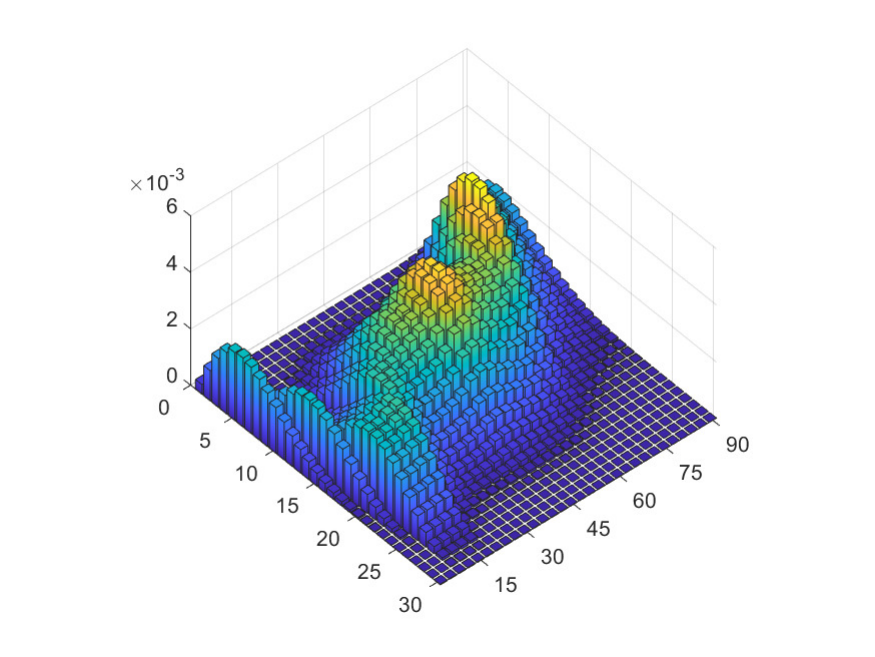}\\
			(a) Japan distribution with $\textsf{Var}=0.002$ &(b) Japan distribution with $\textsf{Var}=0.01$ &(c) Japan distribution with $\textsf{Var}=0.02$
		\end{tabular}
		\begin{tabular}{ccc}
			\includegraphics[trim=1.9cm 0.6cm 1.9cm 0.8cm, clip, width=0.31\textwidth]{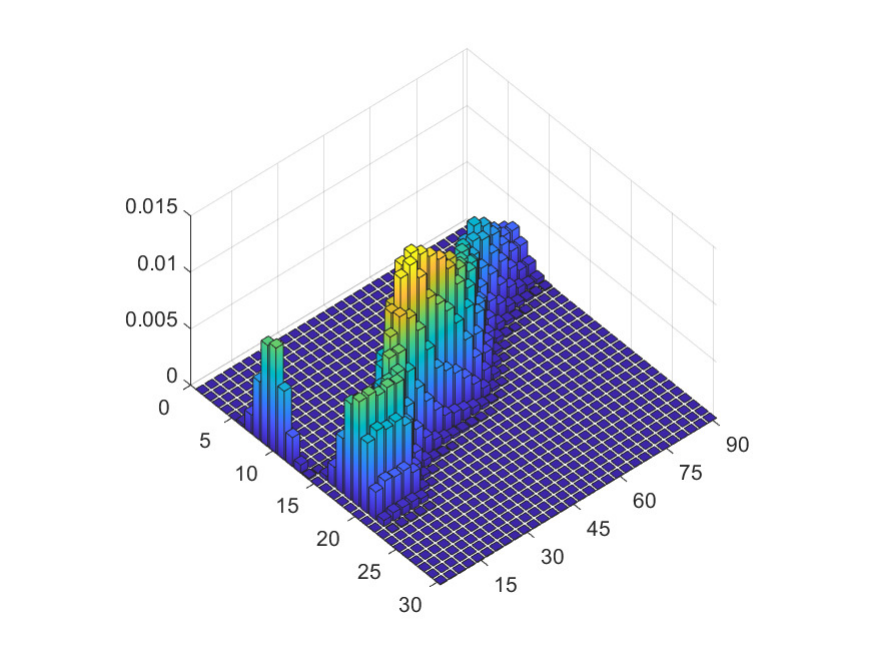}&
			\includegraphics[trim=1.9cm 0.6cm 1.9cm 0.8cm, clip, width=0.31\textwidth]{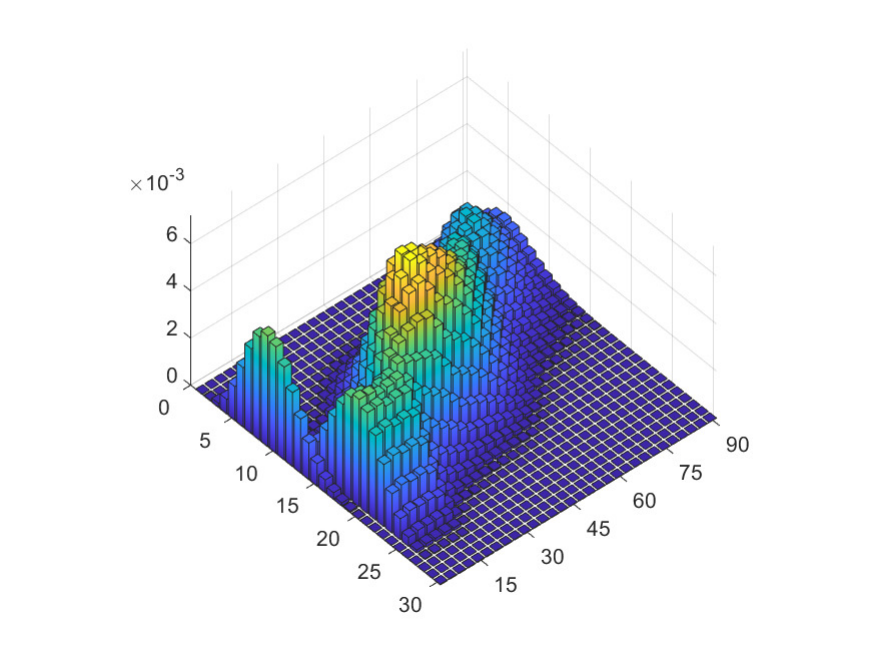}&
			\includegraphics[trim=1.9cm 0.6cm 1.9cm 0.8cm, clip, width=0.31\textwidth]{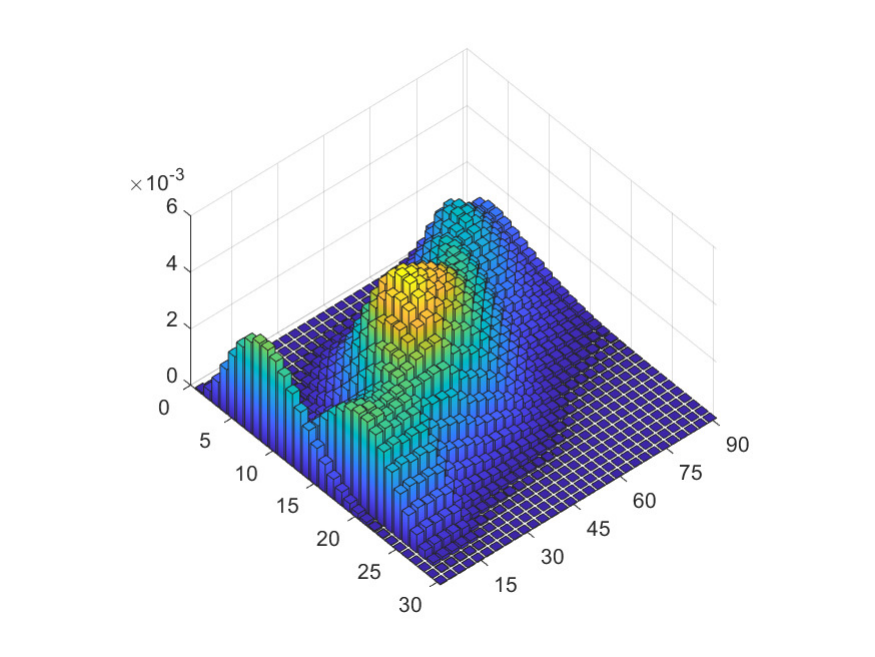}\\
			(d) Spain distribution with $\textsf{Var}=0.002$ &(e) Spain distribution with $\textsf{Var}=0.01$ &(f) Spain distribution with $\textsf{Var}=0.02$
		\end{tabular}
		\begin{tabular}{ccc}	
			\includegraphics[trim=1.9cm 0.6cm 1.9cm 0.8cm, clip, width=0.31\textwidth]{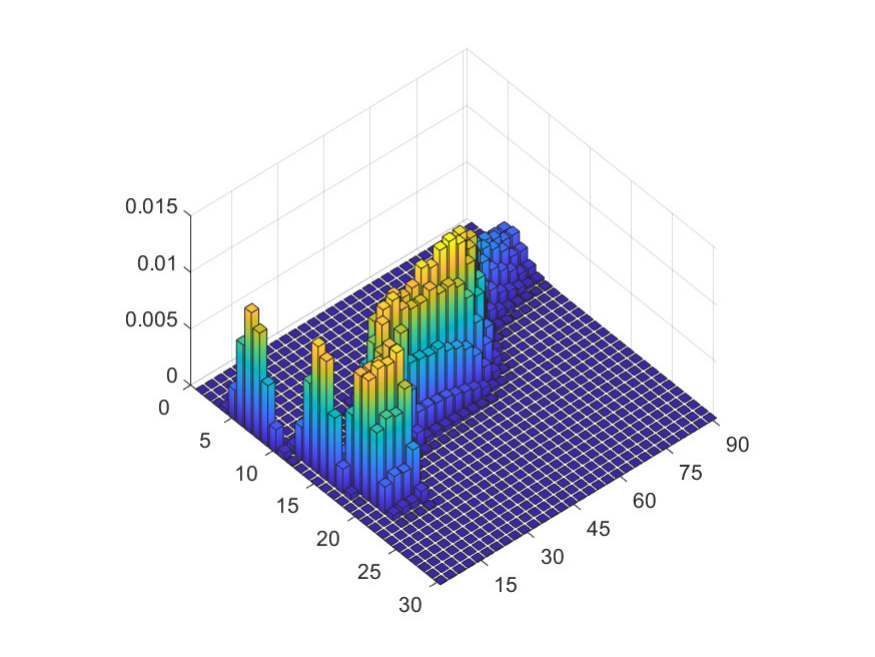}&
			\includegraphics[trim=1.9cm 0.6cm 1.9cm 0.8cm, clip, width=0.31\textwidth]{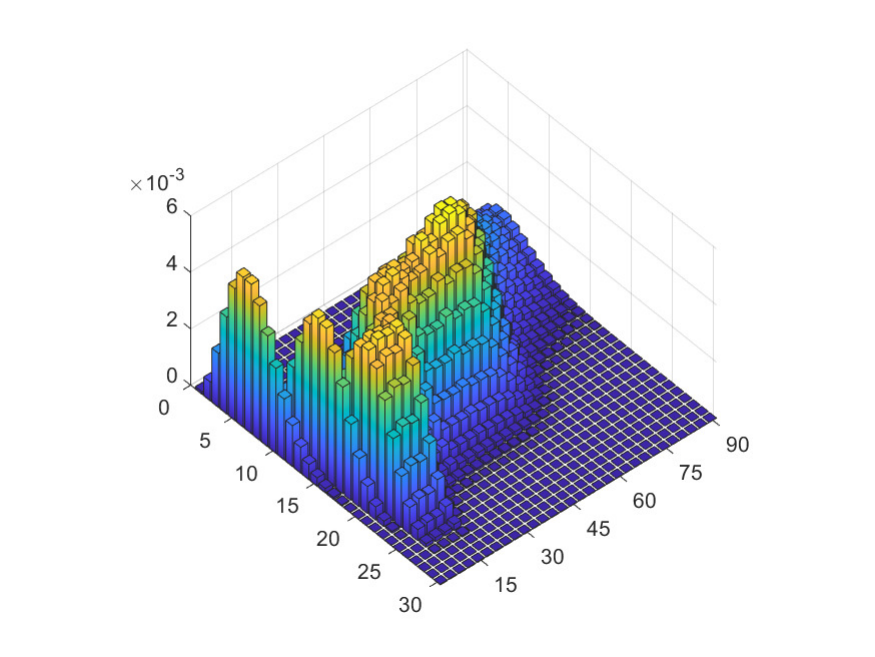}&
			\includegraphics[trim=1.9cm 0.6cm 1.9cm 0.8cm, clip, width=0.31\textwidth]{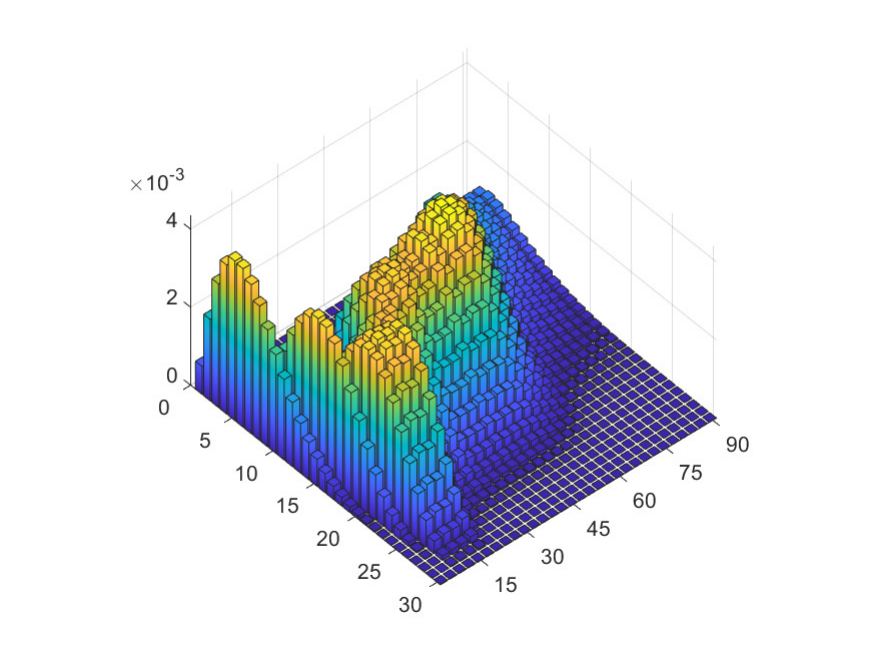}	\\
			(g) USA distribution with $\textsf{Var}=0.002$ &(h) USA distribution with $\textsf{Var}=0.01$ &(i) USA distribution with $\textsf{Var}=0.02$
	\end{tabular}    \end{center}
	\caption{$f_{r,a}$ distributions for the countries that have been considered: Japan, Spain and the United States. The age bins are $A=\{\mbox{0-2},\mbox{3-5},..,\geq 87 \}$ the daily number of contacts $r \in [0,r_{max}=30]$.} \label{fig:synthetic_populations2}
\end{figure*}

\subsection{Construction of the Distribution}\label{sec:frp}
The quantities reported in the previous paragraphs are the essential ingredients for constructing the $f_{r,p}$ distribution function. Indeed, we start from the knowledge of the age distribution $d_a$, the average daily contacts $\mathbb{E}[r_a]$, and the case fatality rate $p_a$, both given for certain age classes $a$. The flowchart reported in Figure \ref{fig:f_ra_flowchart} exemplifies the procedure followed to construct the $f_{r,p}$ distribution, highlighting at what point we used which data.
First, we partitioned the population into uniform age classes, each with a width equal to three years $A=\{0-2,..,\geq 87\}$. It must be observed that the raw data refer to age bins of different sizes: one-year width for the average contacts as well as the population's composition and ten-years width for the case fatality rate. They must be manipulated to refer to the same age classes $v \in A$. Thus, we averaged the mean daily contacts, weighting the average value by the relative frequency of the sub-bin of the wider class $v$. Starting from $\mathbb{E}[r_i], \,\, \text{for} \, i \in A_r=\{0-1,1-2,..,>85\}$ we obtained:
\begin{equation}
	\mathbb{E}[r_v]=\sum_{j \in v}\frac{n_j}{\sum_{k \in v} n_k} r_j \quad \forall v \in A=\{0-2,..,\geq 87\}
\end{equation}
where $n_j$ indicates the number of individuals in the age class $j \in A_r$. Similarly, we have to adapt the CFR defined for ten-year wide classes $A_p=\{0-9,..,>80\}$ to the new three-year bins. For this purpose, the empirical values of the CFR have been linearly interpolated. This allows finding a CFR ($p$) value for each $v \in A$.

Given that we do not consider individual traits such as gender, ethnicity, or pre-existing medical conditions, we consider the death probability $p$ as a fixed value, given the age class $v \in A$. This value has some variability because of the aspects just discussed and individuals' intrinsic differences. Nevertheless, the mortality probability $p$ does not directly affect the system's dynamical behavior; it just determines the number of individuals entering the "death" $D$ state, so its contribution matters on average. Note that in our setting, it is equivalent to talking about $f_{r,p}$ and $f_{r,a}$ due to the one-to-one relationship between an age class $a$ and the corresponding probability of death $p_a$.
On the other hand, the number of daily contacts $r$ enters directly into the epidemiological system's dynamical rule (see Equations \ref{eq:dyn_sys_sm}. Assuming $r$ constant over a specific age interval would be a too substantial simplification. It is known that individuals of the same age have very different mobility patterns due to, for instance, their occupation, their lifestyle, and their medical condition. To account for this variability, we spread the individual's daily contacts around their mean value according to a Beta distribution with a given variance $\sigma$. We choose the support of the Beta distribution to be $[0,r_{max}]$. It is clear that $r \geq 0$, at the same time, we decided to fix an upper bound for the number of contacts assuming that just a negligible number of individuals would fall outside this interval. The average value of contacts $\mathbb{E}[r_v]$ fixes one of the Beta distribution parameters. The other parameter characterizing the Beta distribution appears to be a free variable for which we selected different values to obtain sufficiently different variance values. The resulting distributions have been fed as input to the dynamical model giving qualitatively the same results. Since the desired $f_{r, a}$ distribution needs to be discrete. We defined a set of equispaced $r$ values for which we computed the discrete probability by integrating the Beta distribution in the interval of interest. This procedure allows obtaining a discrete bi-dimensional distribution for 3-year-wide age bins (each corresponding to a certain mortality probability $p$) and equispaced values of $r$ in the interval of definition $[0,r_{max}]$.

\subsection{Estimation of Years of Life Lost}

One of the metrics employed to compare the proposed vaccination strategies consists of evaluating the years of life lost (YLL) due to deaths attributable to the virus. It is a measure of premature mortality that not only considers the chance of death from the virus but also the age of the deceased individuals. In order to perform such an estimation, country-specific demographic data are needed for all countries considered to construct the synthetic populations. The World Health Organization (WHO) collects health-related statistics in the Global Health Observatory \cite{WHO_dataset}, gathering information regarding all 194 WHO member states. In particular, it provides the life expectation stratified by age $\mathbb{E}[L_a]$, pivotal to compute the years of life lost. $\mathbb{E}[L_a]$ is provided for 5-year-wide age bins except for the first bin, which is partitioned into the individuals younger and older than one, and the last one, which groups all people older than 85. We made an interpolation to obtain a $\mathbb{E}[L_a]$ value for each of the age classes of our synthetic populations. We linearly interpolated the empirical data of the Global Health Observatory, providing the life expectation for each three-year-wide age class. To present some of the differences that might be present in $\mathbb{E}[L_a]$ considering different countries, we report the WHO's life expectation values for Italy (left) and China (right) in Table~III. 

\begin{table}[h]
	\begin{center} 
		\begin{tabular}[t]{cc}
			\hline
			\hline
			Age Class & Life Expectation\\
			\hline
			$<$1 & 82.97\\
			1-4 & 82.20\\
			5-9 & 78.24\\
			10-14 & 73.27\\
			15-19 & 68.30\\
			20-24 & 63.36\\
			25-29 & 58.44\\
			30-34 & 53.52\\
			35-39 & 48.61\\
			40-44 & 43.73\\
			45-49 & 38.90\\
			50-54 & 34.15\\
			55-59 & 29.53\\
			60-64 & 25.04\\
			65-69 & 20.73\\
			70-74 & 16.63\\
			75-79 & 12.80\\
			80-84 & 9.24\\
			$\geq$ 85 & 6.21\\
			\hline
			\hline
		\end{tabular}
		\quad\quad
		\begin{tabular}[t]{cc}
			\hline
			\hline
			Age Class & Life Expectation\\
			\hline
			$<$1 & 77.43\\
			1-4 & 76.96\\
			5-9 & 73.04\\
			10-14 & 68.11\\
			15-19 & 63.18\\
			20-24 & 58.28\\
			25-29 & 53.44\\
			30-34 & 48.60\\
			35-39 & 43.79\\
			40-44 & 39.05\\
			45-49 & 34.39\\
			50-54 & 29.79\\
			55-59 & 25.33\\
			60-64 & 21.06\\
			65-69 & 17.01\\
			70-74 & 13.24\\
			75-79 & 9.95\\
			80-84 & 7.06\\
			$\geq$ 85 & 4.76\\
			\hline
			\hline
		\end{tabular}
	\end{center}
	\caption*{Table III: Life expectation - Italy (left) and China (right).}
	\label{tab:life_exp}
\end{table}

{\color{black}
	\section{The parameters of the reference scenario}\label{sm:param}
	
	\subsection{Table of parameters - The COVID-19 case}
	
	Here we report a table with all the (default) values used in our simulation, if not otherwise stated in the text. In the following subsection, we motivate the choice of the values for these parameters, which are roughly set to represent the COVID-19 epidemic in Italy (or any other comparable population in terms of size and contact patterns).
	
	\begin{center}
		\begin{table}[htb]\centering
			\caption*{Table IV: Parameters of reference scenerio}\label{tab:param}
			\boxed{\scalebox{0.9}{%
					\begin{tabular}{l|l|l}
						symbol & value & description \\ \hline
						$N$ \rule{0pt}{3ex} & 60 million & total population size \\
						$N^{\textrm{novax}}$ & 6 million & no-vax population size \\
						$t_{\max}$ & 3 years & time-horizon \\
						$t_V$ & 1 year & time at which vaccinations start \\
						$t_2$ & 2 years & time at which strain 2 appears \\
						$R_0^1$ & 6 & basic reproduction number of strain 1 \\
						$R_0^2$ & 12 & basic reproduction number of strain 2 \\
						$1/\gamma$ & 8 days & average sojourn time in state $I$ \\
						$1/\phi$ & 16 days & average sojourn time in state $H$ \\
						$1/\tau$ & 16 days & average sojourn time in state $T$ \\
						$\mathsf{VE}^{2}_1$ & 0.9 & vaccine efficacy against strain 1 \\
						$\mathsf{VE}^{2}_2$ & 0.7 & vaccine efficacy against strain 2 \\
						$q_{21}$ & 5 & mortality reduction of strain 2 vs strain 1 \\
						$q_{\texttt{post}}$ & 20 & mortality reduction after exposure \\
						&& to virus/vaccine \\
						$\rho_{\max}$ & 15 & maximum transmissibility reduction \\
						$T_{\max}$ & 20,000 & ICU control parameter \\
						$H_{\max}$ & 40,000 & Hospitalizations control parameter \\
						$\widehat T$ & 20,000 & ICU capacity \\
						$\widehat H$ & 50,000 & Hospitalizations capacity \\
						
						$\theta$ & 10 & mortality increase due to ICU saturation  \\
						$\alpha$ & 2 & exponent of economic cost
			\end{tabular}}}
		\end{table}
	\end{center}

	\subsection{Motivation of the choice of the parameters}
	
	The parameters specified in Table IV refer to the specific behavior of the COVID-19 pandemic. This section discusses the choice of such parameters, providing the necessary references that justify them.
	
	In our reference scenario, we considered two phases of the pandemic, which is, of course, a first simplification. In the first phase of our scenario, there was more than one variant on the national territory. Indeed, there were several different variants with slightly different basic reproduction numbers. For simplicity, since our work does not aim to represent the pandemic's evolution accurately but rather to provide a data-driven demonstration of the proposed model, we considered only one reproduction number for the period, roughly equivalent to that of the Delta variant but also compatible with earlier variants. The Delta variant was first identified in India in October 2020 and has a basic reproduction number that ranges from 3.2 to 8 \cite{Liu2021}. Indeed, the value chosen for strain 1 is not far from the values of the basic reproduction number in the early outbreak, which averaged 4.22 when considering several European countries, and was estimated to be 6.33 in Germany and 5.88 in the Netherlands, while in Italy it was 4.25 \cite{Linka2020}. The value of 6 chosen for strain 1 is roughly between the values for the early outbreak and the later Delta variant. For strain 2, “The Omicron variant has an average basic reproduction number of 9.5 and a range from 5.5 to 24” \cite{Linka2020}. These values motivate us to choose a basic reproduction number of 12 for the second strain in our reference scenario since its occurrence corresponds approximately to the onset of the Omicron variant in Italy.
	
	Regarding the average length of stay in the different states, it is reported in \cite{Shryane2021} that the time spent in the ICU was 18.4 days (before 25th March) and 15.4 days (after 7th April), which is why we chose 16 days. This choice is also consistent with the study in \cite{Rees2020}, which examined several studies and found that the median length of total hospital stay (Length of Stay) ranged from 5 to 29 days. “Most studies (43/52) reported LoS for total hospitalization only, with four studies reporting LoS for ICU only, and five studies reporting both.” \cite{Rees2020}. Regarding the length of stay in the ICU in this study, the median ranged from 5 to 19 days \cite{Rees2020}. We used an average sojourn time of 16 days for both hospitalizations and ICU. Table 4 of \cite{Byrne2020} reports the infectious period (IP) for symptomatic cases from several studies, ranging from 3 to 20 days. The infectious period seems to be about a week. We set it at eight days in our scenario.
	
	On the website of the Italian Government \cite{governoitaliano}, a report regarding vaccinations is available, as for 17th July 2023, 90,25\% of the population over 12 has completed the vaccinal cycle, justifying the 10\% of no-vax individuals considered in our scenario.
	
	To what concerns the mortality reduction of strain 2 compared to strain 1 and the mortality reduction associated with acquired/natural immunity: In \cite{Liu2022}, it is found that in South Africa, the infection fatality ratio was reduced by 78.7\%, which approximately corresponds to a factor 5 decrease, as used in our scenario. For the reduction in mortality, we assumed the effect of vaccination and natural immunity are comparable. In \cite{Rahmani2022} is reported that the COVID-19-related mortality of the Pooled Vaccine Effectiveness (PVE) was 92\%, corresponding to a hazard ratio of 0.08, which we optimistically associated with a factor of reduction of 20.
	
	As for the number of “regular” hospital and intensive care beds, we rely on the ISTAT (Italian Institute of Statistics) report \cite{ISTAT2020}. The number of “regular” beds is not easy to determine since there are different types of hospital beds, depending on the expected length of stay and the medical service to be provided. This number has declined over the last 30 years and is now about 200000 units. We took a conservative approach in our analysis and assumed $\hat{H}=50000$ for two reasons: First, the above number concerns all types of hospital beds; second, not all beds can be reserved for COVID-19 patients; we considered a quarter of the total number a reasonable choice. Regarding ICU, the report mentioned earlier \cite{ISTAT2020} states that the number of ICU beds is, on average, 15.1 for every 100000 individuals, which is about 10000 beds. Note that while we have set $\hat{T}$ to 20000, we have considered $T_{max}=10000$ in most of our simulations. We argue that it was also interesting to investigate scenarios with a larger capacity.
	
}

\section{Impact of population heterogeneity in an uncontrolled scenario}

To test the results of our dynamical model, we constructed different distributions trying to take sufficiently different countries in terms of the average age of the population, culture, geographical location, mobility patterns, and the healthcare system. Intending to do this, we considered Australia, China, Italy, Japan, Spain, and the United States. Figures \ref{fig:synthetic_populations1}-\ref{fig:synthetic_populations2}  above report the $f_{r,a}$ distributions obtained from the data of Italy, Japan, and the United States. As discussed in the previous paragraph, when introducing the Beta distribution for the number of contacts $r$, we introduced a free parameter (the variance). To perform a sensitivity analysis, we defined the country-specific distribution for three different choices of this variance parameter.

In this experiment, we are interested in assessing the impact of population distribution on the evolution of epidemics. 
We consider a large population of size $N=59.55\cdot10^6$, and we run simulations with different population distributions $f_{r,p}$ retrieved from empirical data as described in the previous section.
We set the initial conditions $$I_{r,p}(0)=\frac{\lambda_0}{\gamma}\frac{rf_{r,p}}{\sum_{r,p}rf_{r,p}}, \qquad S_{r,p}(0)=Nf_{r,p}-I_{r,p}(0)$$ with initial infections $\lambda_0=10000$, and $M_{r,p}(0)=D_{r,p}(0)=0$. We consider the scenario with $I_{\max}=700000$, the mortality increasing factor $\theta=3$, the recovery rate $\gamma=1/14$, and different values of $\widetilde{\mathcal{R}}_0^{\text{e}}\in\{3,9,12\}.$

It is worth noticing that, due to differences in the population related to risk exposure and mortality, the thresholds to reach herd immunity are different, becoming country-specific. We refer to herd immunity thresholds as the percentage of people who need to be immune against the disease to achieve herd immunity.
The values are reported in Table V.
\begin{table}
	\begin{minipage}[c]{\linewidth} 
		\centering
		\begin{tabular}[t]{lccc}
			\hline
			\hline
			Country & $\widetilde{\mathcal{R}}_0^{\text{e}}=3$ & $\widetilde{\mathcal{R}}_0^{\text{e}}=9$ & $\widetilde{\mathcal{R}}_0^{\text{e}}=12$\\
			\hline
			Australia & 59.5964$\%$ & 85.9962$\%$ & 87.1337$\%$\\
			China & 62.7654$\%$ & 88.0774$\%$ & 89.1635	$\%$\\
			Italy & 59.5964$\%$ & 84.4589$\%$ & 85.6286$\%$\\
			Japan & 61.3880$\%$ & 86.5939$\%$ & 87.7222$\%$\\
			Spain & 61.0961$\%$ & 86.1429$\%$ & 87.2721$\%$\\
			USA & 60.9627$\%$ & 86.0664$\%$ & 87.2037$\%$\\
			\hline
			\hline
		\end{tabular}
		\caption*{Table V - Immunity thresholds}
	\end{minipage}
\end{table}

From these results, we can observe that China exhibits the highest percentage of people who need to be immune against the disease to extinguish the virus's circulation, followed by Japan, Spain, the USA, and Australia. At the same time, Italy is the country with the lowest thresholds.
If we look at the number of deaths in an uncontrolled regime, based on the heterogeneity in the population distribution, we can see the smallest values for China, followed by Australia and USA, then Spain, Italy, and Japan. Despite the highest herd immunity thresholds in China, the population is relatively younger compared to other countries, making Cina the country with the lowest percentage of deaths.
The order of final deaths can be related to the population's median age (see Table VI). 
\begin{table}[h!]
	\begin{minipage}[c]{\linewidth} 
		\centering
		\begin{tabular}[t]{lcccc}
			\hline
			\hline
			Country & Median age (years) & $\widetilde{\mathcal{R}}_0^{\text{e}}=3$ & $\widetilde{\mathcal{R}}_0^{\text{e}}=9$ & $\widetilde{\mathcal{R}}_0^{\text{e}}=12$\\
			\hline
			Australia & 37.9 & 12.3067$\%$ & 12.0818$\%$ & 12.2591$\%$\\
			China     & 38.4 &  7.5232$\%$ &  9.1322$\%$ &  9.1814$\%$\\
			Italy 	  & 47.3 & 12.3067$\%$ & 16.8368$\%$ & 17.1428$\%$\\
			Japan     & 48.4 & 17.2595$\%$ & 21.4392$\%$ & 21.6290$\%$\\
			Spain     & 44.9 & 10.2865$\%$ & 13.7109$\%$ & 13.9261$\%$\\
			USA       & 38.1 &  9.3247$\%$ & 12.2321$\%$ & 12.4003$\%$\\
			\hline
			\hline
		\end{tabular}
		\caption*{Table VI - Positive correlation between deaths and median age}
	\end{minipage}
\end{table}

Finally, Table VII 
reports the PYLL and YLL ($\%$ on the PYLL) for each country.
If we sort countries by PYLL indicator, we see that younger populations, therefore with higher PYLL, such as China, Australia, and USA, if hit by the pandemic rampantly, could have in principle a greater loss in terms of YLL. 
What we see from the results, instead, is the lowest percentage of YLL for China, followed by the USA, Spain, Australia, Italy, and Japan.

\begin{table}[h!]
	\begin{minipage}[c]{\linewidth} 
		\centering
		\begin{tabular}[t]{lcccc}
			\hline
			\hline
			Country & PYLL (years) & $\widetilde{\mathcal{R}}_0^{\text{e}}=3$ & $\widetilde{\mathcal{R}}_0^{\text{e}}=9$ & $\widetilde{\mathcal{R}}_0^{\text{e}}=12$\\
			\hline
			Australia & 2.7411$\cdot10^9$ & 4.1193$\%$ & 4.2774$\%$ & 4.3147$\%$\\
			China     & 2.8003$\cdot10^9$ & 3.0344$\%$ & 3.5353$\%$ & 3.5459$\%$\\
			Italy 	  & 2.4898$\cdot10^9$ & 4.5350$\%$ & 5.8434$\%$ & 5.9145$\%$\\
			Japan     & 2.3062$\cdot10^9$ & 6.4868$\%$ & 7.6870$\%$ & 7.7285$\%$\\
			Spain     & 2.6742$\cdot10^9$ & 3.6432$\%$ & 4.5695$\%$ & 4.6159$\%$\\
			USA       & 2.6742$\cdot10^9$ & 3.4649$\%$ & 4.2741$\%$ & 4.3091$\%$\\
			\hline
			\hline
		\end{tabular}
		\caption*{Table VII - Total YLL ($\%$ of PYLL)}
	\end{minipage}
\end{table}
\medskip

\section{Incorporate vaccinations into the model}\label{sec:full_model}
Vaccines are assumed to guarantee partial protection. 
According to classification in \cite{MATRAJT201517}, we consider two descriptors for the vaccines: reduction in the probability of becoming infected (vaccine efficacy on susceptibility) and reduction in the pathogenicity (vaccine efficacy to prevent or diminish symptoms). Moreover, for simplicity, we model the vaccine efficacy on the population by neglecting the response transition, and we restrict to the case of a single type of vaccine, which is administered in two doses separated by a fixed interval of $\Delta$ days.
We assume that the administration rate of either dose is fixed, equal to $\xi$, so the entire population can be potentially vaccinated (with two doses) after $\mathcal{T}_v$ days. Hence we set $\xi= N/(\mathcal{T}_v- \Delta)$.

Let $\mathsf{VE}^{1}, \mathsf{VE}^{2}$ be the vaccine efficacy on susceptibility after one or two doses, respectively. 
Moreover, we assume that mortality is reduced by a factor $q_{\texttt{post}}$ after a single dose of vaccine.

We assume that $N^{\textrm{novax}}$ people, uniformly distributed over the population, refuse to be vaccinated. Equations \equaref{dyn_sys} describe the evolution of their status. Let $S_{r,p}^{\textrm{novax}}(t)$ be the number of no-vax people in class $(r,p)$ who are still susceptible at time~$t$.

We describe the system dynamics assuming individuals do not return to the susceptible state after being infected or vaccinated. Adding this possibility is not difficult, but we omit it here for brevity. 

The evolution of the vaccinated population requires adding a few more compartments compared to those introduced in sec. \ref{sec:model}:   
Let $V_{r,p}^{\textrm{1m}}(t)$ be the number of people in class $(r,p)$ who have received just the first dose, which is already effective against the virus (i.e., they can no longer be infected). Let $V_{r,p}^{\textrm{1s}}(t)$ be the number of people in class $(r,p)$ still susceptible after receiving just the first dose.
Let $V_{r,p}^{\textrm{2m}}(t)$ be the number of people in class $(r,p)$ who have received both doses and are immune.
At last, let $V_{r,p}^{\textrm{2s}}(t)$ be the number of people in class $(r,p)$ who have received both doses but are still susceptible.
Due to strict prioritization among classes, a given  class $(r,p)$ receives the first dose at full rate $\xi$ 
only within a specific time window: 
$[\mathcal{T}^{\min}_{r,p},\mathcal{T}^{\max}_{r,p}]$. {\color{black}
	The vaccination window for each class is computed based on the class priority:
	$
	\mathcal{T}^{(1), \max}_{r,p}= \inf \{t: 	S_{r,p}(t) =0 \}
	$; 
	$
	\mathcal{T}^{(1),\min}_{r,p}= \max_{ (r',p')\in HP(r,p)}\{  \mathcal{T}^{\max}_{r',p'}\}
	$,
	where $HP(r,p)$ is the set of classes with higher priority than $(r,p)$.
	The vaccination rate for the first dose of vaccine for the population class $(r,p)$ can be written as:
}
\begin{equation*}
	\xi_{r,p}^{(1)}(t)=
	\left\{\begin{array}{ll}
		0  &   t< \mathcal{T}^{\min}_{r,p}\\
		\xi  &    \mathcal{T}^{ \min}_{r,p} \le t < \mathcal{T}^{\max}_{r,p} \\ 
		0  &  t \ge \mathcal{T}^{\max}_{r,p}\\
	\end{array}\right.
\end{equation*}

Let  $V_{r,p}^1(t) = \int_{t-\Delta}^t  \xi _{r,p}^{(1)}(t) \mathrm{d} t $
be the number of people in class $(r,p)$ who have received just the first   dose of vaccine at time $t$.
We assume that the second dose of vaccine is administered at rate:
\[
\xi_{r,p}^{(2)}(t)=\frac{V_{r,p}^{1s} (t)+V_{r,p}^{1m} (t) }{V_{r,p}^1(t)}  \xi_{r,p}^{(1)}(t-\Delta)
\]
only to individuals who have received the first dose and have not been 
infected in the meanwhile.

At last, let:
\[
\widehat S(t)= \sum_{r,p}r(S_{r,p}(t)+
V_{r,p}^{\textrm{1s}}(t)+
V_{r,p}^{\textrm{2s}}(t)+
S_{r,p}^{\textrm{novax}}(t))
\]
be the total number of susceptible edges at time $t$.
Note that
$$\lambda(t) = \frac{\sigma}{\rho(t)} \left( \sum_{r,p} r I_{r,p}(t) \right)
\frac{\widehat{S} (t) }{\mathbb{E}[r]N} $$

Since people who receive at least one dose are less likely to die, we need to keep track of them, hence vaccinated people who get infected traverse a separate chain of compartments $I_{r,p}^v(t), H_{r,p}^v(t), T_{r,p}^v(t)$ with respect to those who do not receive any dose.
The dynamics governing the evolution of $H_{r,p}^v(t), T_{r,p}^v(t)$ are analogous to those in \equaref{dyn_sys} with the only difference that $p_{r,p}^{TD}(t)$ is replaced by $p_{r,p}^{TD}(t)/q_{\texttt{post}}$.

The new vaccination dynamics are \textcolor{black}{represented in Figure \ref{fig:our_model_vax}} and described by the following equations:
\begin{align}
	\dot{S}_{r,p}(t) & =  -\lambda(t) \frac {r S_{r,p}(t)}{\widehat S(t)} 
	- \xi^{(1)}_{r,p} (t) 
	\nonumber  \\
	\dot{I}_{r,p}(t) & = \lambda(t) \frac {r S_{r,p}(t)}{\widehat S(t)} - \gamma I_{r,p}(t)    \nonumber  \\
	\textcolor{black}{\dot H_{r,p}(t) } &= 	\textcolor{black}{\gamma p_{r,p}^{IH}  I_{r,p}(t) - \phi H_{r,p}(t) } \nonumber \\
	\textcolor{black}{\dot T_{r,p}(t)} & = \textcolor{black}{\phi p_{r,p}^{HT} H_{r,p}(t) -\tau T_{r,p}(t)} \nonumber\\
	\dot{I}_{r,p}^v(t) & = \lambda(t) \frac {r (V_{r,p}^{\textrm{1s}}
		+ V_{r,p}^{\textrm{2s}})}{\widehat S(t)} -  \gamma I_{r,p}^v(t)   \nonumber  \\
	\textcolor{black}{\dot H^v_{r,p}(t) } &= 	\textcolor{black}{\gamma p_{r,p}^{IH}  I^v_{r,p}(t) - \phi H^v_{r,p}(t) } \nonumber \\
	\textcolor{black}{\dot T^v_{r,p}(t)} & = \textcolor{black}{\phi p_{r,p}^{HT} H^v_{r,p}(t) -\tau T^v_{r,p}(t)} \nonumber\\
		\textcolor{black}{\dot M_{r,p}(t) }  &= \textcolor{black}{ \gamma  (1-p_{r,p}^{IH}) (I_{r,p}(t) + I^v_{r,p}(t)) }\nonumber\\
	& \,\, \textcolor{black}{+ \phi (1-p_{r,p}^{HT}) (H_{r,p}(t)+H^v_{r,p}(t))} \nonumber \\
	& \,\, \textcolor{black}{ + \tau \left((1-p_{r,p}^{TD}(t)) T_{r,p}(t) + (1 -p_{r,p}^{TD}(t) /
		q_{\scriptscriptstyle post}) T^v_{r,p}(t) \right)
	} \nonumber \\
	\dot{V}_{r,p}^{\textrm{1m}}(t) & = \xi^{(1)}_{r,p} (t) \mathsf{VE}^{1} 
	- \xi^{(2)}_{r,p} (t) \frac{V_{r,p}^{\textrm{1m}}(t)}{V_{r,p}^{1s}(t)+V_{r,p}^{1m} (t)   \nonumber  } \\ 		\dot{V}_{r,p}^{\textrm{1s}}(t)  &=  (1-\mathsf{VE}^{1})\xi^{(1)}_{r,p}
	-\lambda(t) \frac{r V_{r,p}^{\textrm{1s}}(t)}{\widehat S(t)} - \frac{ \xi^{(2)}_{r,p} (t) V_{r,p}^{\textrm{1s}}(t)}{V_{r,p}^{1s}(t)+ V_{r,p}^{1m}(t)}   \nonumber  \\ 
	\dot{V}_{r,p}^{\textrm{2m}}(t)  &= \xi^{(2)}_{r,p} \frac{V_{r,p}^{\textrm{1m}}(t) + 
		\frac{\mathsf{VE}^2- \mathsf{VE}^1}{1-\mathsf{VE}^1 }	V_{r,p}^{\textrm{1s}}(t)}{ V_{r,p}^{1s}(t)+ V_{r,p}^{1m}(t)}   \nonumber  \\
	\dot{V}_{r,p}^{\textrm{2s}}(t)  &= \xi^{(2)}_{r,p} \frac{1-\mathsf{VE}^2}{ 1-\mathsf{VE}^1} \frac{V_{r,p}^{\textrm{1s}}(t)}{V_{r,p}^{1s}(t)+V_{r,p}^{1m}(t) } 
	- \lambda(t) \frac{r V_{r,p}^{\textrm{2s}}(t)}{\widehat S(t)} \nonumber\\
	\textcolor{black}{\dot D_{r,p}(t)} &= \textcolor{black}{\tau p_{r,p}^{TD}(t) \left(T_{r,p}(t) + \frac{T^v_{r,p}(t)}{q_{\scriptscriptstyle post}} + T^{\text{no-vax}}_{r,p}(t)\right) }
	\label{eq:dyn_sys2}
\end{align}

{\color{black}
	\section{Additional proofs}
	Here, we present the proofs of Corollary 2 and 3, for completeness we report again the statements.
	
	In these following corollaries we explore two interesting cases of delay distributions. 
	\begin{corollary}[Exponential delay distribution]
		If $\mathfrak{f}_d(\tau) = u(\tau) \delta e^{-\delta (\tau)}$, then 
		the system is always (locally) stable. 
	\end{corollary}
	\begin{proof}
		Consider an exponential delay distribution of parameter $\delta$ (mean $1/\delta$).
		We obtain
		\begin{equation}\label{eq:es1}
			\mathcal{L}\{ \eta(t) \}  = \frac{\eta(0)}{s + \frac{\gamma \delta}{s + \delta}}
		\end{equation} 
		The poles of \equaref{es1} are the roots of the second-order equation  
		$s^2 + s \delta +\gamma \delta$. Since the the real part of both roots is negative
		for any $\delta$, the system is always stable. 
		For $\frac{1}{\delta} > \frac{1}{4 \gamma}$ it exhibits
		dumped oscillations, otherwise it exhibits an exponential decay.
	\end{proof}
	
	\begin{corollary}[Shifted exponential delay distribution]
		Let $\mathfrak{f}_d(\tau) = u(\tau-d) \delta e^{-\delta (\tau-d)}$. 
		For any given $\delta > 0$, there exists a critical delay
		$d^* = \frac{1}{\gamma} f(\delta)$, such that the system
		is (locally) stable if $d < d^*$, otherwise the system is unstable.
		As $\delta$ grows from 0 to $\infty$, $d^*$ grows
		from $1/\gamma$ to $\pi/(2 \gamma)$.
	\end{corollary}
	\begin{proof}
		Consider the shifted exponential distribution:
		$\mathfrak{f}_d(\tau) = u(\tau-d) \delta e^{-\delta (\tau-d)}$.
		In this case we have
		\begin{equation}\label{eq:es2}
			\mathcal{L}\{ \eta(t) \}  = \frac{\eta(0)}{s + \frac{\gamma \delta}{s + \delta}e^{-sd}}
		\end{equation} 
		whose poles $z = b + i \omega$ satisfy the equations:
		\begin{eqnarray}\label{eq:bomegashift}
			\begin{cases}
				b^2 - \omega^2 + b \delta + \gamma \delta e^{-b d} \cos(\omega d)  = 0\\
				2 b \omega + \omega \delta - \gamma \delta e^{-b d} \sin(\omega d) = 0 
			\end{cases}\,.
		\end{eqnarray}   
		At the critical point, $b = 0$, and above equations reduce to:
		\begin{eqnarray}\label{eq:redshift}
			\begin{cases}
				\omega^2 = \gamma \delta \cos(\omega d) \\
				\omega = \gamma \sin(\omega d) 
			\end{cases}\,.
		\end{eqnarray}   
		For given $\delta$, we can solve the above two equations in the unknowns $\omega^{\star}$, $d^{\star}$, obtaining:
		\begin{eqnarray*}\label{eq:star}
			\omega^{\star} = \sqrt{\frac{-\delta^2+\sqrt{\delta^4+4\gamma^2\delta^2}}{2}} ,\quad d^{\star} = \frac{\arcsin(\frac{\omega}{\gamma})}{\omega}.
		\end{eqnarray*}   
		
	\end{proof}

\begin{figure}[h!]\centering
	\includegraphics[width=0.7\columnwidth]{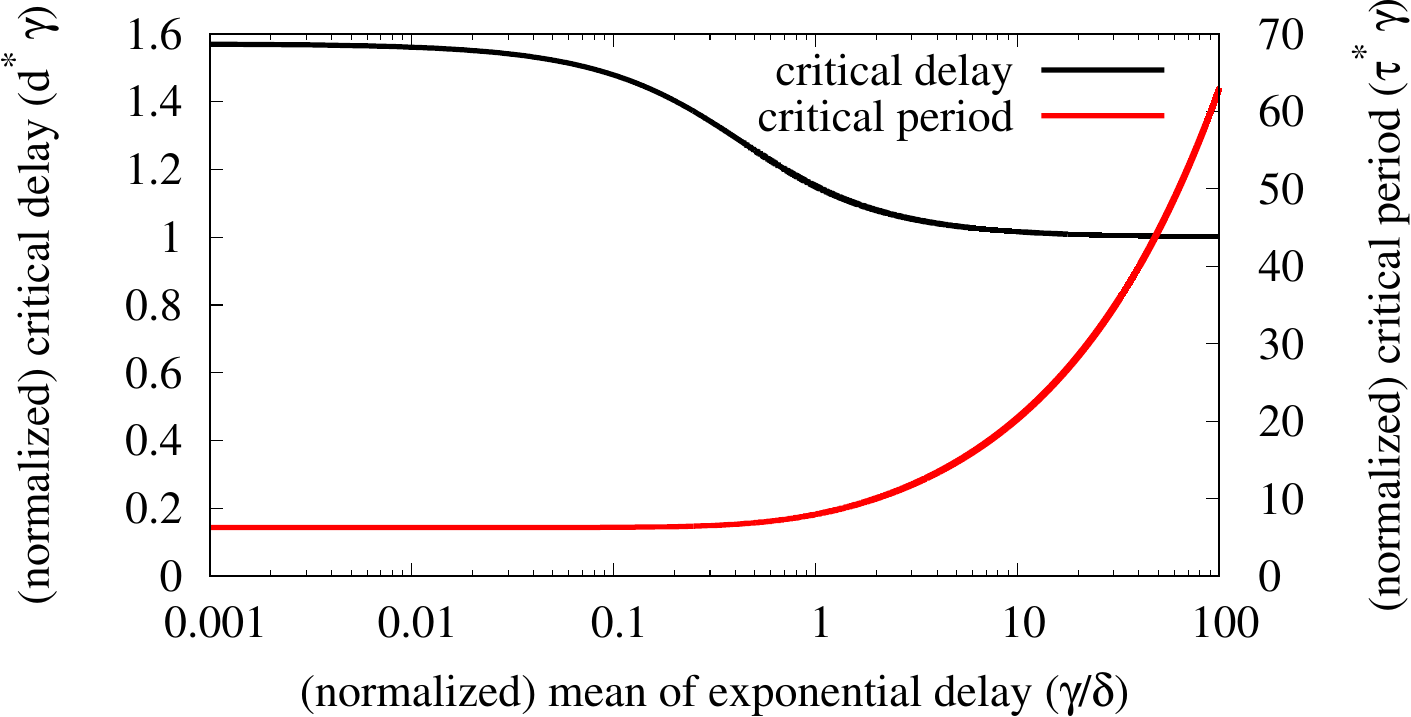}
	\caption{\textcolor{black}{Critical values $d^{\star}$ (left y axes)  and $\tau^{\star}$ (right y axes) as function of $1/\delta$. All quantities are normalized by $1/\gamma$.}}
	\label{fig:expdelay}
\end{figure}  
	
	Fig. \ref{fig:expdelay} shows the critical values $d^{\star}$ and $\tau^{\star} = 2 \pi/\omega^{\star}$ as function of $1/\delta$  (in the plot all quantities are normalized by the average sojourn time $1/\gamma$ in the infectious state).

Interestingly, the critical value $d^{\star}$ of the shift decreases as we increase the average  $1/\delta$ of the exponential distribution.
For example, with $\delta = \gamma$ we have $d^{\star} \gamma \approx 1.15$, in constrast to $d^{\star} \approx 1.57$ with deterministic delay.  
}

\begin{figure*}[h!]
	\centering
	\includegraphics[width=0.9\textwidth]{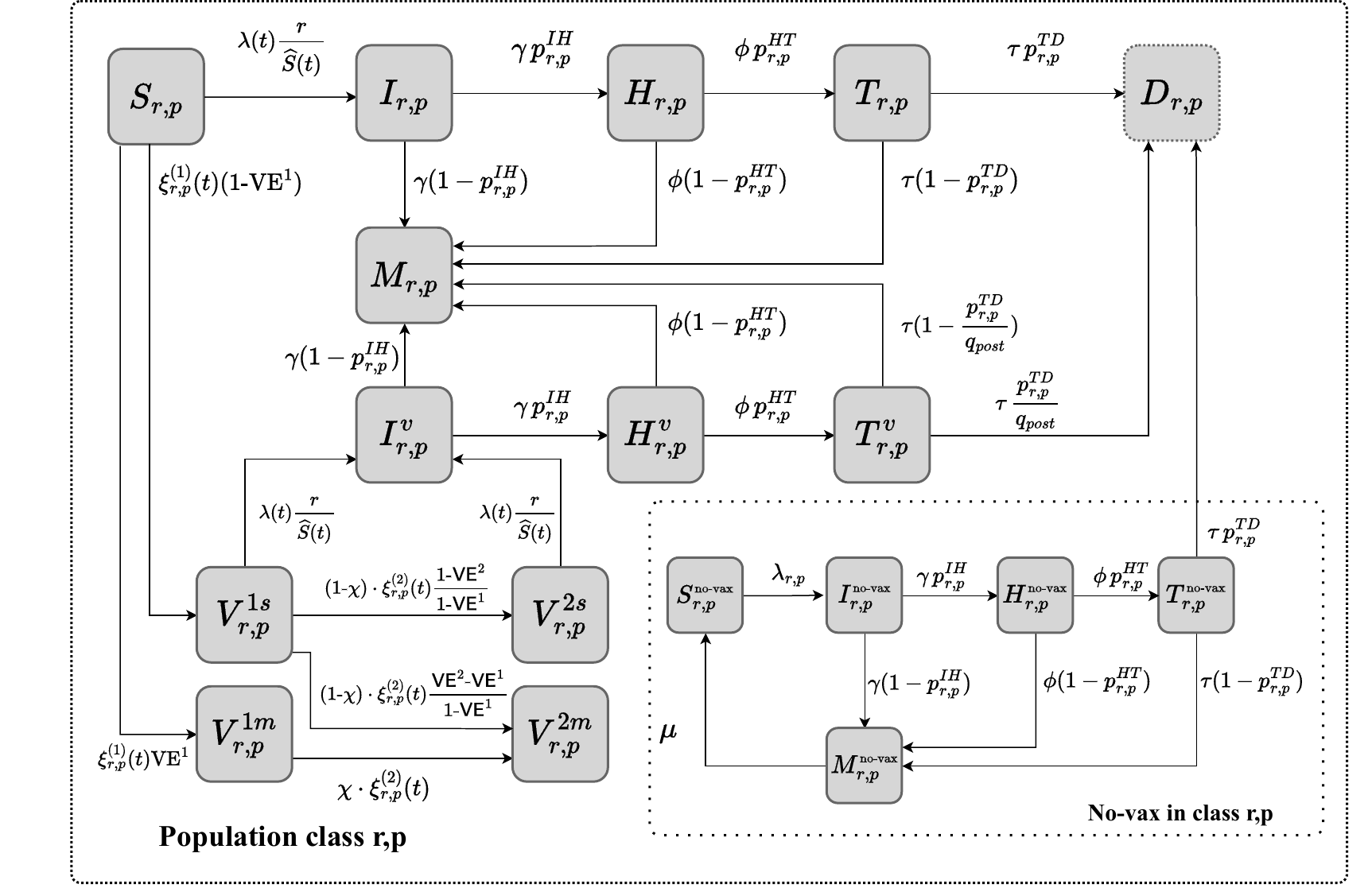}
	\caption{\textcolor{black}{Schematic representation of the proposed model with vaccinations, corresponding to the system of equations (\ref{eq:dyn_sys2}). For ease of representation we showed the dynamics of only one class $(r,p)$ we remark that the dynamics of different classes are intertwined both through the infection rate, as depicted in Figure \ref{fig:our_model_sm}, and the process of vaccinations prioritization. For compactness, we defined $\chi := \frac{V_{r,p}^{\textrm{1m}}(t)}{V_{r,p}^{1s}(t)+V_{r,p}^{1m} (t)}$.}}
	\label{fig:our_model_vax}
\end{figure*}
	
	\bibliographystyle{IEEEtran}
	\bibliography{ref}
	
	\begin{IEEEbiography}[{\includegraphics[width=1in,height=1.25in,clip,keepaspectratio]{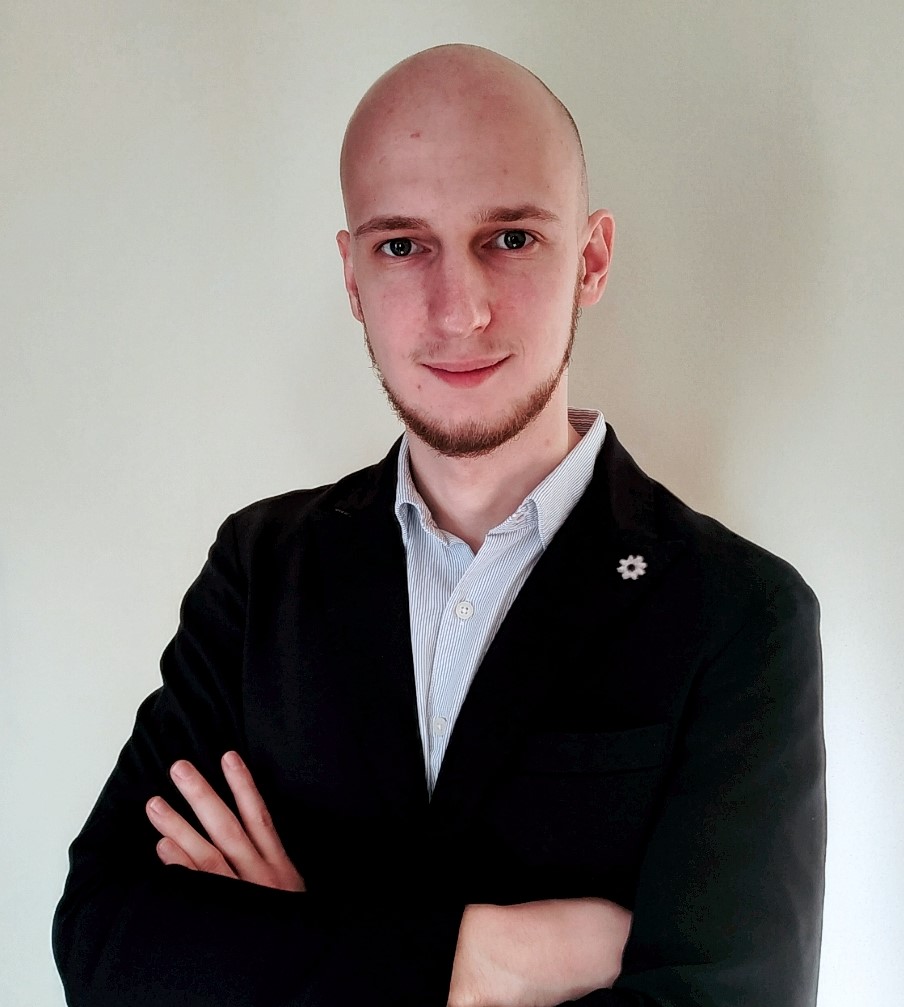}}]{Franco Galante} received the Bachelor Degree from Università degli Studi di Padova in 2018, and the Master Degree from Politecnico di Torino in Communications and Computer Network Engineering in 2020. In 2019 he spent a semester at TU Delft. He is currently a Ph.D. student at Politecnico di Torino within the department of Electronics and Telecommunications. His research interests include dynamics over networks, randomized algorithms and control of networks.
	\end{IEEEbiography}
	
	\vspace{-1cm}
	
	\begin{IEEEbiography}[{\includegraphics[width=1in,height=1.25in,clip,keepaspectratio]{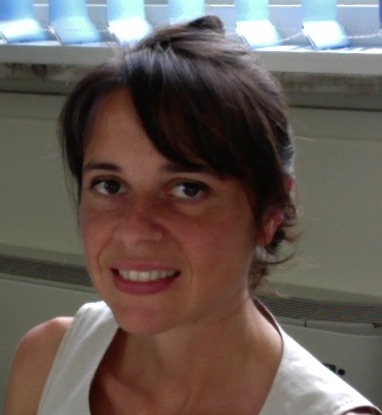}}]{Chiara Ravazzi}
		(M'13) is currently a Tenured Researcher with the CNR-IEIIT and Adjunct Professor at Politecnico di Torino. She received the Ph.D. degree in Mathematics for Engineering Sciences from Politecnico di Torino, in 2011. In 2010, she was a visiting member at the Massachusetts Institute of Technology, Cambridge (LIDS) and, she held Postdoctoral positions with Politecnico di Torino (2011-2016). She has been serving as AE of the IEEE Transactions on Signal Processing since 2019 and AE of the IEEE Control Systems Letters since 2020. Her research interests include signal processing, optimization, and control of network systems.
	\end{IEEEbiography}
	
	\vspace{-1cm}
	
	\begin{IEEEbiography}[{\includegraphics[width=1in,height=1.25in,clip,keepaspectratio]{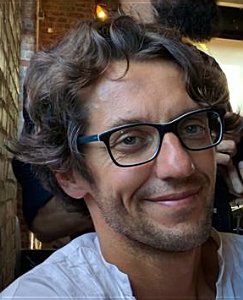}}]{Michele Garetto} (M'04) received the Dr.Ing. degree in Telecommunication Engineering and the Ph.D. degree in Electronic and
		Telecommunication Engineering, both from Politecnico di Torino,
		Italy, in 2000 and 2004, respectively. In 2002, he was a visiting
		scholar with the Networks Group of the University of Massachusetts,
		Amherst, and in 2004 he held a postdoctoral position at the ECE
		department of Rice University, Houston. He is currently associate
		professor at the Computer Science Department of University of Torino, Italy. 
	\end{IEEEbiography}
	\vspace{-1cm}
	
	\begin{IEEEbiography}[ {\includegraphics[width=1.1in,height=1.35in,clip,angle=-90,keepaspectratio]{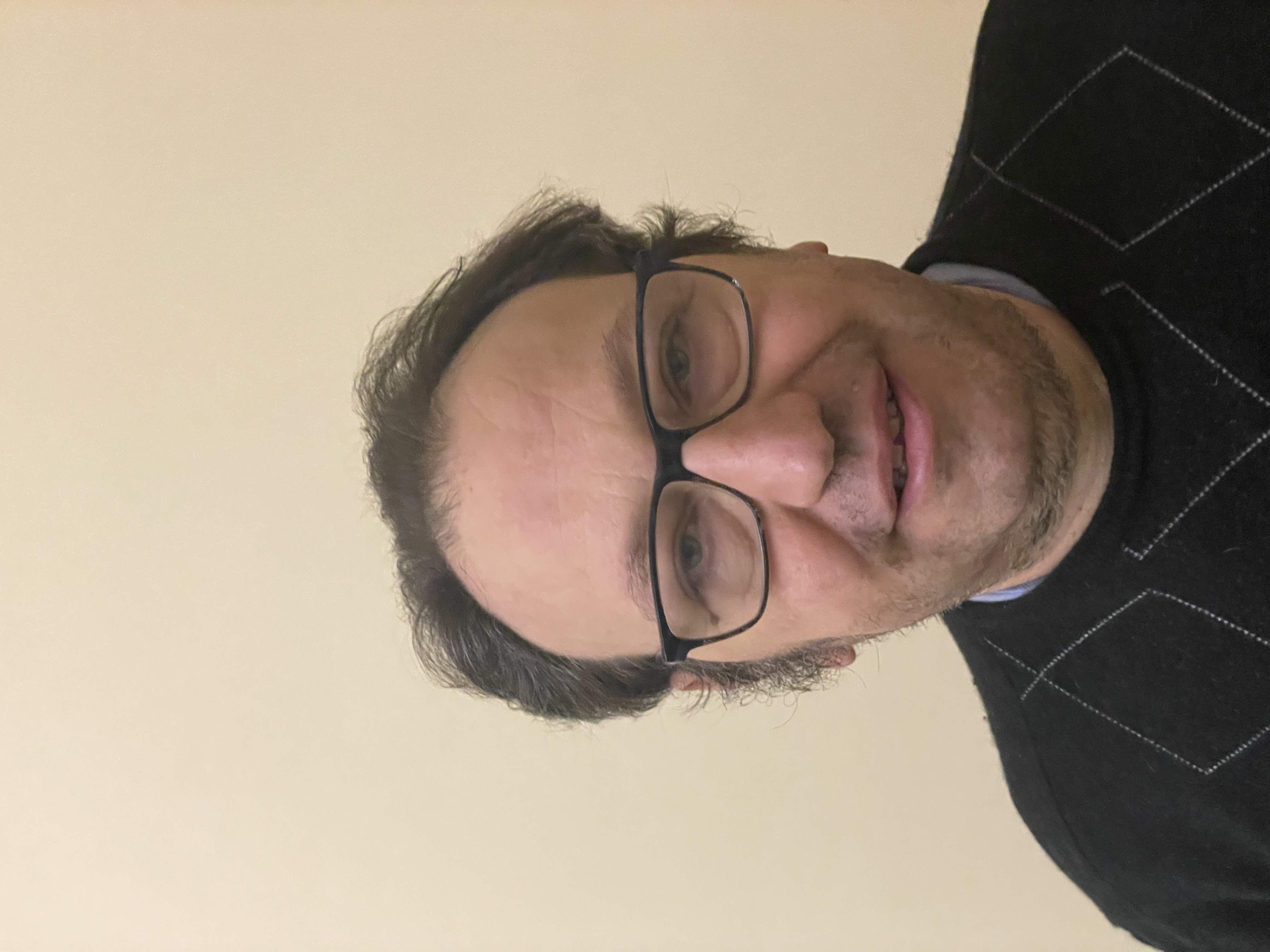} } ]{Emilio Leonardi} is a professor with the Department of Electronics and Telecommunications, Politecnico di Torino.  
		He visited:  UCLA, CS dept. in 1995,   Lucent Bell-labs (Holmdel) in 1999,  Stanford EE dept.  in 2001,	Sprint Labs  (Burlingame) in 2003, NEC Labs (Heidelberg)  in 2012, INRIA (Sophia Antipolis)  in 2016.
		His research interests include performance evaluation of computer networks and distributed systems, dynamics over networks, and human centric computation.
	\end{IEEEbiography}

\end{document}